\documentclass{amsart}

\usepackage{amsmath, amssymb, amsthm, mathtools, amsfonts, bbm}
\usepackage[utf8]{inputenc}
\usepackage[all]{xy}
\usepackage{graphicx, epsfig, multirow}
\usepackage{enumerate, xcolor, color, enumitem}
\definecolor{darkviolet}{rgb}{0.58,0,0.83}
\usepackage[colorlinks,allcolors=blue]{hyperref} 
\usepackage{cleveref}

\usepackage{appendix,pdfsync}
\usepackage{float}
\usepackage{setspace}
\usepackage{xpatch}
\usepackage{subcaption}
\usepackage{pst-node}
\usepackage{tikz-cd}
\usetikzlibrary{positioning}
\usetikzlibrary{shapes,arrows}
\usepackage{tikz}   

\newtheorem{theorem}{Theorem}[section]
\newtheorem{lemma}[theorem]{Lemma}

\theoremstyle{definition}
\newtheorem{definition}[theorem]{Definition}

\newtheorem{claim}[theorem]{Claim}
\newtheorem{conj}[theorem]{Conjecture}
\newtheorem{proposition}[theorem]{Proposition}

\newtheorem{ques}[theorem]{Question}

\theoremstyle{remark}
\newtheorem{remark}[theorem]{Remark}

\numberwithin{equation}{section}

\newcommand{\lk}{\mbox{\upshape lk}\,}
\newcommand{\Sp}{\mathbb{S}}
\newcommand{\D}{\Delta}
\newcommand{\f}{\lfloor\frac{r}{2}\rfloor}
\newcommand{\cil}{\lceil\frac{r}{2}\rceil}
\newcommand{\M}{\mathcal{M}}
\newcommand{\G}{\mathcal{G}}
\newcommand{\cP}{\mathcal{P}}
\newcommand{\T}{\mathcal{T}}

\newcommand{\vr}[2]{\mathcal{VR}(#1;#2)}

\def\A{{\mathcal{A}}}

\def\R{{\mathbb{R}}}
\def\Z{{\mathbb{Z}}}

\def\n{{\mathrm{N}}}

\definecolor{blue1}{rgb}{0.10,0.60,0.8}
\definecolor{ss}{rgb}{0.16,0.01,0.40}
\definecolor{s}{rgb}{0.00,0.00,0.00}
\definecolor{p}{rgb}{0.00,0.00,0.79}
\definecolor{r}{rgb}{0.27,0.00,0.53}
\definecolor{t}{rgb}{0.22,0.00,0.44}
\definecolor{a}{rgb}{0.10,0.70,0.80}
\definecolor{an}{rgb}{0.60,0.00,1.00}
\definecolor{darkviolet}{rgb}{0.9, 0.0, 0.9}
\definecolor{yellow}{rgb}{0.7, 0.0, 0.6}

\begin{document}

\title{On the Vietoris-Rips Complexes of Integer Lattices}
\author{Raju Kumar Gupta}
\address{Institute of Mathematics, Slovak Academy of Sciences, Bratislava, Slovakia}
%\curraddr{Department of Mathematics and Statistics,
%Case Western Reserve University, Cleveland, Ohio 43403}
\email{rajukrg3217@gmail.com}
\thanks{The first author is supported by Ministry of Education of Slovak republic grant VEGA 2/0056/25.}

\author{Sourav Sarkar}
\address{Department of Mathematics, Indian Institute of Technology (IIT) Madras, India}
\email{sarkarsourav610@gmail.com}
\thanks{The second author is supported by the CRG Grant (CRG/2023/000239), India.}

\author{Samir Shukla}
\address{School of Mathematical and Statistical Sciences, Indian Institute of Technology (IIT) Mandi, India}
\email{samir@iitmandi.ac.in}
\thanks{ The third author is supported by the seed grant project IITM/SG/SMS/95 by IIT Mandi, India.}
%    General info
\subjclass[2020]{Primary 55P10, 55N31, 20F65,  51M05; Secondary 57M07, 52C07}

%\date{January 1, 2001 and, in revised form, June 22, 2001.}


\keywords{Vietoris-Rips complex, contractible, discrete-Morse theory, integer lattice, word metric}

\begin{abstract}
For a metric space $X$ and $r \geq 0$, the Vietoris-Rips complex $\vr{X}{r}$ is a simplicial complex whose simplices are finite subsets of $X$ with diameter at most $r$.  Vietoris-Rips complexes has applications in various places, including data analysis, geometric group theory, sensor networks, etc. Consider the integer lattice $\mathbb{Z}^n$ as a metric space equipped with the $d_1$-metric (the Manhattan metric or standard word metric in the Cayley graph).  Ziga Virk [Contractibility of the Rips complexes of integer lattices via local domination, Trans. Amer. Math. Soc. 378, no. 3, 1755–1770, 2025] proved that if  either $r \geq n^2(2n-1)$, or $n \in \{1, 2,3\}$ and $r \geq n$,   then the complex $\vr{\Z^n}{r}$ is contractible, and posed a question if $\vr{\Z^n}{r}$ is contractible for all $r \geq n$.    Recently, Matthew Zaremsky [Contractible  Vietoris–Rips Complexes  of  $\mathbb{Z}^n$, Proc. Amer. Math. Soc, 2025] improved Ziga's result and proved that  $\vr{\Z^n}{r}$ is contractible if $r \geq n^2+ n-1$.  Further, he  conjectured that $\vr{\Z^n}{r}$ is contractible for all $r \geq n$. We prove Zaremsky's conjecture for  $n \leq 5$, i.e, we prove that  $\vr{\Z^n}{r}$ is contractible if $n \leq 5$ and $r \geq n$. Further, we prove that $\vr{\Z^6}{r}$ is contractible for $r \geq 10$. 

We determine the homotopy type of $\vr{\Z^n}{2}$, and show that these complexes are homotopy equivalent to a wedge of countably infinite copies of $\Sp^3$. We also show that $\vr{\Z^n}{r}$ is simply connected for $r \geq 2$.

\end{abstract}

\maketitle

\section{Introduction}

For a metric space $(X, d)$ and $r\geq 0$, the Vietoris-Rips complex $\vr{X}{r}$ is a simplicial complex on $X$, where a finite set $\sigma \subseteq X$ is a simplex if and only if diameter of $\sigma$ is at most $r$, {\it i.e.,} $d(x, y) \leq r$ for all $x, y \in \sigma$.  The Vietoris-Rips complex was first discovered by Vietoris \cite{Vietoris27} to define a homology theory for metric spaces and independently re-descovered by E. Rips for studying hyperbolic groups, where it has been popularised as Rips-complex \cite{Ghys, Gromov87}.  One of the main motivations behind introducing these complexes was to create a finite simplicial model for metric spaces.  

Vietoris-Rips complexes  have been used  in  topological data analysis to probe the shape of a point cloud data using persistence homology 
\cite{Bauer2021, Carlsson2009,  ZomorodianCarlsson2005,Zomorodian2010}.  These complexes  have  been used heavily in computational topology, as a simplicial
model for point-cloud data \cite{Carlsson06, CarlssonZomorodian2005,CarlssonIshkhanovDeSilvaZomorodian2008,  DesilvaCarlsson2004} and as simplicial completions of
communication links in sensor networks \cite{DesilvaGhrist2006, DesilvaGhrist2007, Muhammad2007}.

In this article, we consider the Vietoris-Rips complexes of the integer lattice $\Z^n$ with the Manhattan metric $d$ (standard word metric in the Cayley graph), {\it i.e.,} for any $x = (x_1, \ldots, x_n)$ and $y = (y_1, \ldots, y_n) \in \Z^n$, $d(x, y) = \sum_{i = 1}^n |x_i - y_i|$.  One of the main motivations for our results in this article has a connection to the world of geometric group theory and topological
finiteness properties of groups. Recall that a group is of type $F_n$ if it admits a geometric (that is, proper and cocompact) action on an $(n-1)$-connected CW complex.  A group is of type $F_{\ast}$ if it admits a geometric action on a
contractible CW complex. Zaremsky \cite{Zaremsky2022} pointed out that an adequate understanding of the Vietoris–Rips complexes of a group $G$ with the word metric can reveal topological finiteness properties of $G$. Using the Brown’s Criterion\footnote{If a group acts properly on an $(n-1)$-connected CW complex $X$ with an invariant
	cocompact filtration $(X_t)_{t \in \R}$, then the group is of type $F_n$ if and only if this filtration is
	essentially $(n-1)$-connected, meaning for all $t$ there exists $s \geq t$ such that the inclusion
	$X_t \to X_s$ induces the trivial map in $\pi_k$ for $k\leq n-1$.} \cite{Brown1985}, he \cite[Lemma 3.6]{Zaremsky2022} proved that,  {\it $G$ (with the word metric corresponding to some finite generating
	set) is of type $F_n$ if and only if the filtration $(\vr{G}{t})_{t \in \R}$ is essentially $(n-1)$-connected. If some $\vr{G}{t}$ is contractible, then $G$ is of type $F_{\ast}$.}
Rips proved that if a hyperbolic group is equipped with a word metric, then for sufficiently large scale $r$, its Vietoris-Rips complex is contractible \cite[Proposition III.$\Gamma$.3.23]{Brid1999}.
Beyond the hyperbolic case, the contractibility of its Vietoris-Rips complexes is quite hard to prove.  It is clear from the definition that the Vietoris–Rips complex of bounded metric spaces is contractible for sufficiently large scale $r$. The contractibility of Vietoris–Rips complexes at large scales is less understood for unbounded metric spaces, even for simple examples such as integer lattices. For the  group $\Z^n$, the question of contractibility of $\vr{\Z^n}{r}$ was first posed by Zaremsky  in 2018 \cite{Zaremsky2022}: {\it  Are the Rips complexes of the
	free finitely generated Abelian groups (integer lattices in word metric)
	contractible for large scales?} This question remains open and had been attracting the attention of researchers for more than seven years. In \cite{Virk2025}, Virk introduced the local domination technique, and using it he proved that $\vr{\Z^n}{r}$ is contractible for $r \geq n^2(2n-1)$. In \cite{Zaremsky2025}, the author applied Bestvina--Brady discrete Morse theory and improved the bound to $r \geq n(n+1)-1$. Using local domination McCarty \cite[Theorem 3.1]{McCarty2025}  showed that $\vr{\Z^n}{r}$ is contractible for $ r \geq n(n+1)$.   In \cite{Zaremsky2025}, Zaremsky made the following conjecture (also posed as a question in \cite[Section 6]{Virk2025}). 

\begin{conj}[Zaremsky]\label{conj:main}
	For any $ r\geq n$, the Vietoris-Rips complex $\vr{\mathbb{Z}^n}{r}$ of $\Z^n$ with the Manhattan metric (standard word metric)  is contractible.  
\end{conj}

The Conjecture \ref{conj:main} is known to be true for $n \leq 3$ (see \cite{Virk2025, Wang2024}). One of the main results of this article is that the Conjecture \ref{conj:main} is true for $n \leq 5$. We also prove that $\vr{\Z^6}{r}$ is contractible for $r \geq 10$.

\begin{theorem} {\rm (Theorem \ref{maintheorem})}\label{thm:n<=5}
	For $n \leq 5$ and $r \geq n$, $\vr{\Z^n}{r}$ is contractible. 
\end{theorem}

\begin{theorem}{\rm (Theorem \ref{thm:n=6(Inside)})}\label{thm:n=6}
	$\vr{\Z^6}{r}$ is contractible for $r \geq 10$.
\end{theorem}

Virk proved the Conjecture \ref{conj:main} for $n=1, 2$ using domination (see Definition \ref{def:domination}) and remarked \cite[Remark 3.3]{Virk2025}  that the domination cannot be used for $n \geq 3$. For $n =3$, he used the local domination technique to prove the Conjecture \ref{conj:main}.  
We observed that for $n \geq 3$, the domination cannot be used directly, but it can be used recursively in the links of the vertices chosen carefully in a certain order. To prove Theorems \ref{thm:n<=5} and \ref{thm:n=6}, we prove that the $\vr{\{0, 1, \ldots M\}^n}{r}$ is contractible for all $M$. We establish a series of  lemmas (Lemmas \ref{onevetex} to \ref{sumofthree} and  \ref{fourvertex} to \ref{sumoffour2}), which are true for any  $r \geq n \geq 2$ (except a few obvious lower bound conditions for $n$ and $r$, and the condition $n \geq 5$ and $r \geq 10$ in Lemma \ref{sumoffour2}). The idea is to reduce (without changing the homotopy type) the links of vertices (recursively by choosing an ordering of vertices) to a smaller induced subcomplex on the vertices, such that the sum of the absolute value of its coordinates has a fixed upper bound.  In Lemma \ref{sumoffour2}, we have deduced the condition that the sum of the absolute values of any of the $4$ coordinates is $\leq r-1$. We also believe that our proof strategy should work to fully settle Conjecture \ref{conj:main}. In particular, if  Lemma \ref{sumoffour2} can be generalized to $n-2$ coordinates (see Conjecture \ref{conj:future}), then we can prove the Conjecture \ref{conj:main} (see Section \ref{Conclusion} for the proof assuming that Conjecture \ref{conj:future} is true.).

%We present some novel observations on $\vr{\Z^n}{r}$ that may help address the above conjecture.  \samir{discuss here about our lemmas and results for general r. Also discuss something about methodology}

In \cite{Virk2025}, the author also noted that the bound given in Conjecture \ref{conj:main} is optimal in the sense that $\vr{\Z^n}{r}$ is not contractible if $r < n$. In fact, it is shown in \cite{AdamsVirk2024} that $\vr{\{0, 1\}^n}{r}$ is not contractible for $r < n$, and in \cite{Virk2022Persistance} showed that the inclusion
$\{0, 1\}^n \to \Z^n$ induces an injection on the homology of Vietoris-Rips complexes
at each scale $r$.  The complex $\vr{\{0, 1\}^n}{r}$ has been paid a lot of attention in recent years \cite{Adamaszek2021,AdamsVirk2024,Martin2023,Ziqin2024,Shukla2023}. In \cite{Adamaszek2021}, authors proved that $\vr{\{0, 1\}^n}{2}$  is homotopy equivalent to wedge of 3-dimensional spheres $\Sp^3$'s. Since, there is an injection  $\tilde{H}_{\ast}(\vr{\{0, 1\}^n}{2}) \to \tilde{H}_{\ast}(\vr{\Z^n}{2})$, it is a natural question  to ask whether $ \vr{\Z^n}{2}$ is homotopy equivalent to wedge of $\Sp^3$'s. Motivated by this question, using (Forman's) discrete Morse theory, we prove the following.

\begin{theorem}{\rm (Theorem \ref{thm:(Z^n;2)})}\label{thm:r=2_Intro}
	For $n \geq 3$, $\vr{\Z^n}{2}$ is homotopy equivalent to wedge sum of countably infinite copies of $\Sp^3$'s.
\end{theorem}

We also prove that $\vr{\Z^n}{r}$ is simply connected for $r \geq 2$. 

\begin{theorem}{(Theorem \ref{theorem:simplyconnected_inside})}\label{thm:simplyconnected_intro}
	$\vr{\Z^n}{r}$ is simply connected for $r \geq 2$. 
	
\end{theorem}

The article is organized as follows. In Section \ref{Preliminaries}, we recall the necessary definitions and basic results used throughout the paper. Section \ref{sec:vr1} is devoted to the study of $\vr{\mathbb{Z}^n}{r}$ for $r \geq n$, where we present several observations, establish key lemmas, and prove our main result on the contractibility Theorems \ref{thm:n<=5} and \ref{thm:n=6}. In Section \ref{sec:vr2}, we first characterize the maximal simplices of $\vr{\mathbb{Z}^n}{2}$ for $n \geq 3$, and then determine its homotopy type using discrete Morse theory. Finally, in Section \ref{Conclusion}, we summarize our contributions and discuss directions for future research. In particular, we prove Conjecture \ref{conj:main} under the assumption of Conjecture \ref{conj:future}. We propose questions and a conjecture that naturally arise from this work.

\section{Preliminaries}\label{Preliminaries}
A {\em graph} $G$ is an ordered pair $(V(G), E(G))$, where $V(G)$ is a finite set called the {\em vertex set}, and $E(G) \subseteq \binom{V(G)}{2}$ is a set of $2$-element subsets of $V(G)$, called the {\em edge set} of $G$. A {\em subgraph} of $G$ is a graph $H$ such that $V(H) \subseteq V(G)$ and $E(H) \subseteq E(G)$. The {\em induced subgraph} of $G$ on a subset $W \subseteq V(G)$, denoted by $G[W]$, is the graph with vertex set $W$ and edge set $\{\{u,v\} \in E(G) : u,v \in W\}$. For more details on graph-related terminologies, we refer to \cite{Douglas}.

An {\em (abstract) simplicial complex} $\Delta$ on a vertex set $V$ is a collection of finite subsets of $V$ such that if $\sigma \in \Delta$ and $\tau \subseteq \sigma$, then $\tau \in \Delta$. If $\sigma \in \Delta$ and $\mathrm{Card}(\sigma) = k+1$, then $\sigma$ is called a simplex of dimension $k$, or a $k$-simplex, here $\mathrm{Card}(S)$ denote the cardinality of the set $S$.  We assume that every simplicial complex contains the empty set as the simplex of dimension $-1$. The $0$-dimensional simplices are referred to as the {\em vertices} of $\Delta$, and the set of all vertices in $\Delta$ is denoted by $V(\Delta)$. 

A {\em subcomplex} of a simplicial complex is a subcollection of simplices that also forms a simplicial complex. For a subset  $S\subseteq V(\D)$, $\mathrm{Ind}_{\D}(S)=\{\sigma  \cap S : \sigma \in \D\}$ is called the {\em induced subcomplex} of $\D$ induced on the vertex set $S$.

If $\Delta$ is a simplicial complex and $\sigma \in \Delta$, then the {\em link}  of $\sigma$ is a subcomplex, defined as follows
$\lk(\sigma,\D) := \{\tau \in \Delta : \sigma \cup \tau \in \Delta \text{ and } \sigma \cap \tau = \emptyset\}$. The {\em deletion} of $\sigma$ is defined as the subcomplex  $\{\tau\in\Delta : \sigma\nsubseteq\tau\}$ and is denoted by $\text{del}(\sigma,\D)$.

In this article,  we consider any simplicial complex as a topological space, namely its geometric realization. For the definition of
geometric realization and details about terminologies related to simplicial complexes, we refer the reader to \cite{Kozlovbook}. For details on homotopy theory and related topological concepts,  we refer to \cite{Hatcher02}.

\begin{proposition}[{\rm Lemma 2.5, \cite{HenryShuklaAnurag2025}}]\label{Link and Deletion} 
    Let $v$ be a vertex of a simplicial complex $K$. If the inclusion $\lk(v,K) \xhookrightarrow {} K \setminus v$ is null-homotopic, then we have
    $$K\simeq (K \setminus v) \vee \Sigma \lk(v,K).$$

    Here, $K \setminus v = \text{del}(v,K)$  is the deletion of $v$, and $\Sigma X$ denote the suspension of the space $X$.
\end{proposition}

\begin{definition}(Domination)\label{def:domination}
	\rm{   Let $K$ be a simplicial complex and let $a, b \in V(K)$ with $a \neq b$. We say that the vertex $a$ is {\it dominated} by $b$, if for every simplex $\sigma \in K$ containing $a$, the set $\sigma \cup \{b\}$ is also a simplex in $K$. }
    \end{definition}

For a simplicial complex $K$  and a vertex $v$, we define the {\it open neighborhood} of $v$ as $\n(v, K)=\{u\in V(K): \{u,v\}\in K\}$, and the {\it closed neighborhood} as $\n[v, K]=\n(v, K)\cup\{v\}$.
    
A simplicial complex $K$ is called {\it flag} if for any $\sigma \subseteq V(K)$, $\sigma \in K$ if and only if $\{u, v\} \in K$ for any $u, v \in \sigma$.  	Observe that if $K$ is a flag simplicial complex and $\n[a, K]\subseteq \n[b, K]$, then $a$ is dominated by $b$. From \cite[Proposition 3.2]{Prisner1992}, we have following.

\begin{proposition}\label{flag}
	Let $K$ be a simplicial complex and $a,b\in V(K)$ such that $a \neq b$. Let $\n[a, K]\subseteq \n[b, K]$. If $K$ is a flag complex, then $K \simeq K \setminus a$.
\end{proposition}

Note that Vietoris-Rips complexes are flag simplicial complexes. Hence Proposition \ref{flag} is true for any  Vietoris-Rips complex.
%\begin{proof} 
	%If $\n[a, K]\subseteq \n[b, K]$ in a flag simplicial complex $K$, then the vertex $a$ is dominated by the vertex $b$. Now, the result follows from [Theorem 3.4, \cite{Virk2025}].
%\end{proof}

Let $G$ be a finitely generated group with a finite symmetric generating set $S$ (i.e., $S = S^{-1} := \{x^{-1} : x \in S\}$) such that the identity element $e \notin S$. The {\em Cayley graph} $\mathrm{Cay}(G, S)$ is the graph whose vertex set is $G$, with an edge between $g, h \in G$ if and only if $g^{-1}h \in S$.

The {\em word metric}   $d: G \times G \to \mathbb{N} \cup \{0\}$  on a group $G$ associated with the generating set $S$ is defined by
$$
d(g, h) := \min \left\{ n \in \mathbb{N} \cup \{0\} \;\middle|\; g^{-1}h = s_1 s_2 \cdots s_n \text{ for some } s_i \in S \right\}.
$$
This metric coincides with the graph distance (minimum path length distance) between $g$ and $h$ in the Cayley graph $\mathrm{Cay}(G, S)$.

Let $\Z^n$ denote the $n$-dimensional integer lattice (free abelian group  on $n$ generators). For $x\in \mathbb{Z}^n$, we denote by $x_i$ the $i^{th}$ coordinate of $x$. We consider $\Z^n$ as a metric space equipped with Manhattan distance $d$, {\it i.e.,}   for any $x = (x_1, \ldots, x_n)$ and $y = (y_1, \ldots, y_n) \in \Z^n$, $d(x, y) = \sum_{i = 1}^n |x_i - y_i|$. We also consider $\Z^n$ as a graph, where any two vertices $x$ and $y$ are connected by an edge if and only if $d(x, y) = 1$. Note that 
the metric $d$ in $\Z^n$ is the word metric associated with the standard generators $\{\pm e_1, \ldots, \pm e_n\}$, where $e_i$ is the element with $i^{th}$ coordinate $1$, and all other co-ordiantes are $0$.

%For a positive integer  $n$, we let $[n]= \{1, \ldots, n\}$. 

% \begin{proposition}\label{link}
	%     Let \( K \) be a flag simplicial complex and \( a \in V(K) \). If  $\lk_K(a)$ is contractible in $\mathrm{Ind}_K(V(K) \setminus \{a\})$,  
	% then $K \simeq \mathrm{Ind}_K(V(K) \setminus \{a\}).$
	% \end{proposition}

\section{The complex $\vr{\Z^n}{r}, r \geq n$}\label{sec:vr1}

In this section, we prove Theorems \ref{thm:n<=5} and \ref{thm:n=6}. For a positive integer  $n$, let $[n]= \{1, \ldots, n\}$.  For $m > 0$,  let $\G_m^n$ denote the induced  subgraph $\Z^n[\{0, \ldots, m\}^n]$. Let $\Delta_m^{n, r} = \vr{\G_m^n}{r}$. By Whitehead's theorem \cite{Whitehead}, if all homotopy groups of a CW complex $X$ are trivial, then the unique map $X \to \star$ induces isomorphisms on all homotopy groups and is therefore a homotopy equivalence, here $\star$ denote the one-point space. Hence, $X$ is contractible. Observe that  any homotopy class of $\vr{\mathbb{Z}^n}{r}$ is contained in the $\vr{\G_m^n}{r}$ for some $m\in \mathbb{N}$.
Thus, to prove Theorems \ref{thm:n<=5} and \ref{thm:n=6}, it suffices to show that for each positive integer $m$, the complex $\Delta_m^{n, r} = \vr{\G_m^n}{r}$ is contractible.

We first fix some notations. Let  $\,\mathrm{sgn}:\mathbb{Z}\to\{-1,1\}$ denote the sign function, {\it i.e.,}  $\mathrm{sgn}(x)=1$ if $x>0$ and $\mathrm{sgn}(x)=-1$ if $x<0$. Throughout this article, whenever we use $\mathrm{sgn}(x)$ for some $x\in\mathbb{Z}$, the value of $x$ is always nonzero, i.e., either a positive or a negative integer.

\begin{definition}
	{\rm For $x\in\mathbb{Z}^n$ and $S\subseteq [n]$, we define $\lambda^{[x;S]}$ in $V(\Z^n)$ as follows: 
		$$
		\lambda^{[x;S]}_j = \begin{cases}
			x_j -1, \ \ \text{if} \ \  j\in S\ \text{and }  x_j > 0,\\
			x_j +1, \ \ \text{if} \  \ j\in S \ \text{and }   x_j <  0,\\
			x_j, \ \ \ \ \ \ \   \text{elsewhere} .
		\end{cases}
		$$ Moreover, if $S=\{i_1,\dots,i_r\}\subseteq[n]$, then as a notation, we write $\lambda^{[x;S]}$ as $\lambda^{[x;i_1i_2\dots i_r]}$.}

\end{definition}

Throughout this section, we fix a positive integer $n \geq 2$. Let $\prec$ denote the  anti-lexicographic order on $V(\Z^n)$, {\it i.e.,}   for any two distinct vertices  $x = (x_1, \dots, x_n)$, $y = (y_1, \dots, y_n) \in V(\Z^n)$, we have $x \prec y$ if and only if the largest index $i$ at which $x_i \neq y_i$,  $x_i < y_i$.

Fix a positive integer $m$.  For $1 \leq \alpha \leq \mathrm{Card}(V(\G_m^n))$,  let $\mathcal{H}_m^{n, \alpha}$ denote the subset of $V(\G_m^n)$ after removing the first $\alpha$ elements (with respect to order $\prec$) from 
$V(\G_m^n)$. Let $\delta= (\delta_1 , \ldots, \delta_n)$ be the least element of $\mathcal{H}_m^{n, \alpha}$.  Let $Y_m^{n, \alpha}$ be the induced subgraph of $\Z^n$ on the vertex set 
$ \{x-\delta: x \in \mathcal{H}_m^{n, \alpha}\}$.  Clearly, $(0, \ldots, 0) \in V(Y_m^{n, \alpha})$ and it is the smallest element of $V(Y_m^{n, \alpha})$.

Let $\bf{0}$ denote the vertex  $(0, \ldots, 0) \in V(\Z^n)$. Let $\Gamma_n^{\alpha,r}=\lk({\bf{0}},\vr{Y_m^{n, \alpha}}{r})$. Since the link of a vertex in a flag complex is flag, it follows that, $\Gamma_n^{\alpha,r}$ is flag.

For any  $r \geq 1$, we have 
$$V(\Gamma_{n}^{\alpha,r})\subseteq \underbrace{([-r,r]\cap \Z)\times\cdots\times([-r,r]\cap\Z)}_{(n-1)\text{-times}}\times([0,r]\cap\Z).$$

Observe that, for any  $x \in V(\Z^n)$,  if ${\bf{0}} \prec x$, $x+ \delta \in \mathcal{H}_m^{n, \alpha}$  and $d({\bf{0}}, x) \leq r$, then $x \in  V(\Gamma_n^{\alpha, r})$. Further, if $z \in V(\Gamma_n^{\alpha, r})$, ${\bf{0}} \prec x \prec z$ and $d({\bf{0}}, x) \leq r$, then  $x \in  V(\Gamma_n^{\alpha, r})$.

To show that $\Delta_m^{n,r}$ is contractible, we successively remove vertices from $\G_m^n$, using Proposition \ref{Link and Deletion}, without changing the homotopy type, until only one vertex remains in the complex. For this, it is sufficient to show that $\Gamma_n^{\alpha,r}$ is contractible for all $1 \le \alpha \le \mathrm{Card}(V(\G_m^n)) - 1$. To prove the contractibility of $\Gamma_n^{\alpha,r}$, we first establish a series of lemmas.

\begin{lemma}\label{onevetex}
	Let $r\geq 2$ and let  $\D$ be  a subcomplex of $\Gamma_n^{\alpha, r}$. Let $x,y,z\in V(\D)$ be  such that $z\in V(\lk(x,\D))$ and  $d(x,y)=1$. Let $i \in [n]$ be such that   $|y_i|=|x_i|-1$ and $y_j=x_j$  for all $j\neq i$.

	\begin{itemize}
		\item[(i)] If $|z_i|<|x_i|$, then $d(z,y) \leq r-1$. 
		\item[(ii)] If $|z_i|\geq |x_i|$ and, $sgn(z_ix_i)=-1$, then $d(z,y)\leq r-1$.
	\end{itemize}
\end{lemma}

\begin{proof}
	Given $z\in V(\lk(x,\D))$, it follows that $d(x,z)\leq r$. Observe that each of the hypotheses in $(i)$ and $(ii)$ leads to $|z_i - y_i| = |z_i - x_i| - 1$. Since $|z_j - y_j| = |z_j - x_j|$ for every $j \neq i$, we have  $
	d(z, y) = |z_i - y_i| + \sum_{j \neq i} |z_j - y_j| = (|z_i - x_i| - 1) + \sum_{j \neq i} |z_j - x_j| = d(z, x) - 1 \leq r - 1$. This completes the proof.
\end{proof}

\begin{lemma}\label{r/2}
	Let $r\geq 2$. Then $\Gamma_{n}^{\alpha,r}$ is homotopy equivalent to the induced subcomplex of $\Gamma_{n}^{\alpha,r}$ on the vertex set 
	$\{x\in V(\Gamma_{n}^{\alpha,r}) : |x_i|\leq \f$ for all $1\leq i\leq n\}$.  
	
\end{lemma} 

\begin{proof}
	Without changing the homotopy type of $\Gamma_{n}^{\alpha,r}$, we remove all the vertices $x$ such that  $|x_i| > \f$ for some $i$. Let $x \in V(\Gamma_n^{\alpha, r})$ be such that $|x_i| = r$ for some $i \in [n]$.  Since  $d({\bf{0}}, z)  = \sum_{i=1}^{n} |z_i| \leq r$ for all $z \in V(\Gamma_n^{\alpha, r})$ and $x \succ {\bf{0}}$, we get that  $x_i = r$ and $x_j = 0$ for all $j \neq i$. Let us consider the vertex $\lambda^{[x;i]}$. Since $r \geq 2$, $\lambda_i^{[x;i]}>0$ and $\lambda_j^{[x;i]}=0$ for all $j\neq i$. This implies that  $\lambda^{[x;i]} \succ {\bf{0}}$ and $\lambda^{[x;i]} \in V(\Gamma_n^{\alpha, r})$. Clearly, $d(\lambda^{[x;i]}, x) = 1$.
	
	We show that $\n[x, \Gamma_{n}^{\alpha,r}] \subseteq \n[\lambda^{[x;i]}, \Gamma_{n}^{\alpha,r}]$. Clearly, $x \in \n[\lambda^{[x;i]}, \Gamma_{n}^{\alpha,r}]$. Let $z \in \n[x, \Gamma_n^{\alpha, r}]$ with $z \neq x$.  Since $z \succ {\bf{0}}$, $d(z, x) \leq r$ and $x_i = r$, we conclude that $|z_i | < r = x_i$. Therefore, from Lemma \ref{onevetex}, $d(z, \lambda^{[x;i]}) \leq r-1$. Thus, $z \in \n[\lambda^{[x;i]}, \Gamma_{n}^{\alpha,r}]$, and hence $\n[x, \Gamma_{n}^{\alpha,r}] \subseteq \n[\lambda^{[x;i]}, \Gamma_{n}^{\alpha,r}]$. From Proposition \ref{flag}, we conclude that $\Gamma_n^{\alpha, r} \simeq \mathrm{Ind}_{\Gamma_n^{\alpha, r}}\left(V(\Gamma_{n}^{\alpha,r}) \setminus \{x\}\right).$
	Let $A= \{y \in  V(\Gamma_{n}^{\alpha,r}) : y_i = r \ \text{for some } \ i \in [n] \ \text{and} \  y_j = 0 \ \text{for all } \ i \neq j \}$.
	Repeating the above argument for each vertex $y \in A$, we get  that 
	$$
	\Gamma_n^{\alpha, r} \simeq \mathrm{Ind}_{\Gamma_n^{\alpha, r}}\left(V(\Gamma_n^{\alpha, r}) \setminus A \right).
	$$
   If $r=2$, then for evrely $z\in \mathrm{Ind}_{\Gamma_n^{\alpha, r}}\left(V(\Gamma_n^{\alpha, r}) \setminus A \right)$, $|z_j|\leq 1$ for all $j\in [n]$ and thus we are done. Now, we assume that $r\geq 3$.	
   
   Inductively, assume that, without changing the homotopy type of  $\Gamma_{n}^{\alpha,r}$, we have removed all the vertices $y$ from $\Gamma_n^{\alpha, r}$ such that $r \geq |y_i| \geq \f + t + 1$ for some $t \geq 1$ and $i \in [n]$.  Let  $\Gamma'$ denote the resulting subcomplex of $\Gamma_{n}^{\alpha,r}$ after removal of all such vertices. Then   $\Gamma'$ is the induced  subcomplex on the vertex set $\{x\in V(\Gamma_{n}^{\alpha,r}): |x_i| \leq \f + t \ \text{for all} \ i\in [n]\}$ and $\Gamma'\simeq \Gamma_{n}^{\alpha,r}$. 
	
	Let $u \in V(\Gamma')$ be such that $|u_s| = \f + t$ for some $s\in [n]$. Since $d({\bf{0}}, u)  = \sum_{j \in [n]} |u_j| \leq r$, we see that $\sum_{j \neq s} |u_j| \leq \cil - t$. 
	Consider the element $\lambda^{[u;s]}$. Clearly, $d({\bf{0}}, \lambda^{[u;s]}) \leq d({\bf{0}}, u)-1 \leq r-1$ and $d(u, \lambda^{[u;s]}) = 1$. Since  $r \geq 2$, and $0\prec \lambda^{[u;s]}\prec u$, we get  $\lambda^{[u;s]} \succ {\bf{0}}$ and  $\lambda^{[u;s]} \in V(\Gamma_n^{\alpha, r})$. Further, since $|\lambda^{[u;s]}_j| < \f +t $ for all $j \in [n]$, we have $\lambda^{[u;s]} \in V(\Gamma')$.

	We now prove that $\n[u, \Gamma'] \subseteq \n[\lambda^{[x;s]}, \Gamma']$. Clearly, $u \in \n[\lambda^{[u;s]}, \Gamma']$. Let $v \in \n[u, \Gamma']$ with $v \neq u$. If $|v_s| < |u_s|$ or, $|v_s| \geq |u_s|$ and $sgn(v_s u_s) = -1$, then from Lemma \ref{onevetex}, $d(v, \lambda^{[u;s]}) \leq r$. Thus, in this case, $v \in \n[\lambda^{[u;s]}, \Gamma']$.
	
	Suppose $|v_s| \geq |u_s|$ and $sgn(v_s u_s) = 1$. Since $|v_s| \leq  \f + t$, it follows that $|v_s| = \f + t$ and $v_s = u_s$. Further, since $v\in V(\Gamma')\subseteq V(\Gamma_{n}^{\alpha,r})$, $d({\bf{0}}, v)=\sum_{j\in [n]}|v_j|\leq r$. Thus, we have $\sum_{j \neq s} |v_j| \leq \cil - t$. Therefore,
	\begin{align*}
		d(v,\lambda^{[u;s]}) = & |v_s - \lambda^{[u;s]}_s| + \sum_{j \neq s} |v_j - \lambda^{[u;s]}_j|\\
		= & |v_s - u_s| + 1 + \sum_{j \neq s} |v_j - u_j| \leq 1 + \sum_{j \neq s} |u_j| + \sum_{j \neq s} |v_j| \\
		& \leq 1 + 2\cil - 2t \leq  r + 2 - 2t \leq r, \ \text{as}\ t \geq 1.
	\end{align*}
	This implies that $\n[u, \Gamma'] \subseteq \n[\lambda^{[u;s]}, \Gamma']$. Therefore, from Proposition \ref{flag},
	$\Gamma'\simeq \mathrm{Ind}_{\Gamma'}(V(\Gamma')\setminus \{u\}).$
	Let $B=\{x\in V(\Gamma'):|x_i|=\f+t \ \text{for some}\ i\in [n] \}$. Repeating the above argument for each vertex $y\in B$, we find that $\Gamma'\simeq \mathrm{Ind}_{\Gamma'}(V(\Gamma')\setminus B)$. Consequently, $\Gamma_{n}^{\alpha,r}\simeq \mathrm{Ind}_{\Gamma'}(V(\Gamma')\setminus B)$.

	By induction, we conclude that $\Gamma_{n}^{\alpha,r}$ is homotopy equivalent to the induced subcomplex of $\Gamma_{n}^{\alpha,r}$ on the vertex set 
	$\{x\in V(\Gamma_{n}^{\alpha,r}) : |x_i|\leq \f \ \forall \ 1\leq i\leq n\}$. 
\end{proof}

\begin{lemma}\label{twovertex}
	Let $r\geq 2$ and let  $\D$ be  a subcomplex of $\Gamma_n^{\alpha, r}$. Let $x,y,z\in V(\D)$ be  such that $z\in V(\lk(x,\D))$ and $d(x,y)=2$. Let $|y_i|=|x_i|-1$ and $|y_j|=|x_j|-1$ for some  distinct $i,j \in [n]$,  and $y_k=x_k$ for all $k\neq i,j$.  
	\begin{itemize}
		\item[(a)] If $|z_i|<|x_i|$ or $|z_j|<|x_j|$,  then $d(z,y)\leq r$. 
		\item[(b)] If  $|z_i|\geq |x_i|$ and $sgn(z_ix_i)=-1$, or  $|z_j|\geq |x_j|$ and $sgn(z_jx_j)=-1$,  then $d(z,y)\leq  r$.
		
		% and $z_j$ and $x_j$ are of opposite sign, then $d(z,y)\leq d(z,x)\leq r$. 
	\end{itemize}
\end{lemma}

\begin{proof}
	Suppose $|z_i| < |x_i|$. Then we have $|z_i - y_i| = |z_i - x_i| - 1$. Since $|y_j|=|x_j|-1$, regardless of whether $|z_j| < |x_j|$ or $|z_j| \geq |x_j|$, we have $|z_j - y_j| \leq |z_j - x_j| + 1$. Therefore,
	$d(z, y) = |z_i - y_i| + |z_j - y_j| + \sum_{k \neq i, j} |z_k - y_k| \leq (|z_i - x_i| - 1) + (|z_j - x_j| + 1) + \sum_{k \neq i, j} |z_k - x_k| \leq r.$
	A similar argument applies if $|z_j| < |x_j|$. This proves part $(a)$.

	For part $(b)$, note that if $|z_i| \geq |x_i|$ and $sgn(z_i x_i) = -1$, then $|z_i - y_i| = |z_i - x_i| - 1$. Thus, similar computations as above show that $d(z, y) \leq r$. A similar argument applies if $|z_j| \geq |x_j|$ and $sgn(z_j x_j) = -1$. This proves part $(b)$.
\end{proof}

\begin{lemma}\label{xi+xj}
	Let $r \geq 3$. Then $\Gamma_n^{\alpha, r}$ is homotopy equivalent to the induced subcomplex of $\Gamma_n^{\alpha, r}$ on the vertex set $\{x \in \Gamma_n^{\alpha, r} : |x_i| \leq \f \ \text{and} \ |x_j| + |x_k| \leq \cil \ \ \text{for all } \ i,j,k\in [n], j\neq k\}$.
	
\end{lemma}

\begin{proof}
	From Lemma \ref{r/2},  $\Gamma_n^{\alpha, r}$ is homotopy equivalent to the induced subcomplex, say $\D$, of $\Gamma_n^{\alpha, r}$ on the vertex set $\{y \in \Gamma_n^{\alpha, r} : |y_i| \leq \f \text{ for all } 1 \leq i \leq n\}$.
	
	Consider the dictionary order $\prec'$ on the set $\{(i,j): i,j \in [n], i < j\}$. Then we have $(1,2) \prec' (1,3) \prec' \cdots \prec' (1,n) \prec' (2,3) \prec' \cdots \prec' (n-1,n)$.   Without changing the homotopy type of $\Delta$, we remove all the vertices whose sum of the $i$-th and $j$-th entries exceeds $\cil$ for some $i, j \in [n]$. We remove such vertices in a sequence following the order $\prec'$.
	
	We begin with the pair $(1,2)$ and remove all vertices $y$ from $\D$ such that $|y_1| + |y_2| > \cil$, using induction on the value $r - |y_1| - |y_2|$.
	
	For the base case, let $x \in V(\D)$ with $|x_1| + |x_2| = r$. Since $d({\bf{0}}, x) \leq r$ and $|x_1|,|x_2|\leq \f$, we see that $x_1, x_2 \neq 0$ and $x_j = 0$ for all $3 \leq j \leq n$. If $r = 3$, then since  $|x_1|, |x_2| \leq \lfloor{\frac{3}{2}}\rfloor = 1$, we get  $|x_1| + |x_2| \leq 2$, contradicting the assumption. Thus we assume $r \geq 4$, which implies $|x_1|, |x_2| \geq 2$.
	Now, consider $\lambda^{[x;12]}$. Since $r \geq 4$ and ${\bf{0}} \prec x$, it follows that $x_2\geq 2$. Thus, from the definition of $\lambda^{[x;12]}$ we conclude that ${\bf{0}} \prec \lambda^{[x;12]}$ and $\lambda^{[x;12]} \in V(\Delta)$. Clearly, $d(x, \lambda^{[x;12]}) = 2 \leq r$ and therefore  $x \in \n[\lambda^{[x;12]}, \D]$.  We show that $\n[x, \D]  \subseteq \n[\lambda^{[x;12]}, \D]$.
	
	Let $z \in \n[x, \D]$ with $z \neq x$. If $|z_1| < |x_1|$ or $|z_2| < |x_2|$, then by Lemma \ref{twovertex}, $d(z, \lambda^{[x;12]}) \leq r$. Similarly, if $|z_1| \geq |x_1|$ and $sgn(z_1 x_1) = -1$, or $|z_2| \geq |x_2|$ and $sgn(z_2 x_2) = -1$, then again $d(z, \lambda^{[x;12]}) \leq r$.

	Now consider the case $|z_1| \geq |x_1|$, $sgn(z_1 x_1) = 1$ and  $|z_2| \geq |x_2|$, $sgn(z_2 x_2) = 1$. Since $|x_1| + |x_2| = r$, it follows that $z_1 = x_1$ and $z_2 = x_2$, hence $d(z, \lambda^{[x;12]}) = 2 \leq r$. Therefore, $\n[x, \D] \subseteq \n[\lambda^{[x;12]}, \D]$. From Proposition \ref{flag}, $\D \simeq \D \setminus x$. 
	
	By repeating the above argument for all vertices $y$ such that $|y_1| + |y_2| = r$,  we get that 
	$
	\Delta \simeq \mathrm{Ind}_{\Delta}(V(\Delta) \setminus \{y\in V(\D) : |y_1|+ |y_2| = r\}).$
	
	Assume that, without changing the homotopy type of  $\Delta$, we have removed all the vertices $y$ from $\Delta$ such that $r \geq |y_1| + |y_2| \geq \cil + t + 1$ for some $t \geq 1$.  Let  $\Delta_1$ denote the resulting subcomplex of $\Delta$ after removal of all such vertices. Then   $\Delta_1$ is the induced  subcomplex on  $\{a \in V(\D) : |a_1| + |a_2| \leq \cil + t, t\geq 1\}$ and $\Delta \simeq \Delta_1$.

	Let $p \in V(\D_1)$ with $|p_1| + |p_2| = \cil + t$. By the definition of $\D_1$,  we have $p \succ {\bf{0}}$, moreover, since $r \geq 3$ and $|p_1|, |p_2| \leq \f$, we get  $p_1, p_2 \neq 0$.  Now, we have the following cases:

	\noindent\textbf{Case (i):} 
	Suppose that either $p_2 \geq 2$, or $p_2\leq 1$ and there exists $k > 2$ such that $p_k>0$. Without loss of generality, assume that when $p_2 \leq 1 $,  $k$ is the largest such index. Since $p \succ {\bf{0}}$, $p_j =0$ for all $j > k$. 
	
	Consider the element $\lambda^{[p;12]}$. Observe that $d(p, \lambda^{[p;12]}) = 2$, $|\lambda^{[p;12]}_1 - p_1| = 1$, $|\lambda^{[p;12]}_2 - p_2| = 1$, and $\lambda^{[p;12]}_j = p_j$ for all $j \geq 3$.  In case of  $p_2\leq 1$, $\lambda^{[p;12]}_k > 0$ and $\lambda^{[p;12]}_j = 0$ for all $j > k$. Therefore, $\lambda^{[p;12]} \succ {\bf{0}}$. In case of $p_2\geq 2$, there exists $s\geq 2$ such that $p_s>0$ and $s$ is the largest such index. Then from the definition of $\lambda^{[p;12]}$, $\lambda^{[p;12]}_s>0$ and $\lambda^{[p;12]}_j=0$, for all $j\geq s$. Therefore, $\lambda^{[p;12]} \succ {\bf{0}}$. Moreover, since $|\lambda^{[p;12]}_j|=|p_j|-1$ for $j\in\{1,2\}$, we have $|\lambda^{[p;12]}_1|+|\lambda^{[p;12]}_2|\leq \cil + t-2$. Therefore, $\lambda^{[p;12]}\in V(\Delta_1)$.

	We claim that $\n[p, \D_1] \subseteq \n[\lambda^{[p;12]}, \D_1]$. Let $q \in \n[p, \D_1]$ with $q \neq p$. If $|q_1| < |p_1|$ or $|q_2| < |p_2|$, then $d(q, \lambda^{[p;12]}) \leq r$ by Lemma \ref{twovertex}. Similarly, if $|q_1| \geq |p_1|$ and $sgn(q_1 p_1) = -1$, or $|q_2| \geq |p_2|$ and $sgn(q_2 p_2) = -1$, then $d(q, \lambda^{[p;12]}) \leq r$.
	
	Now consider $|q_1| \geq |p_1|$, $sgn(q_1 p_1) = 1$, and $|q_2| \geq |p_2|$, $sgn(q_2 p_2) = 1$. Since $|q_1| + |q_2| \leq \cil + t$ and  $|p_1| + |p_2| = \cil + t$, we get $q_1 = p_1$ and $q_2 = p_2$. Then $\sum_{i \geq 3} |q_i| \leq \f -  t$ and $\sum_{i \geq 3} |p_i| \leq \f - t$. Thus,
	\begin{align*}
		d(q, \lambda^{[p;12]}) &= |q_1 - \lambda^{[p;12]}_1| + |q_2 - \lambda^{[p;12]}_2| + \sum_{i \geq 3} |q_i - \lambda^{[p;12]}_i|\\ &=  |p_1 - \lambda^{[p;12]}_1| + |p_2 - \lambda^{[p;12]}_2| + \sum_{i \geq 3} |q_i - p_i| \\
		& \leq 2 + \sum_{i \geq 3} |p_i| + \sum_{i \geq 3} |q_i| \leq 2 + 2\f - 2t\leq r + 2 - 2t \leq r. 
	\end{align*}
	Therefore, $\n[p, \D_1] \subseteq \n[\lambda^{[p;12]}, \D_1]$. From Proposition \ref{flag}, $\D_1 \simeq \D_1 \setminus p$.

	\noindent\textbf{Case (ii):} $p_2 \leq 1$ and $p_k \leq 0$ for all $k > 2$.
	
	Since $p \succ {\bf{0}}$ and $p_2 \neq 0$, we conclude that $p_2 =1$ and $p_j = 0$ for all $j > 2$.
	
	Using $|p_1| + |p_2| = \cil + t$, we get $|p_1| = \cil + t - 1$. Since $t \geq 1$, it follows that $|p_1| \geq \cil $.  Further, since $|p_1 | \leq \f$, we get that $r$ is even and $|p_1| = \frac{r}{2}$. Consider the element $\lambda^{[p;1]}$.
	
	Note that $|\lambda^{[p;1]}_1|=|p_1|-1$, and $\lambda_j^{[p;1]} = p_j$ for all $j \geq 2$. Since $|p_1|+|p_2|=\cil + t$ we get $|\lambda^{[p;1]}_1|+|\lambda^{[p;1]}_2| < \cil + t$. Using $p\succ {\bf{0}}$, $p_2 = \lambda_2^{[p;1]}=1$ and $p_j = \lambda_j^{[p;1]}=0$ for all $j>2$, we see that $\lambda^{[p;1]}\succ {\bf{0}}$ and $\lambda^{[p;1]} \in V(\Delta_1)$.
	
	Let $u \in \n[p, \D_1]$ with $u \neq p$. If $|u_1| < |p_1|$, then $d(u, \lambda^{[p;1]}) \leq r$ by Lemma \ref{onevetex}. Similarly, if $|u_1| \geq |p_1|$ and $sgn(u_1 p_1) = -1$, then $d(u, \lambda^{[p;1]}) \leq r$.
	
	Now consider $|u_1| \geq |p_1|$, $sgn(u_1 p_1) = 1$. Since $|u_1|\leq \frac{r}{2}$, we must have $u_1 = p_1$. Then $\sum_{i \geq 2} |u_i| \leq \frac{r}{2}$ and $\sum_{i \geq 2} |p_i| =1$. Thus,
	\begin{align*}
		d(u, \lambda^{[p;1]}) &= |u_1 - \lambda^{[p;1]}_1| + \sum_{i \geq 2} |u_i - \lambda^{[p;1]}_i| = |p_1 - \lambda^{[p;1]}_1| +  \sum_{i \geq 2} |u_i - p_i| \\
		& \leq 1 + \sum_{i \geq 2} |p_i| + \sum_{i \geq 2} |u_i| \leq 1 + 1 + \frac{r}{2} \leq r. 
	\end{align*}
	Thus, $\n[p, \D_1] \subseteq \n[\lambda^{[p;1]}, \D_1]$, and from Proposition \ref{flag},  $\D_1 \simeq \D_1 \setminus p$. 
	
	From above Case (i) and Case (ii), and using induction, we conclude that 
	$$\Gamma_n^{\alpha, r} \simeq \D \simeq \mathrm{Ind}_{\Delta}(V(\D) \setminus \{x \in V(\D) : |x_1| + |x_2| > \cil\}).
	$$

	Let $\Delta_2$ denote the induced complex  $\mathrm{Ind}_{\Delta}(V(\D) \setminus \{x \in V(\D) : |x_1| + |x_2| > \cil\})$. Let $l_1, l_2 \in [n]$ such that $(l_1, l_2) \neq (1, 2)$. Assume that  $\Delta_2$ is homotopy equivalent to the induced subcomplex of $\Delta_2$ on the vertex set $V(\Delta_2) \setminus \{u: |u_s| + |u_t| > \cil$ for all ordered pairs $(s, t) \prec (l_1, l_2) \}$.

	Using a similar argument as above, we get that 
	$$
	\Delta_2 \simeq \mathrm{Ind}_{\Delta_2}(V(\D_2) \setminus \{u \in V(\D_2) : |u_s| + |u_t| > \cil \ \forall \ (s, t) \preceq (l_1, l_2)  \}).
	$$
	By induction, we conclude that 
	$$\Gamma_n^{\alpha, r} \simeq \D \simeq \Delta_2 \simeq \mathrm{Ind}_{\Delta_2}(V(\D_2) \setminus \{u \in V(\D_2) : |u_s| + |u_t| > \cil \ \forall \ s, t \in [n] \}).$$
	Hence, the result follows.
\end{proof}

% \begin{remark}
	% \rm {Need to write the explanation in case of $(l_1,l_2)$.}
	% \end{remark}

\begin{lemma}\label{lessthanfloor}
	Let $r\geq 4$. Then $\Gamma_n^{\alpha, r}$ is homotopy equivalent to the induced subcomplex of $\Gamma_n^{\alpha, r}$ on  $\{x\in \Gamma_n^{\alpha, r}: |x_i|< \f,  |x_j|+|x_k|\leq \cil \ \text{for all } \ i,j,k\in [n], j\neq k \}$. 
\end{lemma}
\begin{proof}
	It follows from Lemma \ref{xi+xj} that for $r\geq 4$, $\Gamma_n^{\alpha, r}$ is homotopy equivalent to the induced subcomplex of $\Gamma_n^{\alpha, r}$, say $\D$, on the following set of vertices:
	$$
	\{y \in V(\Gamma_n^{\alpha, r}) : |y_i| \leq \f, \ |y_j| + |y_k| \leq \cil \ \text{for all } i,j,k \in [n], j\neq k\}.
	$$
	To complete the proof, it is sufficient to show that $\D$ is homotopy equivalent to the induced subcomplex of $\D$ on the vertex set
	$$
	\{y \in V(\D) : |y_i| < \f \ \text{for all } i \in [n]\}.
	$$
	
	We shall remove the vertices from $\D$ without changing the homotopy types in three steps. In step I, we remove the vertices of type $u\in \D$ such that $|u_i|=\f$, and $\sum_{k\neq i}|u_k|\leq \cil-2$ for some $i\in [n]$. In step II,  we remove the vertices of type $v\in \D$ such that $|v_i|=\f$, and $\sum_{k\neq i}|v_k|= \cil-1$ for some $i\in [n]$. In the last step we remove the vertices of type $w\in \D$ such that $|v_i|=\f$, and $\sum_{k\neq i}|w_k|= \cil$ for some $i\in [n]$

    \smallskip
	\noindent \textbf{Step I:} Let $x$ be a vertex in $\D$ such that $|x_s| = \f$ and $\sum_{j \neq s} |x_j|\leq \cil - 2$. If $|x_k| > 1$ for some $k \neq s$, then  $|x_s| + |x_k| > \cil$, which is a contradiction. Hence, $|x_k| \leq 1$ for all $k \neq s$. Let us consider $\lambda^{[x;s]}\in V(\Z^n)$. Then, $d(x,\lambda^{[x;s]})=1$. Since $x\succ {\bf{0}}$ and $r\geq 4$, we conclude that $\lambda^{[x;s]}\succ {\bf{0}}$ and $\lambda^{[x;s]}\in V(\D)$.

	We show that $\n[x,\D]\subseteq \n[\lambda^{[x;s]},\D]$. For this, let  $z \in \n[x, \D]$,  $z \neq x$. If $|z_s| < \f$, or $|z_s| \geq  \f$ and $sgn(z_s x_s) = -1$, then by Lemma \ref{onevetex}, $d(z, \lambda^{[x;s]}) \leq r - 1$. So, assume that $|z_s| \geq  \f$ and $sgn(z_s x_s) = 1$. Now,  $|z_s| \leq  \f$ implies that $|z_s| = \f$. Then $z_s = x_s$, $|z_k| \leq 1$ for all $k \neq s$, and $\sum_{k \neq s} |z_k| \leq \cil$. Therefore,
	$$
	\begin{aligned}
		d(z, \lambda^{[x;s]}) &= |z_s - \lambda^{[x;s]}_s| + \sum_{k \neq s} |z_k - \lambda^{[x;s]}_k| \\
		&= 1 + \sum_{k \neq s} |z_k - x_k| \leq 1 + \sum_{k \neq s} |z_k| + \sum_{k \neq s} |x_k| \leq 1 + \cil + (\cil - 2) \leq r.
	\end{aligned}
	$$
	Thus, $z \in \n[\lambda^{[x;s]}, \D]$. Hence $\n[x,\D]\subseteq \n[\lambda^{[x;s]},\D]$ and therefore $\D \simeq \D \setminus x$. 
	
	Let $\A=\{a\in \D: |a_i|=\f\ \text{and}\ \sum_{k\neq i}|a_k|\leq \cil-2 \ \text{for some} \ i\in[n]\}$. By repeating the above process for each $a\in \A$, we find that $\D \simeq \D_1$, where $\D_1 = \mathrm{Ind}_{\D}(V(\D) \setminus \A)$. 
	
    \smallskip
	\noindent 
	\textbf{Step II:} Let $p \in V(\D_1)$ such that $|p_s| = \f$ and $\sum_{j \neq s} |p_j| = \cil - 1$. Since $r \geq 4$, there exists $j \neq s$ such that 
	$|p_j| \neq 0$. If $r$ is even, then $|p_s| + |p_j| >  \cil = \frac{r}{2}$, which is a contradiction as $p \in \Delta$. Therefore, $r$ must be odd and $r \geq 5$. 
	Since $|p_s|+|p_k|\leq \cil$ for all $k \neq s$, it follows that $|p_k|\leq 1$ for all $k \neq s$. Since $r\geq 5$, there exists $j,l\in [n] \setminus \{s\}$ such that  $j<l$ and $|p_j|  = |p_l|=1$. 
	
	Let us take $\lambda^{[p;sj]}\in V(\Z^n)$. Clearly, $d(p,\lambda^{[p;sj]})=2$. 
	Let $p_{i_0}$ be the last non-zero entry of $p$. Since $p\succ {\bf{0}}$, we must have $p_{i_0}>0$. Since $|p_l|=1$ and $l>j$, we have $i_0>j$. Now, if $i_0=s$, then, $\lambda^{[p;sj]}_s>0$ and $\lambda^{[p;sj]}_s$ is the last non-zero entry of $\lambda^{[p;sj]}$. If $i_0\neq s$, then $p_{i_0}=\lambda^{[p;sj]}_{i_0}>0$. Therefore, we conclude that $\lambda^{[p;sj]}\succ {\bf{0}}$. Since $|\lambda^{[p;sj]}_s|\leq \f-1$, and $\sum_{k\neq s}|\lambda^{[p;sj]}_s|< \cil-1$, it follows that $\lambda^{[p;sj]}\in V(\D_1)$. We first show that  $\n[p, \D_1] \subseteq \n[\lambda^{[p;sj]}, \D_1]$.

	Let $q \in \n[p, \D_1]$, $q \neq p$. If $|q_s| < |p_s|$ or $|q_j| < |p_j|$, then by Lemma \ref{twovertex}, $d(q, \lambda^{[p;sj]}) \leq r$. If $|q_s| \geq |p_s|$ and $sgn(q_s p_s) = -1$, or $|q_j| \geq |p_j|$ and $sgn(q_j p_j) = -1$, then again by Lemma \ref{twovertex}, $d(q, \lambda^{[p;sj]}) \leq r$.
	
	Now, consider the case when $|q_t| \geq |p_t|$ and $sgn(q_t p_t) = 1$ for $t\in\{s,j\}$. Since $|q_s|+|q_j|\leq \cil=|p_s|+|p_j|$,  we get $|q_s| = |p_s|$ and $|q_j| = |p_j|$. This implies that $q_s= p_s$ and $q_j = p_j$.
    Since  $\sum_{k \neq s,j} |q_k| \leq \f$, we get  
	$$
	\begin{aligned}
		d(q, \lambda^{[p;sj]}) &= |q_s - \lambda^{[p;sj]}_s| + |q_j - \lambda^{[p;sj]}_j| + \sum_{k \neq s,j} |q_k - \lambda^{[p;sj]}_k|\\ 
		&= 1 + 1 + \sum_{k \neq s,j} |q_k| + \sum_{k \neq s,j} |\lambda^{[p;sj]}_k| \leq 2 + \f + (\cil - 2) \leq r  .
	\end{aligned}
	$$
	Hence, $q \in \n[\lambda^{[p;sj]}, \D_1]$. Thus $\n[p, \D_1] \subseteq \n[\lambda^{[p;sj]}, \D_1]$. From Proposition \ref{flag},  $\D_1 \simeq \D_1 \setminus p$. 
	Let $B=\{a\in V(\D_1): |a_s|=\f\ \text{and}\ \sum_{k\neq s}|a_k|=\cil -1 \ \text{some some}\ s\in [n]\}$. By repeating the same argument for every vertex $b\in B $,  we get that $\D_1 \simeq \Delta_2$, where $ \Delta_2 = \mathrm{Ind}_{\D_1}(V(\D_1) \setminus B)$.

    \smallskip
	\noindent 
	\textbf{Step III:}
	Let $w \in V(\D_2)$ with $|w_l| = \f$ and $\sum_{k \neq l} |w_k| = \cil$.  If $r$ is even, then $|p_s| + |p_j| >  \cil = \frac{r}{2}$ for some $j \neq s$,  a contradiction as $p \in \Delta$. Thus, $r$ must be odd and $r\geq 5$.
	
	For each $k \neq l$, $|w_l|+ |w_k| \leq \cil$ implies that $|w_k| \leq 1$.  Let us consider $\lambda^{[w;l]}\in V(\Z^n)$. Let $w_{j_0}$ be the last non-zero entry of $w$. Since $w\succ {\bf{0}}$, we have $w_{j_0}>0$. If $j_0=l$, then clearly $\lambda^{[w;l]}_{j_0}>0$ and $\lambda^{[w;l]}_{j_0}$ is the last non-zero entry. If ${j_0}\neq l$, then 
	$w_{j_0}=\lambda^{[w;l]}_{j_0}$. This implies that $\lambda^{[w;l]}_{j_0}>0$ and $\lambda^{[w;l]}_{j_0}$ is the last non-zero entry.  Thus, we conclude that $\lambda^{[w;l]}\succ {\bf{0}}$. Since $|\lambda^{[w;l]}_l|\leq \f-1$ and $\sum_{ l\neq j} |\lambda^{[w;l]}_j|=\cil$, and $|\lambda^{[w;l]}_k|\leq 1$ for all $k\neq l$, we find that $\lambda^{[w;l]}\in V(\D_2)$. We claim that $\n[w, \D_2] \subseteq \n[\lambda^{[w;l]}, \D_2]$.
	
	Let $v\in \n[w, \D_2]$ with $v \neq w$. Note that if $|v_l| < \f$, then $d(v, \lambda^{[w;l]}) \leq r - 1$, and hence $v\in \n[\lambda^{[w;l]}, \D_2]$. So assume that $|v_l| = \f$.  Now, if $sgn(v_l w_l) = -1$, then $d(v, \lambda^{[w;l]}) \leq r$ and we are done. 
	
	Suppose $sgn(v_l w_l) = 1$. Then $v_l = w_l$. If $\sum_{j \neq l} |v_j| \leq \cil-1$,  then $v \notin \D_2$, so we must have $\sum_{k \neq l} |v_j| = \cil$.

	Let $S = \{j \in [n] \setminus \{l\} : w_j, v_j \neq 0 \ \text{and} \ sgn(v_j w_j) = -1  \}$. Since for each $j \neq l , |v_j| \leq 1$, we have 
	$
	\sum_{j \in S} |v_j - w_j| = 2 \cdot \mathrm{Card}(S),
	$
	and for $j \notin S \cup \{l\}$, either $w_j = 0$ or $v_j = 0$. Therefore,
	$$
	\sum_{j \notin S \cup \{l\}} |w_j| = \sum_{j \notin S \cup \{l\}} |v_j| = \cil - \mathrm{Card}(S).
	$$
	Hence,
	$$
	\begin{aligned}
		d(v, w) &= |v_l - w_l| + \sum_{j \in S} |v_j - w_j| + \sum_{j \notin S \cup \{l\}} |w_j - v_j| \\
		&= 0 + 2\cdot \mathrm{Card}(S) + 2(\cil - \mathrm{Card}(S)) = 2\cil = r + 1,
	\end{aligned}
	$$
	which contradicts $v \in \n[w,\D_2]$. Therefore, $sgn(v_l w_l)$ cannot be 1. Hence, we conclude that $\n[w, \D_2] \subseteq \n[\lambda^{[w;l]}, \D_2]$.
	
	From Proposition \ref{flag}, we have $\D_2 \simeq \D_2 \setminus w.$ Let $C=\{c\in V(\D_2): |c_i|=\f\ \text{and}\ \sum_{j\neq i}|c_j|=\cil\ \text{for some } i\in [n]\}$. By repeating the same above arguments for each $c\in C$, we find that  $\D_2 \simeq \mathrm{Ind}_{\D_2}(V(\D_2) \setminus C).$ 
	
	Therefore, for $r\geq 4$, $\Gamma_n^{\alpha,r}\simeq \D \simeq \D_1 \simeq \D_2\simeq \mathrm{Ind}_{\D_2}(V(\D_2) \setminus C)$, and this completes the proof.
\end{proof}

\begin{lemma}\label{threevertex}
	Let $r\geq 2$, and let  $\D$ be  a subcomplex of $\Gamma_n^{\alpha, r}$ that contains the vertex $x$. Then for any two vertices $y$ and $z$ in $\D$, where $z\in V(\lk(x,\D))$ and  $|y_l|=|x_l|-1$ for $l\in\{i,j,k\}$ and $y_l=x_l$ for $l\in [n]\setminus\{i,j,k\}$, we have the following: 
	\begin{itemize}
		\item[(i)]  If $|z_l|<|x_l|$ for at least two choices of $l\in\{i,j,k\}$,  then $d(z,y)\leq  r-1$.
		\item[(ii)]  Let $|z_k|<|x_k|$. If $|z_i|\geq |x_i|$ and $sgn(z_ix_i)=-1$,  or $|z_j|\geq |x_j|$ and $sgn(z_jx_j)=-1$, then $d(z,y)\leq  r$.
		\item[(iii)] If $|z_l|\geq |x_l|$ for all $l\in\{i,j,k\}$ and $sgn(z_sx_s)=-1$ for at least two choices of $s\in\{i,j,k\}$, then $d(z,y)\leq r$.
	\end{itemize}
	
\end{lemma}
\begin{proof}
	\begin{itemize}
		\item[(i)] 
		Without loss of generality, we assume that $|z_i|<|x_i|$ and $|z_j|<|x_j|$. Then,  $|z_i - y_i| = |z_i - x_i| - 1$ and $|z_j - y_j| = |z_j - x_j| - 1$. 
		Since $|y_k|=|x_k|-1$, it follows that $|z_k-y_k|\leq |z_k-x_k|+1$. Therefore, $d(z,y) = |z_i - y_i| + |z_j - y_j| +|z_k-x_k|+ \sum_{l \neq i,j,k} |z_l - y_l|,$ which gives $d(z,y)\leq (|z_i - x_i| - 1) + (|z_j - x_j| - 1) + (|z_k-x_k|+1)+\sum_{l \neq i,j,k} |z_l - x_l| \leq d(z,x)-1\leq r-1$. 
		
		\item[(ii)] Let $|z_k|<|x_k|$. Then $|z_k-y_k|=|z_k-x_k|-1$. If $|z_i|\geq |x_i|$ and $sgn(z_ix_i)=-1$, then $|z_i-y_i|=|z_i-x_i|-1$. Since $|y_j|=|x_j|-1$, $|z_k-y_k|\leq |z_k-x_k|+1$. Hence, a similar calculation as above shows that $d(z,y)\leq r$. Similarly, the result follows when $|z_j|\geq |x_j|$ and $sgn(z_jx_j)=-1$. 
		
		\item[(iii)] It is given that $|z_l|\geq |x_l|$ for all $l\in\{i,j,k\}$. Now, assume that $sgn(z_ix_i)=sgn(z_jx_j)=-1$. Then $|z_i-y_i|=|z_i-x_i|-1$, and $|z_j-y_j|=|z_j-x_j|-1$. Since $|y_k|=|x_k|-1$, regardless of whether $|z_k|<|x_k$ or $|z_k|\geq |x_k|$, we find that $|z_k-y_k|\leq |z_k-x_k|+1$. 
		Now, a similar computation as in part $(i)$ proves the result.
	\end{itemize}
    \vspace{- 0.4 cm}
\end{proof}

\begin{lemma}\label{sumofthree}
	Let $r\geq n\geq 4$ . Then $\Gamma_n^{\alpha, r}$ is homotopy equivalent to the induced subcomplex $\Delta'$ of $\Gamma_n^{\alpha, r}$, where every $x \in V(\Delta')$ satisfies the following: (i) $|x_i|< \f, |x_j|+|x_k|\leq \cil$ for all $i, j, k \in [n], j \neq k$, and (ii) $|x_i| + |x_j| + |x_k| < r-1$ for all distinct $i, j, k \in [n]$. 
\end{lemma}
\begin{proof}
	From Lemma \ref{lessthanfloor}, $\Gamma_n^{\alpha, r}$ is homotopy equivalent to the induced subcomplex of $\Gamma_n^{\alpha, r}$, say, $\D$, on the vertex set  $\{x\in \Gamma_n^{\alpha, r}: |x_i|< \f,  |x_j|+|x_k|\leq \cil \ \text{for all } \ i,j,k\in [n], j\neq k \}$.
	Let $A=\{x\in \D:|x_i| + |x_j| + |x_k| \geq r-1 \ \text{for some distinct}\  i,j,k \}$.
	We show that $\D\simeq \mathrm{Ind}_{\Delta}(V(\D)\setminus A)$.

	Let $z \in A$. Then there exist distinct $i, j, k$ such that $|z_i|+|z_j|+|z_k| \geq r-1$. Without loss of generality, we assume that $i<j<k$. Since $z \in \Delta$, we have
	$|z_i|+|z_j|\leq \cil$ and $|z_k|<\f$, which implies that  $|z_i|+|z_j|+|z_k| \neq  r$. Hence $|z_i|+|z_j|+|z_k| =  r-1$, and  $\sum_{l \neq i,j,k} |z_l| \leq 1$.
	
	If $|z_i| + |z_j| < \cil$, then  $|z_k| < \f$ implies that $|z_i| + |z_j| + |z_k| < r - 1$, which is a contradiction. Hence $|z_i| + |z_j| = \cil$ and $|z_k| = \f - 1$.
	Similarly, we deduce that $|z_j| + |z_k| = \cil$, $|z_i| + |z_k| = \cil$, and $|z_i| = |z_j| = \f - 1$. This  implies that $3 \f -3 = r-1$ and hence $r \in \{4, 7\}$.
	
	\medskip
	
	\noindent\textbf{Case (i):} $r=4$.

	\noindent Here, $|z_i|= |z_j| = |z_k| = 1$, and since $r\geq n\geq 4$, we have $n=4$. Since $i<j<k$ and $z\succ {\bf{0}}$, it follows that $z_k=1$, and  $3\leq k\leq 4$. 
	Let $\D$ contains a vertex $a$ such that $a_4>0$. In this case, let $\lambda\in V(\Z^n)$ be defined by $\lambda_4 = 1$ and $\lambda_l = 0$ for all $l \neq 4$. Then $\lambda\in \D$.
	Let $y\in \n[z,\D]$. Then from the fact that $|y_{l}|<2$ for any $l\in [4]$,  we have $|y_l| \leq 1$ for all $l\in [4]$. 
	
	Therefore, $d(y, \lambda) \leq 4$.
	Hence, $y\in \n[\lambda,\D]$. Thus, $\n[z,\D]\subseteq \n[\lambda,\D]$. From Proposition \ref{flag}, $\D\simeq \D \setminus z$. By repeating the similar argument as above for each element of $A$, we conclude that  $\D\simeq \mathrm{Ind}_{\D}(V(\D)\setminus A)$.
	
	Assume that there is no element $b\in \D$ such that $b_4>0$. This implies that $k=3$, $|z_1|=|z_2|=1$, $z_3=1$, and $z_4=0$.

	Consider a vertex $\lambda' \in V(\Z^n)$ such that $\lambda'_3 = 1$ and $\lambda'_l=0$ for $l\neq 3$. Then $\lambda'\in \D$, and 
	for any vertex $u\in\n[z,\D]$, $d(u,\lambda')\leq 4$. Thus, $u\in \n[\lambda',\D]$. Hence $\n[z,\D]\subseteq \n[\lambda',\D]$. From Proposition \ref{flag}, $\D\simeq \D\setminus z$. By repeating the same process for each $a\in A$, we get that $\D\simeq \mathrm{Ind}_{\D}(V(\D)\setminus A)$.

	\noindent\textbf{Case (ii):} $r = 7$.
	
	\noindent Here, $|z_i| = |z_j| = |z_k| = 2$.
	Consider the element  $\lambda^{[z;ijk]}\in V(\Z^n)$. Clearly,  $d(z,\lambda^{[z;ijk]})=3$ and  $|z_l-\lambda^{[z;ijk]}_l|=1$ for $l\in \{i,j,k\}$.
	Since $z\succ {\bf{0}}$, there exists $s\geq k$ such that $z_{s}>0$ and $z_{t}=0$ for all $t > s$. Since $|z_k|=2$, $|\lambda^{[z;ijk]}_k|=|z_k|-1$, and $\lambda^{[z;ijk]}_l=z_l$ for all $l>k$, we see that $\lambda^{[z;ijk]}_{s}>0$ and $\lambda^{[z;ijk]}_{t}=0$ for all $t > s$. Thus, $\lambda^{[z;ijk]}\succ {\bf{0}}$. Since $|\lambda^{[z;ijk]}_l|= |z_l|-1$ for all $l\in \{i,j,k\}$, and $\sum_{l\neq i,j,k}|\lambda^{[z;ijk]}_l|\leq 1$, we see that $|\lambda^{[z;ijk]}_{i_0}|+|\lambda^{[z;ijk]}_{j_0}|+|\lambda^{[z;ijk]}_{k_0}|\leq 3$ for any $i_0,j_0,k_0\in [n]$. Hence $\lambda^{[z;ijk]} \in \D$.
	
	\begin{claim}\label{claim1}
		$\n[z,\D] \subseteq  \n[\lambda^{[z;ijk]},\D]$.
	\end{claim}
	
	\noindent{\em Proof of Claim \ref{claim1}:} Clearly, $z\in \n[\lambda^{[z;ijk]},\D]$. Let $w\in \n[z,\D]$ with $z\neq w$. 
	On the contrary, suppose that $w\notin \n[\lambda^{[z;ijk]}, \D]$. Then  $d(w, \lambda^{[z;ijk]}) > 7$.  From Lemma \ref{threevertex} $(i)$, $|w_l| < |z_l|$ for at most one $l \in \{i, j, k\}$.
	
	We first consider the case that there exists $l \in \{i, j, k\}$ such that  $|w_l| < |z_l|$. If  $|w_i| < |z_i|$, then from Lemma \ref{threevertex} $(ii)$, we have $|w_j| \geq |z_j|$, $sgn(w_jz_j)=1$, $|w_k| \geq |z_k|$,  and $sgn(w_kz_k)=1$. Since $|w_j| + |w_k| \leq 4$  and $|z_j|+|z_k|=4$, we get that $w_j = z_j$ and $w_k = z_k$.  Further, $d({\bf{0}}, w) \leq 7$ implies that  $\sum_{l \neq j,k} |w_l| \leq 3$. Therefore,
	\begin{align*}
		d(w, \lambda^{[z;ijk]}) &= |w_i - \lambda^{[z;ijk]}_i| + |w_j - \lambda^{[z;ijk]}_j| + |w_k - \lambda^{[z;ijk]}_k| + \sum_{l \neq i,j,k} |y_l - \lambda^{[z;ijk]}_l| \\
		&= |w_i - z_i| - 1 + |w_j - z_j| + 1 + |w_k - z_k| + 1 + \sum_{l \neq i,j,k} |w_l - z_l| \\
		&\leq |z_i| + 1 + \sum_{l \neq j,k} |w_l| + \sum_{l \neq i,j,k} |z_l|\leq 2 + 1 + 3 + 1 = 7.
	\end{align*}
	This implies that $d(w,\lambda^{[z;ijk]})\leq 7$, which is a contradiction to our assumption. Similarly, if $|w_j| < |z_j|$ or $|w_k| < |z_k|$, then we arrive at a contradiction.
	
	Now, assume that $|w_l| \geq |z_l|$ for all $l \in \{i, j, k\}$. Since $|z_i|=|z_j|=|z_k|=2$ and $w \in \Delta$, we have $|w_l| = 2$ for $l \in \{i, j, k\}$. Consequently, it follows that $sgn(w_lz_l)=-1$ for at most one $m \in \{i, j, k\}$; otherwise, a similar computation as in part $(iii)$ of Lemma \ref{threevertex} yields that $d(w, \lambda^{[z;ijk]}) \leq 7$, which is a contradiction.

	Without loss of generality, we assume that $sgn(w_iz_i)=-1$, and $w_j = z_j$, $w_k = z_k$. Therefore, we have $|w_i - \lambda^{[z;ijk]}_i| = |w_i - z_i| - 1, \quad |w_j - \lambda^{[z;ijk]}_j| = |w_j - z_j| + 1 = 1, \quad |w_k - \lambda^{[z;ijk]}_k| = 1,$ and $\sum_{l \neq i,j,k} |w_l| \leq 1$, implying that $d(w, \lambda^{[z;ijk]}) \leq 7$, which is a contradiction.
	
	Thus, we left with the only case that  $sgn(w_lz_l)=1$, for all $l \in \{i, j, k\}$. This implies that $w_l = z_l$ for all $l \in \{i, j, k\}$.
	Then, clearly,  $\sum_{l \neq i,j,k} |w_l|  \leq 1$, and $\sum_{l \neq i,j,k} |z_l|  \leq 1$ implies that $d(w, \lambda^{[z;ijk]}) \leq 5$.   Hence, $d(w,\lambda^{[z;ijk]})< 7$ and thus $w\in \n[\lambda^{[z;ijk]},\D]$. This is a contradiction.
	
	Hence, $w\in \n[\lambda^{[z;ijk]},\D]$. Thus, $\n[z,\D]\subseteq \n[\lambda^{[z;ijk]},\D]$. This completes the proof of Claim \ref{claim1}.

	Using Claim \ref{claim1} and  Proposition \ref{flag},  we get $\D\simeq \D \setminus z$. By using the similar argument  for each $a\in A$, we get that $\D\simeq \mathrm{Ind}_{\D}(V(\D)\setminus A)$.
\end{proof}

\begin{lemma}\label{fourvertex}
	Let $r\geq 4$, and let  $\D$ be  a subcomplex of $\Gamma_n^{\alpha, r}$ that contains the vertex $x$. Then for any two $y, z \in V(\D)$, where $z\in V(\lk(x,\D))$ and  $|y_t|=|x_t|-1$ for $t\in\{i,j,k,l\}$ and $y_s=x_s$ for $s\in [n]\setminus\{i,j,k,l\}$, we have the following: 
	\begin{enumerate}[label=(\roman*)]
		\item If $|z_s|<|x_s|$ for at least two choices of $s\in\{i,j,k,l\}$,  then $d(z,y)\leq r$. 
		\item Let $|z_s|<|x_s|$ for some $s \in \{i, j, k, l\}$. If there exists $t \in \{i, j, k, l\} \setminus \{s\}$ such that $|z_t|\geq |x_t|$ and $sgn(z_tx_t)=-1$, then $d(z,y)\leq  r$. 
		\item If $|z_t|\geq |x_t|$ for all $t\in\{i,j,k,l\}$, and  $sgn(z_sx_s)=-1$  for at least two indices $s\in\{i,j,k,l\}$, then $d(z,y)\leq r$.
	\end{enumerate}
	
\end{lemma}

\begin{proof} 
	
	\begin{enumerate}[label=(\roman*)]
		\item Without loss of generality, assume that $|z_i|<|x_i|$ and $|z_j|<|x_j|$. Thus,
		$|z_i - y_i| = |z_i - x_i| - 1$ and  $|z_j - y_j| = |z_j - x_j| - 1$. 
		Since $|y_t|=|x_t|-1$ for $t\in\{i,j,k,l\}$, and $d(y,x)=4$, we see that $|z_k-y_k|\leq |z_k-x_k|+1$, and $|z_l-y_l|\leq |z_l-x_l|+1$. Therefore, $d(z,y)\leq (|z_i - x_i| - 1) + (|z_j - x_j| - 1) + (|z_k-x_k|+1)+ (|z_l-x_l|+1)+\sum_{t \neq i,j,k,l} |z_t - x_t| \leq d(z,x)=r$. This proves $(i).$
		
		\item Without loss of generality, assume that $|z_k| < |x_k|$, and $|z_i| \geq |x_i|$ and $sgn(z_ix_i)=-1$. Then, $|z_k - y_k| = |z_k - x_k| - 1$ and, $|z_i - y_i| = |z_i - x_i| - 1$. Moreover, $|z_j - y_j| \leq |z_j - x_j| + 1$ and $|z_l - y_l| \leq |z_l - x_l| + 1$. Therefore, $d(z, y) \leq (|z_i - x_i| - 1) + (|z_j - x_j| + 1) + (|z_k - x_k| - 1) + (|z_l - x_l| + 1) + \sum_{t \neq i,j,k,l} |z_t - x_t| \leq d(z, x) = r.$
		This proves $(ii)$.
		
		\item Let $|z_t| \geq |x_t|$ for all $t \in \{i,j,k,l\}$. Without loss of generality, we assume that $sgn(z_ix_i)=sgn(z_jx_j)=-1$. This implies that $|z_i - y_i| = |z_i - x_i| - 1$ and $|z_j - y_j| = |z_j - x_j| - 1$. Also, we have $|z_k - y_k| \leq |z_k - x_k| + 1$, and $|z_l - y_l| \leq |z_l - x_l| + 1$. Now, similar computations as above show that $d(z, y) \leq r$. This proves $(iii)$.
	\end{enumerate}
    \vspace{- 0.4 cm}
	\end{proof}

\begin{lemma}\label{sumoffour}
	Let $r\geq n\geq 4$ . Then $\Gamma_n^{\alpha, r}$ is homotopy equivalent to the induced subcomplex $\Delta'$ of $\Gamma_n^{\alpha, r}$, where every $x \in V(\Delta')$ satisfies the following: (i) $|x_i|< \f, |x_j|+|x_k|\leq \cil$ for all $i, j, k \in [n], j \neq k$ (ii) $|x_i| + |x_j| + |x_k| < r-1$ for all $\{i, j, k\}\subseteq [n]$, and (iii) $|x_i| + |x_j| + |x_k| + |x_l|\leq r-1$ for all  $\{i, j, k, l\} \subseteq [n]$.
\end{lemma}

\begin{proof} From Lemma \ref{sumofthree}, $\Gamma_n^{\alpha, r}$ is homotopy equivalent to the induced subcomplex $\Delta$, where every $x \in V(\Delta)$ satisfies the following: (i) $|x_i|< \f, |x_j|+|x_k|\leq \cil$ for all $i, j, k \in [n], j \neq k$  and (ii) $|x_i| + |x_j| + |x_k| < r-1$ for all  $\{i, j, k \}\subseteq [n]$. 
	
	Let $x\in V(\D)$ be such that $|x_i| + |x_j| + |x_k| + |x_l| = r$ for some  $i< j<k<l$.  
	For any $t \in \{i, j, k, l\}$, if  $|x_t| \leq 1$, then $\Sigma_{s \in \{i, j, k, l\} \setminus \{t\}} |x_s| \geq r-1$, a contradiction. Hence $|x_t| \geq 2$ for all $t \in \{i, j, k, l\}$ and $x_s=0$ for $s \notin \{i, j, k, l\}$. Hence $r \geq 8$. 
	
	Now, let us consider $\lambda^{[x;ijkl]}\in V(\Z^n)$. We see that $d(x,\lambda^{[x;ijkl]})=4$, and $|x_t-\lambda^{[x;ijkl]}_t|=1$ for all $t\in \{i,j,k,l\}$. Since $x\succ {\bf{0}}$, $i<j<k<l$, and 
	$x_t=0$ for $t>l$, we get that $x_l>0$. Now $x_l \geq 2$ implies that   ${\bf{0}} \prec \lambda^{[x;ijkl]} \prec x$. Since $|\lambda^{[x;ijkl]}_t|<|x_t|$ for all $t\in \{i,j,k,l\}$, and $\lambda^{[x;ijkl]}_t=x_t$ for all $t\not\in \{i,j,k,l\}$ it follows that $\lambda^{[x;ijkl]} \in \D$.

	We show that $\n[x, \D] \subseteq \n[\lambda^{[x;ijkl]}, \D]$. Clearly, $x \in  \n[\lambda^{[x;ijkl]}, \D]$. Let $u \in \n[x, \D]$, $u\neq x$.  Suppose $u \notin \n[\lambda^{[x;ijkl]}, \D]$, {\it i.e.}, $d(u, \lambda^{[x;ijkl]}) > r$. First, assume that $|u_t| \geq |x_t|$ for all $t \in \{i, j, k, l\}$. Then from the fact $|x_i|+|x_j|+|x_k|+|x_l|=r$, we  have $|u_t| = |x_t|$ for all $t\in \{i,j,k,l\}$ and $u_s = 0$ for $s \notin \{i, j, k, l\}$. Here, from Lemma \ref{fourvertex} (iii), $sgn(u_tx_t)=-1$ for at most one $t \in \{i, j, k, l\}$.   Now, $|u_t| = |x_t|$ for all $t \in \{i, j, k, l\}$ and $ u \neq x$ implies that there exists exaclty one $s \in \{i, j, k, l\}$ such that  $sgn(u_sx_s) = -1$.  Then $|u_s - \lambda^{[x;ijkl]}_s| = |u_s - x_s| - 1$ and $|u_t - \lambda^{[x;ijkl]}_t| = 1$ for $t \in \{i, j, k, l\} \setminus \{s\}$. Therefore, $d(u, \lambda^{[x;ijkl]}) = |u_i - \lambda^{[x;ijkl]}_i| + |u_j - \lambda^{[x;ijkl]}_j| + |u_k - \lambda^{[x;ijkl]}_k| + |u_l - \lambda^{[x;ijkl]}_l| = |u_s - x_s| - 1 + 1 + 1 + 1 = 2|x_s| + 2.$
	Since for any $t \in \{i, j, k, l\} \setminus \{s\}$, $|x_s| + |x_t| \leq \cil$ and $|x_t| \geq 2$, we get $|x_s| \leq \cil - 2$. Hence,
	$d(u, \lambda^{[x;ijkl]}) \leq  2(\cil - 2) + 2 = 2\cil - 2 < r,$ which is a contradiction.

	Now, we assume that there exists  $p \in \{i, j, k, l\}$ such that $|u_p| < |x_p|$. For any $t \in \{i, j , k ,l\}\setminus\{p\}$, if $|u_t| < |x_t|$, then by Lemma \ref{fourvertex} (i), $d(u,\lambda^{[x;ijkl]})\leq r$, a contradiction. Hence  $|u_t| \geq |x_t|$ for $t \in \{i, j, k, l\}\setminus\{p\} $. 
	
	From Lemma \ref{fourvertex} (ii), $sgn(u_tx_t)=1$ for $t \in \{i,j, k, l\} \setminus \{p\}$.  Write $|u_t| = |x_t| + a_t$ for some $a_t \geq 0$ and $t \in \{i,j, k, l\} \setminus\{p\}$. Then $\Sigma_{t \in \{i,j, k, l\} \setminus\{p\}} |u_t| = \Sigma_{t \in \{i,j, k, l\} \setminus\{p\}} (|x_t| +a_t)$. Hence $\Sigma_{t \not\in \{i,j, k, l\} \setminus\{p\}} |u_t| \leq r- ( \Sigma_{t \in \{i,j, k, l\} \setminus\{p\}} (|x_t| +a_t))$. Since $|u_p| < |x_p|$, we have $|u_p - \lambda^{[x;ijkl]}_p| = |u_p - x_p| - 1$.  
	Therefore,
	\begin{align*}
		d(u, \lambda^{[x;ijkl]}) &= \sum_{t \in \{i, j, k, l\}}|u_t - \lambda^{[x;ijkl]}_t| 
         + \sum_{t \notin \{i, j, k, l\}} |u_t| \\
		&\leq |u_p - x_p| - 1 + \sum_{t \in \{i,j, k, l\} \setminus\{p\}} (|u_t - x_t| + 1)  + \sum_{t \notin \{i, j, k, l\}} |u_t| \\
		&\leq |u_p| + |x_p| +2+ \sum_{t \in \{i,j, k, l\} \setminus\{p\}} a_t + \sum_{t \notin \{i, j, k, l\}} |u_t| \\
		&\leq  |x_p| + 2+\sum_{t \in \{i,j, k, l\} \setminus\{p\}} a_t + \sum_{t \notin \{i, j, k, l\}\setminus\{p\}} |u_t| \\
		&\leq |x_p|+2+\sum_{t \in \{i,j, k, l\} \setminus\{p\}} a_t + r - ( \Sigma_{t \in \{i,j, k, l\} \setminus\{p\}} (|x_t| +a_t))  \\
		&= r + 2 + |x_p| -  \sum_{t \in \{i,j, k, l\} \setminus\{p\}} |x_t|.
	\end{align*}
	
	Let $\{t_1,t_2,t_3\}=\{i,j,k,l\}\setminus \{p\}$. If $|x_{t_1}| + |x_{t_2}| < \f$, then $|x_p| + |x_{t_3}| \leq \cil$, implies $|x_i| + |x_j| + |x_k| + |x_l| < r$, which is a contradiction. Hence, $|x_{t_1}| + |x_{t_2}| \geq \f$. Also, since $|x_p| \leq \cil - 2$, we get: $
	d(u, \lambda^{[x;ijkl]}) \leq r + 2 + \cil - 2 - |x_{t_3}| - \f$.  Since $|x_{t_3}| \geq 2$, we conclude that $d(u, \lambda^{[x;ijkl]}) \leq r $, again contradicting our assumption that $d(u, \lambda^{[x;ijkl]}) > r$. 
	
	Thus, we conclude that $u\in \n[\lambda^{[x;ijkl]},\D]$. Since $x\in \n[\lambda^{[x;ijkl]}, \D]$, $\n[x,\D]\subseteq \n[\lambda^{[x;ijkl]}, \D]$. From Proposition \ref{flag}, $\D\simeq \D\setminus u$. Let $A=\{y\in \D: |y_i|+|y_j|+|y_k|+|y_l|=r \ \text{for all}\ \{i,j,k,l\}\subseteq [n] \}$. By repeating the above process for each $y\in A$, we find that $\D\simeq \mathrm{Ind}_{\D}(V(\D)\setminus A$. 
	This completes the proof.
\end{proof}

\begin{lemma}\label{sumoffour2}
	Let $r \geq n \geq 5$ and  $r\geq 10$. Then $\Gamma_n^{\alpha, r}$ is homotopy equivalent to the induced subcomplex $\Delta$, where $x \in V(\Delta)$ satisfies the following:
	(i) $|x_i|< \f, |x_j|+|x_k|\leq \cil$ for all $i, j, k \in [n], j \neq k$ and (ii) there exists no $\{i, j, k , l\}\subseteq [n]$ such that $|x_i|+ |x_j|+ |x_k| + |x_l| \geq  r-1$ and $x_s=0$ for all $s\notin\{i,j,k,l\}$.

\end{lemma} 

\begin{proof}
	From Lemma \ref{sumoffour},  $\Gamma_n^{\alpha, r}$ is homotopy equivalent to the induced subcomplex $\Delta_1$ of $\Gamma_n^{\alpha, r}$, where $x \in V(\Delta_1)$ satisfies the three conditions $(i), (ii)$ and $(iii)$, as given in Lemma \ref{sumoffour}. 
	
	Let $A = \{x \in V(\Delta_1) : \exists \ \{i, j, k, l\} \subseteq [n] \ \text{such that} \  |x_i|+ |x_j|+ |x_k| + |x_l| = r-1 \ \text{and }  x_s = 0  \text{ for all }  s \notin\{i,j,k,l\} \}$. It is sufficient to show that $\D_1\simeq \mathrm{Ind}_{\D_1}(V(\D_1)\setminus \{A\})$. 
	
 Let $B=\{z\in V(\D_1): \sum_{s\in\{i,j,k,l\}}|z_s|=r-1, \ \text{and} \ i<j<k<l, z_l\geq2\ \text{for some} \ \{i,j,k,l\}\subseteq[n] 
	\ \text{and} \  z_s=0 \ \text{for}\ s\notin\{i,j,k,l\} \}$. 
	
	Let $C=\{z\in V(\D_1)\setminus B: \sum_{s\in\{i,j,k,l\}}|z_s|=r-1 \ \text{and}\  i<j<k<l, z_l=1\ \text{for some} \ \{i,j,k,l\}\subseteq[n] 
	\ \text{and} \  z_s=0 \ \text{for}\ s\notin\{i,j,k,l\} \}$.
	Observe that $A=B\cup C$.
	
	We shall remove all the elements of $A$ from $\D_1$ without changing the homotopy type of $\D_1$ in two steps: in Step I, we shall remove the elements of $B$, and then in Step II, we shall remove the elements of $C$.
	
    \smallskip
	\noindent 
	\textbf{Step I:}
	Let $x\in B$. Then, $|x_i|+|x_j|+|x_k|+|x_l|=r-1$, where $\{i,j,k,l\}\subseteq [n]$, $x_l\geq 2$, $i<j<k<l$, and $x_s=0$ for all $s\notin \{i,j,k,l\}$. If $|x_t|=0$ for some $t\in \{i,j,k,l\}$, then $\sum_{s\in\{i,j,k,l\}\setminus\{t\}}|x_s|=r-1$, which contradicts Lemma \ref{sumofthree} $(ii)$. Therefore, $|x_t|\geq 1$ for all $t\in \{i,j,k,l\}$.  
	
	Consider the vertex $\lambda^{[x;ijkl]}\in V(\Z^n)$. Clearly,  $d(x,\lambda^{[x;ijkl]})=4$, and $|x_s-\lambda_s|=1$ for all $s\in\{i,j,k,l\}$. Since $i<j<k<l$ and $\lambda^{[x;ijkl]}_l>0$, and $\lambda^{[x;ijkl]}_t=0$ for $t>l$, it follows that $\lambda^{[x;ijkl]}\succ {\bf{0}}$. On the other hand $|\lambda^{[x;ijkl]}_s|<|x_s|$ for all $s\in \{i,j,k,l\}$ implies that $\lambda^{[x;ijkl]}\in \D_1$. 
	
	\begin{claim}\label{claim:fourvertex}
		$\n[x, \D_1] \subseteq \n[\lambda^{[x;ijkl]}, \D_1]$.
	\end{claim}
	\noindent {\em Proof of Claim \ref{claim:fourvertex}:}
	Since $r \geq 5$, $x\in\n[\lambda^{[x;ijkl]}, \D_1]$. Let $y\in\n[x,\D_1]$ with $x\neq y$. Suppose $y\notin \n[\lambda^{[x;ijkl]},\D_1]$.
	Then $d(y,\lambda^{[x;ijkl]})>r$. From Lemma \ref{fourvertex} $(i)$, $|y_s|<|x_s|$ for at most one value of $s$ in $\{i,j,k,l\}$, otherwise $d(y,\lambda^{[x;ijkl]})\leq r$, which is a contadiction. We have the following cases:
	
	\smallskip 
	
	\noindent\textbf{Case 1.1:} There exists $p \in \{i, j, k, l\}$ such that $|y_p| < |x_p|$.

	Then  $|y_s|\geq |x_s|$  for all $s\in \{ i, j,k,l\} \setminus \{p\}$. If  $sgn(y_sx_s)=-1$ for some $s \in \{i, j, k, l\} \setminus \{p\}$, then from Lemma \ref{fourvertex} $(ii)$, we get $d(y,\lambda^{[x;ijkl]})\leq r$,  which is a contradiction. Hence $sgn(y_sx_s)=1$ for all $s \in \{i, j, k, l\} \setminus \{p\}$. Let $|y_s|=|x_s|+a_s$, where $a_s\geq 0$ for $s \in \{i, j, k, l\} \setminus \{p\}$.  Since $\sum_{s}|y_s|\leq r$, we see that $\sum_{s\notin \{i, j,k,l\} \setminus \{p\}}|y_s|\leq r-(\sum_{s \in \{i, j, k, l \} \setminus \{p\}}|x_s|+a_s)$. Therefore,
	\begin{align*}
		d(y,\lambda^{[x;ijkl]})&=|y_p-\lambda^{[x;ijkl]}_p|+ \sum_{s \in \{i, j, k, l\}\setminus\{p\}}|y_s-\lambda^{[x;ijkl]}_s|+\sum_{s\notin\{i,j,k,l\}}|y_s-\lambda^{[x;ijkl]}_s| \\
		&\leq  |y_p-x_p|-1+ \sum_{s \in \{i, j, k, l\}\setminus\{p\}}(|y_s-x_s|+1 )+\sum_{s\notin\{i,j,k,l\}}|y_s| \\
		&\leq|x_p|+2+\sum_{s \in \{i, j, k, l\}\setminus \{p\}} a_s+\sum_{s\notin\{i, j,k,l\} \setminus \{p\}}|y_s| \\
		&\leq |x_p|+2 + \sum_{s \in \{i, j, k, l\}\setminus \{p\}} a_s+r- \sum_{s \in \{i, j, k, l\}\setminus \{p\}} (|x_s|+a_s) \\
		&= r+2+|x_p|-\sum_{s \in \{i, j, k, l\}\setminus \{p\}} |x_s|.
	\end{align*}

	If $|x_p|-\sum_{s \in \{i, j, k, l\}\setminus \{p\}} |x_s|\leq -2$, then $d(y,\lambda^{[x;ijkl]})\leq r$, thereby implying that $y \in \n[\lambda^{[x;ijkl]}, \D_1]$. So, assume that $|x_p|-\sum_{s \in \{i, j, k, l\}\setminus \{p\}} |x_s|\geq -1$. On the other hand if $|x_p|-\sum_{s \in \{i, j, k, l\}\setminus \{p\}} |x_s|\geq 1$, then using the fact that $|x_p|+\sum_{s \in \{i, j, k, l\}\setminus \{p\}} |x_s| =r-1$, we get that $|x_p|\geq \f$, which is not possible. Thus, we have only two possibilities, either $|x_p|=\sum_{s \in \{i, j, k, l\}\setminus \{p\}} |x_s|$ or $|x_p|+1=\sum_{s \in \{i, j, k, l\}\setminus \{p\}} |x_s|$.
	
	Now, if $|x_p|=\sum_{s \in \{i, j, k, l\}\setminus \{p\}} |x_s|$, then using equation $|x_p|+\sum_{s \in \{i, j, k, l\}\setminus \{p\}} |x_s| =r-1$, we get that $|x_p|=\frac{r-1}{2}$. This implies that $r$ must be odd and $|x_p|=\f$, which is not possible. 
	Therefore, we have $|x_p|+1=\sum_{s \in \{i, j, k, l\}\setminus \{p\}}$. In this case, $|x_p|=\frac{r-2}{2}$, which implies that $r$ is even. Since $r\geq 10$, there exists a $t\in \{i, j, k, l\}\setminus \{p\}$ such that $|x_t|\geq 2$. This implies that $|x_p|+|x_t|>\frac{r}{2}$. This is a contradiction.
	
	Thus we conclude that $d(y, \lambda^{[x;ijkl]}) \leq r$, and therefore $y \in \n[\lambda^{[x;ijkl]}, \D_1]$.

    \smallskip
	\noindent 
	\textbf{Case 1.2} $|y_s|\geq |x_s|$ for all $s\in\{i,j,k,l\}$.

	From Lemma \ref{fourvertex} $(iii)$, $sgn(y_sx_s)=-1$ is possible for at most one value of $s$ in $\{i,j,k,l\}$. 
	If $|y_s|\geq |x_s|$, and $sgn(y_sx_s)=1$ for all $s\in \{i,j,k, l\}$, then from $|y_i|+|y_j|+|y_k|+|y_l|\leq r-1$, it follows that $x_s=y_s$ for all $s\in \{i,j,k,l\}$. Consequently, $\sum_{s\notin\{i,j,k,l\}}|y_s|\leq 1$ implies that $d(y,\lambda^{[x;ijkl]})\leq 5<r$, which is a contradiction.  
	
	Thus, there exists a $q\in\{i,j,k,l\}$ such that $sgn(x_qy_q)=-1$. Then $sgn(x_sy_s)=1$ for $s\in\{i,j,k,l\}\setminus\{q\}$. Since $|y_i|+|y_j|+|y_k|+|y_l|\leq r-1$, it follows that $y_q=-x_q$ and $y_s=x_s$ for $s\in\{i,j,k,l\}\setminus\{q\}$. Then, we have $\sum_{s\notin\{i,j,k,l\}}|y_s|\leq 1$. Therefore, 
	\begin{align}
		d(y,\lambda^{[x;ijkl]})&=|y_q-x_q|-1+\sum_{s\in\{i,j,k,l\}\setminus\{q\}}(|y_s-x_s|+1)+\sum_{s\notin\{i,j,k,l\}}|y_s| \notag
		\\
		&\leq |y_q|+|x_q|+3. \label{eq2}
	\end{align}
	
	If $r$ is odd then $|x_q|=|y_q|\leq \f-1$. Thus, from Equation (\ref{eq2}),
	$d(y,\lambda^{[x;ijkl]})\leq \f-1+\f-1+3=r$, which is a contradiction.  Therefore, $r$ must be even. 
	Further, if $|x_q|\leq \frac{r}{2}-2$, then  $d(y,\lambda^{[x;ijkl]})\leq \frac{r}{2}-1+\frac{r}{2}-2+3=r$, which is a contradiction. From the bound  $|x_q|<\frac{r}{2}$, we must have $|x_q|=\frac{r}{2}-1$. Since $r\geq 10$, there exists a $t\in \{i, j, k, l\}\setminus \{q\}$ such that $|x_t|\geq 2$. This implies that $|x_q|+|x_t|>\frac{r}{2}$, which is a contradiction.
	
	Thus we conclude that $d(y, \lambda^{[x;ijkl]}) \leq r$, and therefore $y \in \n[\lambda^{[x;ijkl]}, \D_1]$.  Hence, $\n[x,\D_1]\subseteq \n[\lambda^{[x;ijkl]},\D_1]$. This proves Claim \ref{claim:fourvertex}.

	From Proposition \ref{flag}, $\D_1\simeq \mathrm{Ind}_{\D_1}(V(\D_1)\setminus \{x\})$. Now, applying the same arguments for each $b\in B$, we find that $\D_1\simeq \mathrm{Ind}_{\D_1}(V(\D_1)\setminus B)$. 
	Let $\D_2=\mathrm{Ind}_{\D_1}(V(\D_1)\setminus B)$. Then $\D_1\simeq \D_2$.

    \smallskip
	\noindent \textbf{Step II:}
	Let $x\in C\subseteq \D_2$ be such that $|x_i|+|x_j|+|x_k|+|x_l|=r-1$  where $\{i,j,k,l\}\subseteq [n]$, $i<j<k<l$, $x_l=1$, and $x_s=0$ for all $s\notin \{i,j,k,l\}$. If $|x_t|=0$ for some $t\in \{i,j,k,l\}$, then $\sum_{s\in\{i,j,k,l\}\setminus\{t\}}|x_s|=r-1$, which contradicts Lemma \ref{sumoffour} $(ii)$. Therefore, $|x_t|\geq 1$ for all $t\in \{i,j,k,l\}$.
	
	Consider $\lambda^{[x;ijk]}\in V(\Z^n)$. Clearly, $d(x,\lambda^{[x;ijk]})=3$, and $|x_s-\lambda^{[x;ijk]}_s|=1$ for $s\in\{i,j,k\}$. Since  $x_l=1$, and $x_s=0$ for $s>l$, it follows that $\lambda^{[x;ijk]}_l=1$ and $\lambda^{[x;ijk]}_s=0$ for $s>l$. Thus, $\lambda\succ{\bf{0}}$. Moreover, since $|\lambda^{[x;ijk]}_s|<|x_s|$ for all $s\in\{i,j,k\}$, and $\lambda^{[x;ijk]}_s=p_s$ for $s\notin\{i,j,k\}$, we have $\lambda^{[x;ijk]}\in\D_2$.
    
	\vspace{-0.5 cm}
    
	\begin{claim}\label{claim:fourvertex2}
		$\n[x, \D_2] \subseteq \n[\lambda^{[x;ijk]}, \D_2]$.
	\end{claim}
    
    \vspace{-0.5 cm}
    
	\noindent {\em Proof of Claim \ref{claim:fourvertex2}:}
	Since $r > 3$, $x\in\n[\lambda^{[x;ijk]}, \D_2]$. Let $y\in\n[x,\D_2]$ with $x\neq y$. Suppose $y\notin \n[\lambda^{[x;ijk]},\D_2]$. Then $d(y,\lambda^{[x;ijk]})>r$. From Lemma \ref{threevertex} $(i)$, $|y_s|<|x_s|$ for at most one value of $s$ in  $\{i,j,k\}$, otherwise $d(y,\lambda^{[x;ijk]})\leq r$, which is a contadiction. We have the following cases:

    \smallskip
	\noindent \textbf{Case 2.1:} There exists $p \in \{i, j, k\}$ such that $|y_p| < |x_p|$.

	Then  $|y_s|\geq |x_s|$  for all $s\in \{i ,j,k\}\setminus \{p\}$. If  $sgn(y_sx_s)=-1$ for some $s \in \{i, j, k\} \setminus \{p\}$, then from Lemma \ref{threevertex} (ii), we get $d(y,\lambda^{[x;ijk]})\leq r$,  which is a contradiction. Hence $sgn(y_sx_s)=1$ for all $s \in \{i, j, k\} \setminus \{p\}$. Let $|y_s|=|x_s|+b_s$, where $b_s\geq 0$ for $s \in \{i, j, k\} \setminus \{p\}$.  Since $\sum_{s}|y_s|\leq r$, we see that $\sum_{s\notin \{i, j,k\} \setminus \{p\}}|y_s|\leq r-(\sum_{s \in \{i, j, k \} \setminus \{p\}}|x_s|+b_s)$.
	
	\begin{align*}
    \begin{split}
		d(y,\lambda^{[x;ijk]})&=|y_p-\lambda^{[x;ijk]}_p|+ \sum_{s \in \{i, j, k\}\setminus\{p\}}|y_s-\lambda^{[x;ijk]}_s|+\sum_{s\notin\{i,j,k\}}|y_s-\lambda^{[x;ijk]}_s| \\
		&\leq  |y_p-x_p|-1+ \sum_{s \in \{i, j, k\}\setminus\{p\}}(|y_s-x_s|+1 )+\sum_{s\notin\{i,j,k\}}|y_s| +1 \\
		&\leq  |x_p|+2+ \sum_{s \in \{i, j, k\}\setminus \{p\}} b_s+\sum_{s\notin\{i,j,k\}}|y_s| +|y_p| \\
		&\leq|x_p|+2+\sum_{s \in \{i, j, k\}\setminus \{p\}} b_s+\sum_{s\notin\{i, j,k\} \setminus \{p\}}|y_s| \\
		&\leq |x_p|+2 + \sum_{s \in \{i, j, k\}\setminus \{p\}} b_s+r- \sum_{s \in \{i, j, k\}\setminus \{p\}} (|x_s|+b_s) \\
		&= r+2+|x_p|-\sum_{s \in \{i, j, k\}\setminus \{p\}} |x_s|.
         	\end{split}
	\end{align*}
	
	If $|x_p|-\sum_{s \in \{i, j, k\}\setminus \{p\}} |x_s|\leq -2$, then $d(y,\lambda^{[x;ijk]})\leq r$, which is a contradiction. If $|x_p|-\sum_{s \in \{i, j, k\}\setminus \{p\}} |x_s|\geq 1$, then using the fact that $|x_p|+\sum_{s \in \{i, j, k\}\setminus \{p\}} |x_s|=r-2$, we find that $|x_p|\geq \f$, which is not possible. Thus, we have only two possibilities: either $|x_p|=\sum_{s \in \{i, j, k\}\setminus \{p\}} |x_s|$ or $|x_p|+1=\sum_{s \in \{i, j, k\}\setminus \{p\}} |x_s|$. If $|x_p|=\sum_{s \in \{i, j, k\}\setminus \{p\}} |x_s|$, then from the equation, $|x_p|+\sum_{s \in \{i, j, k\}\setminus \{p\}} |x_s|=r-2$, we get $|x_p|=\frac{r-2}{2}$. This implies that $r$ is even. Since $r\geq 10$, there exists $t\in \{i,j,k\}\setminus \{p\}$ such that $|x_t|\geq 2$. Then $|x_p|+|x_t|>\frac{r}{2}$, which is a contradiction. 
	
	Now, we assume that $|x_p|+1=\sum_{s \in \{i, j, k\}\setminus \{p\}} |x_s|$. Then $|x_p|=\frac{r-3}{2}$. This implies that $r$ is odd and $|x_p|=\f-1$. Since $|x_l|=1$, we get $\sum_{s \in \{i, j, k\}\setminus \{p\}} |x_s|=\f$. If $|x_t|\geq 3$ for some $t\in \{i, j, k\}\setminus \{p\}$, then $|x_p|+|x_t|>\cil$, which is a contradiction. Therefore,  $\sum_{s \in \{i, j, k\}\setminus \{p\}} |x_s|\leq 4$. This implies that $\f\leq 4$ and thus, $r\leq 9$, which is a contradiction. Thus we conclude that $d(y, \lambda^{[x;ijk]}) \leq r$, and therefore $y \in \n[\lambda^{[x;ijk]}, \D_2]$.

    \smallskip
	\noindent \textbf{Case 2.2:} $|y_s|\geq |x_s|$ for all $s\in\{i,j,k\}$.

	From Lemma \ref{threevertex} $(iii)$, $sgn(y_sx_s)=-1$ is possible for at most one value of $s$ in $\{i,j,k\}$. 
	If $|y_s|\geq |x_s|$, and $sgn(y_sx_s)=1$ for all $s\in \{i,j,k\}$, then since from Lemma \ref{sumoffour} $(iii)$,   $|y_i|+|y_j|+|y_k|\leq r-2$, it follows that $x_s=y_s$ for all $s\in \{i,j,k\}$. Consequently, $\sum_{s\notin\{i,j,k\}}|y_s|\leq 2$ implying that $d(y,\lambda^{[x;ijk]})\leq 6$, which contradicts our assumption that $y\notin \n[\lambda^{[x;ijk]},\D_2]$.  
	
	Thus, there exists a $q\in\{i,j,k\}$ such that $sgn(x_qy_q)=-1$. Then $sgn(x_sy_s)=1$ for $s\in\{i,j,k\}\setminus\{q\}$. Since $|y_i|+|y_j|+|y_k|\leq r-2$, it follows that $y_q=-x_q$, $y_s=x_s$, for $s\in\{i,j,k\}\setminus\{q\}$. Then, we have $\sum_{s\notin\{i,j,k\}}|y_s|\leq 2$. Therefore, 
	\begin{align}
		d(y,\lambda^{[x;ijk]})&=|y_q-x_q|-1+\sum_{s\in\{i,j,k\}\setminus\{q\}}(|y_s-x_s|+1)+\sum_{s\notin\{i,j,k\}}|y_s-x_s| \notag
		\\
		&\leq |y_q|+|x_q|+1+\sum_{s\notin\{i,j,k\}}|y_s|+\sum_{s\notin\{i,j,k\}}|x_s| \notag
		\\
		&\leq |y_q|+|x_q|+4 = 2|x_q|+4. \label{eq3}
	\end{align}
	
	If $|x_q|  \leq \f-2 $, then  $d(y,\lambda^{[x;ijk]}) \leq r$, a contradiction to our assumption. So   $|x_q|=\f-1$ (because $|x_q|< \f$). 
	
	Suppose $r$ is odd. Since $|x_l|=1$,  $\sum_{s \in \{i, j, k\}\setminus \{q\}} |x_s|=\f$. If $|x_t|\geq 3$ for any $t\in \{i, j, k\}\setminus \{q\}$, then $|x_q|+|x_t|>\cil$, which is a contradiction. Therefore,  $\sum_{s \in \{i, j, k\}\setminus \{q\}} |x_s|\leq 4$. This implies that $\f\leq 4$ and thus, $r\leq 9$, a contradiction. 
	
	Suppose $r$ is even. Since $r\geq 10$, there exists a $t\in\{i,j,k\}\setminus \{q\}$ such that $|x_t|\geq 2$. This implies that $|x_q|+|x_t|>\frac{r}{2}$. This is a contradiction.
		
	Thus we conclude that $d(y, \lambda^{[x;ijk]}) \leq r$, and therefore $y \in \n[\lambda^{[x;ijk]}, \D_2]$. 
	Hence, $\n[x,\D_1]\subseteq \n[\lambda^{[x;ijk]},\D_1]$. This proves Claim \ref{claim:fourvertex2}.

	From Proposition \ref{flag}, $\D_2\simeq \D_2 \setminus x$. Now, applying the same arguments for each $c\in C$, we find that $\D_2\simeq \mathrm{Ind}_{\D_2}(V(\D_2)\setminus C)$. Moreover, since $A=B\cup C$ and $\D\simeq \D_1\simeq \mathrm{Ind}_{\D_1}(V(\D_1)\setminus B)= \D_2\simeq \mathrm{Ind}_{\D_2}(V(\D_2)\setminus C)$, it follows that $\D\simeq \mathrm{Ind}_{\D}(V(\D)\setminus A)$. This completes the proof.
\end{proof}

\begin{theorem}\label{maintheorem}
	For $2\leq n\leq 5$ and $r \geq n$,  $\vr{\mathbb{Z}^n}{r}$ is contractible. 
\end{theorem}
\begin{proof}
	%Fix an $m > 0$. Let $\G_m^n$ denote the induced  subgraph $\Z^n[\{0, \ldots, m\}^n]$. Let $\Delta_m^{n, r} = \vr{\G_m^n}{r}$. 
	% We prove that $\vr{\Z^n}{r}$ is contractible for $r\geq n$ and $2\leq n\leq 5$. To prove this, it suffices to show that for each positive integer $m$, the complex $\Delta_m^{n, r} = \vr{\G_m^n}{r}$ is contractible, where $r \geq n$.
	%For $1 \leq \alpha \leq \mathrm{Card}(V(\G_m^n))$,  let $\mathcal{H}_m^{n, \alpha}$ denote the subset of $V(\G_m^n)$ after removing the first $\alpha$ elements (with respect to order $\prec$) from 
	%$V(\G_m^n)$. Let $\delta= (\delta_1 , \ldots, \delta_n)$ be the least element of $\mathcal{H}_m^{n, \alpha}$.  Let $Y_m^{n, \alpha}$ be the induced subgraph of $\Z^n$ on the vertex set 
	%$ \{x-\delta: x \in \mathcal{H}_m^{n, \alpha}\}$.  Clearly, ${\bf{0}} \in V(Y_m^{n, \alpha})$ and it is the smallest element of $V(Y_m^{n, \alpha})$.
	%
	%Let ${\bf{0}}$ denote the vertex  $(0, \ldots, 0) \in V(\Z^n)$. Let $\Gamma_n^{\alpha,r}=\lk({\bf{0}},\vr{Y_m^{n, \alpha}}{r})$. 
	%
	%For any  $r \geq 1$, we have 
	%$$V(\Gamma_{n}^{\alpha,r})\subseteq \underbrace{([-r,r]\cap \Z)\times\cdots\times([-r,r]\cap\Z)}_{(n-1)\text{-times}}\times([0,r]\cap\Z).$$
	%
	%Now, we shall show that $\Gamma_n^{\alpha,r}$ is contractible for $2\leq n\leq 5$, and $r\geq n$.
	%
	Recalling the discussion at the beginning of this section, to prove that  $\vr{\mathbb{Z}^n}{r}$ is contractible, it is sufficient to show that $\Gamma_{n}^{\alpha,r}$ is contractible for all $ 1 \leq \alpha \leq \mathrm{Card}(V(\G_{m}^n))-1$. We provide a proof for each particular value of $n$ and considering $\alpha$ arbitrary.
	
	\smallskip 
	
	\noindent\textbf{Case (i):} $n=2.$
	
	From Lemma \ref{r/2}, $\Gamma_{2}^{\alpha,r}$ is homotopy equivalent to the induced subcomplex, say $\D$, of $\Gamma_{2}^{\alpha,r}$ on the vertex set $\{x\in V(\Gamma_{2}^{\alpha,r}): |x_1|, |x_2| \leq \f \}$. First,  suppose $\D$ contains the vertex $e = (0,1)$. Let $y \in V(\Delta)$. Since $y\succ (0, 0)$, $y_2\geq 0$. Moreover, if $y_2>0$, then
	$d(y,e)=|y_1|+|y_2-1|\leq |y_1|+|y_2|-1\leq \f+\f-1\leq r-1$, and if $y_2=0$, then $d(y,e)\leq |y_1|+1\leq \f+1\leq r$. Therefore, $\D$ is a cone with apex $e$, and hence contractible. Therefore, $\Gamma_{2}^{\alpha,r}$ is contractible.
	
	%Now, we see that for each $y\in V(\D)$, since $y\succ (0,0)$, $y_2\geq 0$. If $y_2>0$, then
	% $d(y,\lambda)=|y_1-\lambda_1|+|y_2-\lambda_2|\leq |y_1|+|y_2|-1\leq \f+\f-1\leq r-1$. If $y_2=0$, then $d(y,\lambda)\leq |y_1|+1\leq \f+1\leq r$. Thus, for any vertex $v\in\D$, $\n[v,\D]\subseteq \n[\lambda,\D]$. Thus, from Proposition \ref{flag}, $\D$ is homotopy equivalent to the induced subcomplex of $\D$ obtained after removing all the vertices except $\lambda$. This implies that $\D$ is contractible and thus $\Gamma_{n}^{\alpha,r}$ is contractible.
	
	Now, consider the case when $(0,1)\notin\D$.  Then for every vertex $v\in \D$, $v_2=0$.  If $V(\Delta)= \{(0, 0)\}$, then clearly $\Delta$ is contractible. So assume that $\mathrm{Card}(V(\Delta)) \geq 2$. Clearly $(1, 0) \in \D$.  Since $d(y,(1, 0))\leq r$ for all $y\in V(\D)$. Therefore, $\D$ is a cone with apex at $(1, 0)$, and hence contractible. This completes the proof for $n = 2$.
	
	%If there does not exist any vertex $u\in \D$ such that $u_2>0$, then for every vertex $v\in \D$, $v_2=0$. 
	%Now, there must be a vertex, say $\lambda'\in \D$, such that $\lambda_1=1$, and $\lambda_2=0$. In this case, clearly, $d(y,\lambda')\leq r$ for all $y\in V(\D)$. Hence, using Proposition \ref{flag}, $\D$ is contractible and thus $\Gamma_{n}^{\alpha,r}$ is contractible.

	\noindent\textbf{Case (ii):} $n = 3$.  
	
	From Lemma \ref{xi+xj}, $\Gamma_3^{\alpha,r}$ is homotopy equivalent to the induced subcomplex, say, $X$ on the vertex set $\{x \in \Gamma_3^{\alpha, r} : |x_i| \leq \f \ \text{and} \ |x_j| + |x_k| \leq \cil \ \ \text{for all } \ i,j,k\in [3], j\neq k\}$.
	
	If for every $y \in V(X)$, $y_3 = 0$, then $X$ is is an induced subcomplex on the vertex set $\{x \in \Gamma_3^{\alpha, r} : |x_1|, |x_2| \leq \f, |x_1| + |x_2| \leq \cil \ \text{and} \ x_3 = 0\}$. Now, by proceeding in the similar way as in Case $(i)$ above, we conclude that $X$ is contractible.

%     {\red homotopy equivalent to an induced subcomplex $X'$ of $\Gamma_2^{\beta,r}$  and $V(X')=\{z\in V(\Gamma_{3}^{\alpha,r}): |z_1|, |z_2| \leq \f  \text{and} z_3=0\}$.
% satisfying the following for each $z\in X'$: $|z_i| \leq \f \ \text{and} \ |z_j| + |z_k| \leq \cil \ \ \text{for all } \ i,j,k\in [2], j\neq k\}$. 
%     Thus $X'$ is contractible from Case $(i)$.} Therefore $X$ is contractible and so is $\Gamma_3^{\alpha,r}$. 

        So, we assume that there exists an element in $X$ whose third coordinate is non-zero. Then clearly $e':= (0, 0, 1)\in X$. Let $x \in V(X)$. If $x_3 \geq 1$, then 
	$d(x,e') = |x_1| + |x_2| + |x_3 - 1| = |x_1| + |x_2| + |x_3| - 1 \leq r - 1.$
	If $x_3 = 0$, then $|x_1| + |x_2| \leq \cil \leq r - 1$, and hence $d(x, e') = |x_1| + |x_2| + 1 \leq r.$ Thus, $X$ is a cone with apex $e'$, hence contractible. Therefore, $\Gamma_3^{\alpha,r}$ is contractible. This completes the proof for $n = 3$. 
	
	\smallskip
	
	\noindent\textbf{Case (iii):}  $n = 4$.

	From Lemma \ref{sumofthree},  $\Gamma_4^{\alpha,r}$ is homotopy equivalent to the induced subcomplex, say, $\mathcal{K}$ of $\Gamma_4^{\alpha,r}$,  where every vertex  $x \in V(\mathcal{K})$ satisfies the following: (i) $|x_i|< \f, |x_j|+|x_k|\leq \cil$ for all $i, j, k \in [4], j \neq k$  and (ii) $|x_i| + |x_j| + |x_k| < r-1$ for all  $\{i, j, k \}\subseteq [4]$.
	
	If for every $y \in V(\mathcal{K})$, $y_4 = 0$, then by prooceding in the similar way as in Case $(ii)$ above, we conclude that $\mathcal{K}$ is contractible.
    
    % \samir{isomorphic to $\Gamma_{3}^{\beta, r}$ for some $\beta$, and therefore it is contractible from Case $(ii)$}.  
    Suppose that an element whose fourth coordinate is non-zero exists in $\mathcal{K}$. Then we consider the element $\gamma = (0, 0, 0, 1)$. Clearly, $\gamma\succ (0,0,0,0)$, and $\gamma\in V(\mathcal{K})$. 
	Let $z \in V(\mathcal{K})$. If $z_4 \geq 1$, then $d(z, \gamma) = |z_1| + |z_2| + |z_3| + |z_4| - 1 \leq r - 1.$ If $z_4 = 0$, then $|z_1| + |z_2| + |z_3| \leq r - 2$, and hence $d(z, \gamma) = |z_1| + |z_2| + |z_3| + 1 \leq r - 1.$ Thus, $\mathcal{K}$ is a cone with apex $\gamma$, and hence contractible. Therefore, $\Gamma_4^{\alpha,r}$ is contractible. 
	
	\smallskip
	
	\noindent\textbf{Case (iv):}  $n = 5$.
	
	From Lemma \ref{sumoffour}, $\Gamma_5^{\alpha,r}$ is homotopy equivalent to a subcomplex, say, $\mathcal{L}$ of $\Gamma_5^{\alpha,r}$, where every vertex $x\in V(\mathcal{L})$ satisfies (i) $|x_i|< \f, |x_j|+|x_k|\leq \cil$ for all $i, j, k \in [5], j \neq k$ (ii) $|x_i| + |x_j| + |x_k| < r-1$ for all $\{i, j, k\}\subseteq [5]$, and (iii) $|x_i| + |x_j| + |x_k| + |x_l|\leq r-1$ for all  $\{i, j, k, l\} \subseteq [5]$. 
	
	If for every $w \in V(\mathcal{L})$, $w_5 = 0$, then by prooceding in the similar way as in Case $(iii)$ above, we conclude that $\mathcal{L}$ is contractible.  
       
    % \samir{isomorphic to $\Gamma_{4}^{\beta, r}$ for some $\beta$, and therefore it is contractible from Case $(iii)$}. 
    Suppose that an element whose fifth coordinate is non-zero exists in $\mathcal{L}$. Then we consider the element $\beta = (0, 0, 0, 0,1)$. Clearly, $\beta\succ (0,0,0,0,0)$, and $\beta\in V(\mathcal{L})$. 
	Let $v \in V(\mathcal{L})$. If $v_5 \geq 1$, then $d(v, \beta) = |v_1| + |v_2| + |v_3| + |v_4| +|v_5| - 1 \leq r - 1.$ If $v_5 = 0$, then $|v_1| + |v_2| + |v_3| +|v_4| \leq r - 1$, and hence $d(v, \beta) = |v_1| + |v_2| + |v_3| + |v_4|+ 1 \leq r.$ Thus, $\mathcal{L}$ is a cone with apex $\beta$, and hence contractible. Therefore, $\Gamma_5^{\alpha,r}$ is contractible. 
\end{proof}

\begin{theorem}\label{thm:n=6(Inside)}
	$\vr{\mathbb{Z}^6}{r}$ is contractible for $r\geq 10$.
\end{theorem}
\begin{proof}
	It is sufficient to show that $\Gamma_{n}^{\alpha,r}$ is contractible for all $ 1 \leq \alpha \leq \mathrm{Card}(V(\G_{m}^n))-1$. From Lemma \ref{sumoffour2}, $\Gamma_6^{\alpha,r}$ is homotopy equivalent to an induced subcomplex $\D$ such that there exists no vertex  $u\in V(\D)$ with $|u_i|+|u_j|+|u_k|+|u_l|\geq r-1$ and $u_s=0$ for some  $\{i,j,k,l\}\subset [6]$ and $s\notin \{i,j,k,l\}$. 
	
	If for every $y \in V(\D)$, $y_6 = 0$, 
    then by prooceding in the similar way as in Case $(iv)$ of Theorem \ref{maintheorem} above, we conclude that $\mathcal{D}$ is contractible.
   
	Suppose there exists an element in $V(\D)$ with a non-zero sixth coordinate. Let $p = (0, 0, 0, 0, 0, 1)$. Clearly $p\succ (0,0,0,0,0,0)$ and $p\in  V(\D)$.
	
	First, suppose there exists an element in $V(\D)$ with a positive fifth coordinate. Then $q = (0, 0, 0, 0, 1, 0) \in  V(\D)$. Let $w = (0, 0, 0, 0, 1, 1)$. We see that $w\succ (0,0,0,0,0,0)$ and from Lemma \ref{sumoffour2}, $w\in V(\D)$.
	Let $x \in V(\D)$.  
	If $x_6 \geq 1$, then $d(x, w) = |x_1| + |x_2| + |x_3| + |x_4| + |x_5 - 1| + |x_6 - 1| \leq |x_1| + \cdots + |x_5| +1+ |x_6|-1 \leq r.$ If $x_6 = 0$ and $x_5 \neq 0$, then since $x\succ (0,0,0,0,0,0)$, $x_5 \geq 1$, so $d(x, w) = |x_1| + \cdots + |x_4| + |x_5 - 1| + 1 = |x_1| + \cdots + |x_4| + |x_5| \leq r.$ If $x_5 = x_6 = 0$, then since $|x_1| + \cdots + |x_4| \leq r - 2$ for $r\geq 10$, hence
	$d(x, w) = |x_1| + \cdots + |x_4| + 1 +1 \leq r$ for $r\geq 10$.  Thus, $\Delta$ is a cone with apex vertex $w$, and hence it is contractible. Therefore,  $\Gamma_6^{\alpha,r}$ is contractible.

	If there is no element in $V(\D)$ with a positive fifth coordinate, then for any $y\in V(\D)$, either $y_6\geq 1$ or $y_5=y_6=0$. Since we have $|y_1|+|y_2|+|y_3|+|y_4|\leq r-2$, we conclude that $d(y,p)\leq r$. Therefore, $N[x, \D] \subseteq N(p, \D)$.
	Thus, using Proposition \ref{flag}, we remove all the vertices in $\D$ except $p$. Hence, $\D$ is contractible, and thus $\Gamma_6^{\alpha,r}$ is contractible in $\vr{\Z^6}{r}$.
\end{proof}

\section{The complex  $\vr{\Z^n}{2}$} \label{sec:vr2}
In this section, we prove Theorems \ref{thm:r=2_Intro} and \ref{thm:simplyconnected_intro}. We first characterize the maximal simplices of the complex $\vr{\Z^n}{2}$, and then use discrete Morse theory to determine the homotopy type of these complexes.  We begin by defining a few notations that we will use throughout this section. 

  Recall that for a positive integer $n$, $[n] = \{1, 2, \ldots, n\}$.  Let  $[-n] = \{-1, \ldots, -n\}$ and $[n]^{\pm} = [n] \cup [-n]$.  For $ \{i_1, i_2, \ldots, i_k\}\subseteq [n]^{\pm}$ such that $|i_s| \neq |i_t|$ for all $1 \leq s \neq t \leq k$, we define 
$x^{i_1, \ldots, i_k} \in V(\Z^n)$ by 

$$
x^{i_1, \ldots, i_k}(j)  = \begin{cases}
	\ x(j) & \text{if} \  j \notin \{i_1, \ldots, i_k\},\\
	\ x(j)+1  & \text{if} \   j \in \{i_1, \ldots, i_k\},  \\
	\ x(j)-1  & \text{if} \   -j \in \{i_1, \ldots, i_k\}.\\
\end{cases}
$$

For $i \in [n]$ and $k \in \Z$, we define $x^{[i;k^+]} \in V(\Z^n)$ by 
$$
x^{[i;k^+]}(j) = \begin{cases} 
	x(j) \ \ \ \  \ \ \  \ \ \text{if} \ j \neq i \\
	x(i)+k  \ \ \ \  \text{if} \ j =i. 
\end{cases}
$$

%    $$
% x_k^i(j) = \begin{cases} 
	% 	x(j) \ \ \ \  \ \ \  \ \ \text{if} \ j \neq i \\
	% 	x(i)+k  \ \ \ \  \text{if} \ j =i. 
	% 	\end{cases}
% $$

%    Similarly, for $i,j \in [n]$ and $k \in \Z$, we let $x_k^{i,j} \in V(\Z^n)$  defined by 
% $$
% x_k^{i,j}(l) = \begin{cases} 
	% 	x(l) \ \ \ \  \ \ \  \ \ \text{if} \ l \neq i,j \\
	%        x(i)+k \ \ \ \  \text{if} \ l = i \\
	% 	x(j)+k  \ \ \ \  \text{if} \ l =j. 
	% 	\end{cases}
% $$

Recall that $\mathbb{Z}^n$ is a graph, where any two elements $x$ and $y$ are connected by an edge if and only if $d(x, y) = \sum_{i = 1}^n|x_i-y_i| = 1$. Define the \textit{open neighborhood} of a vertex $x$ in $\mathbb{Z}^n$ by 
$N(x, \mathbb{Z}^n) = \{y\in \mathbb{Z}^n : d(x,y)=1\}$, 
and the \textit{closed neighborhood} of $x$ by 
$N[x, \mathbb{Z}^n] = N(x, \mathbb{Z}^n) \cup \{x\}$.
For the simplicity of notation, we write $N(x)$ and $N[x]$ for the sets $N(x, \mathbb{Z}^n)$ and $N[x,\mathbb{Z}^n]$, respectively.

We first characterize the maximal simplices of $\vr{\Z^n}{2}$.  The idea of the proof of the following lemma is similar to \cite[Lemma 3.1]{Shukla2023}. 

\begin{lemma} \label{thm:maximal2}
	Let  $n \geq 3$  and $\tau $ be a maximal simplex of $\vr{\Z^n}{2}$. Then one of the following is true:
	\begin{itemize}
		\item[(i)] $\tau = N[x]$ for some $x \in V(\Z^n)$.
		\item[(ii)] $\tau = \{x, x^{i_0}, x^{j_0}, x^{i_0, j_0}\}$ for some $x \in V(\Z^n)$ and $i_0, j_0 \in [n]^{\pm}$.
		\item[(iii)] $\tau = \{x,x^{i_0,j_0},x^{j_0,k_0},x^{i_0,k_0}\}$ for some $x \in V(\Z^n)$ and $i_0, j_0, k_0 \in [n]^{\pm}$.
		%					\item[(iv)] $\sigma = \{v, v_{2}^l\} \cup  \{v^{l, j} : j \in [n]^{\pm} \setminus \{l, -l\}\}$  for some $v \in V(\Z^n)$ and $l \in [n]$.
		%					
		%					\item[(v)] $\sigma = \{v, v_{-2}^l\} \cup  \{v^{-l, j} : j \in [n]^{\pm} \setminus \{l, -l\}\}$  for some $v \in V(\Z^n)$ and $l \in [n]$.
		
	\end{itemize}
\end{lemma}
\begin{proof}
	We consider the following cases. 
	
	{\it \bf Case 1.} There exists a $y \in \tau$ such that $N(y) \cap \tau \neq  \emptyset$. 
	
	\begin{itemize}
		\item Suppose $|N(y) \cap \tau|=1$, and let $N(y) \cap \tau = \{x\}$. Since $x \in N(y)$, there exists  $l\in [n]^{\pm}$ such that 
		$y = x^{l}$.  We first show that  $N[x] \subseteq \tau$. If possible, let $s \in [n]^{\pm}$ such that $x^{s} \notin \tau$. Then %The maximality of  $\tau$ implies that
        there exists  $v\in\tau$ such that $d(v, x^{s}) \geq 3$. Since $x \in \tau$, we have $d(v, x) \leq 2$. Furthermore, $d(x^{s}, x^{t}) = 2$ for every $t \in [n]^{\pm} \setminus \{s\}$ implies that $v \neq x^t$. Therefore, $d(x, v) = 2$, and hence $v = x^{i, j}$ for some  $i, j \in [n]^{\pm}$, or there exists $k \in [n]$ such that  $v \in \{x^{[k; 2^+]}, x^{[k; -2^+]}\}$.
        
        Since $d(y, v) \leq 2$, for $v=x^{[k; 2^+]}$ we would have $l=k$ and for  $v=x^{[k; -2^+]}$ we would have $l=-k$. However, in both of these situations $v \in N(y)$, which is a contradiction as $N(y) \cap \tau = \{x\}$. On the other hand, if  $v = x^{i, j}$ for some $i, j \in [n]^{\pm}$, then $N(y) \cap \tau = \{x\}$ implies that $l \notin \{i, j\}$. Therefore, $d(v, y) = 3$, which contradicts the fact that $v,y\in\tau$.  Hence,  $N[x]\subseteq \tau$. 
		
		We now show that $\tau \subseteq N[x]$. Suppose, $w\in \tau\setminus N[x]$. Since $x, w \in \tau$, we have $d(x, w) = 2$. Therefore, either $w = x^{p, q}$ for some  $p, q \in [n]^{\pm}$, or there exists $k \in [n]$ such that $w \in \{x^{[k; 2^+]}, x^{[k; -2^+]}\}$.	 
		If  $w = x^{p, q}$, 
		then $d(x^{\alpha}, w) = 3$ for  $\alpha \in [n]^{\pm} \setminus \{p, q\}$, a contradiction as $x^{\alpha} \in \tau$. 
		If $w \in \{x^{[k; 2^+]}, x^{[k; -2^+]}\}$, then $d(x^{\beta}, w) =3$ for any $\beta\in [n]\setminus \{k\}$,  which is again a contradiction. Thus, we conclude that  $\tau = N[x]$ and it is of the type $(i)$.

		\item Let  $|N(y) \cap \tau| \geq 2$.

  Then there exists $i_0, j_0 \in [n]^{\pm}$ such that $y^{i_0}, y^{j_0} \in \tau.$ Thus $\{y, y^{i_0}, y^{j_0}\} \subseteq \tau$. Note that $|N(y^{i_0}) \cap \tau | \geq 1$, as $y\in N(y^{i_0}) \cap \tau$. If $|N(y^{i_0}) \cap \tau | =1 $, then it follows from the previous part that $\tau = N[y]$.  

Let $|N(y^{i_0}) \cap \tau | \geq 2$. Then there exists  $ w\in \tau \setminus \{y\}$ such that $w \in N( y^{i_0})$.  Then
		$w = (y^{i_0})^{ k}$ for some $k \in [n]^{\pm}$. If $k = -i_0$, then $w= y$. Hence $k \neq -i_0$. If $j_0 = -i_0$, then $d(y^{j_0}, w) = 3$, a contradiction. Hence $j_0 \neq \pm i_0$. Since  $d(y^{j_0}, w) \leq  2$, we get that  $k = j_0$.  Thus $w = y^{i_0, j_0}$ and therefore $\{y, y^{i_0}, y^{j_0}, y^{i_0, j_0}\} \subseteq \tau$.

      Suppose  there exists $u \in \tau \setminus \{y, y^{i_0}, y^{j_0}, y^{i_0, j_0}\}$. If $u \in N(y)$, then $u = y^{i}$ for some $i \in [n]^{\pm} \setminus \{i_0, j_0\}$. Here $d(u, y^{i_0, j_0}) = 3$, a contradiction. Hence $u \notin N( y)$, {\it i.e.}, $d(y, u) = 2$. If   $u \in \{y^{[l; 2^+]}, y^{[l; -2^+]}\}$ for some $l \in [n]$, then  $d(u, v) = 3$ for some $v \in \{y, y^{i_0}, y^{j_0}, y^{i_0, j_0}\}$. Hence $u = y^{j, k}$ for some $j, k \in [n]^{\pm}$. If $\{i_0, j_0\} \cap \{j, k\} = \emptyset$, then $d(y^{i_0,j_0}, u) = 4$, a contradiction. Hence $\{i_0, j_0\} \cap \{j, k\} \neq \emptyset$. Without loss of generality, we assume that $i_0 \in \{j, k\}$. In this case $d(u, y^{j_0}) = 3$, a contradiction. Thus $\tau = \{y, y^{i_0}, y^{j_0}, y^{i_0, j_0}\} $. Hence $\tau$ is of the type $(ii)$.
	\end{itemize}
    
	\noindent {\it \bf Case 2.} $N(y) \cap \tau = \emptyset$ for all $y \in \tau$.
    
    Let $y \in \tau$. Choose  $v \in \tau$ such that $v \neq y$. Since $N(y) \cap \tau = \emptyset$ and $d(y, v) \leq 2$, we have $d(y,v ) = 2$.

    %\vspace{.2cm}
	\begin{itemize}
		\item[(1)] Let $v= y^{i_0, j_0}$ for some $i_0, j_0 \in [n]^{\pm}$, where   $|i_0| \neq |j_0|$. 

Then, $\{y, y^{i_0, j_0}\} \subseteq \tau$. Since  $d(y^{i_0,t }, y) = 2 = d(y^{i_0, t}, y^{i_0, j_0})$ for every $t \in [n]^{\pm} \setminus \{\pm i_0, \pm j_0\}$, we see that $\{y, y^{i_0, j_0}, y^{i_0, t}\} \in\vr{\mathbb{Z}^n}{2} $. Thus $\{y, y^{i_0, j_0}\}$ is  not a maximal simplex. Let $u \in \tau \setminus \{y, y^{i_0, j_0}\}$. By the assumption of $N(u) \cap \tau = \emptyset$, we have $d(y, u) = 2$. Therefore, either $u \in \{y^{[l; 2^+]}, y^{[l; -2^+]}\}$ for some $l \in [n]$ or $u = y^{i, j}$ for some $ i, j \in [n]^{\pm}$. 
		
		\vspace{.25cm}
		
		\begin{itemize}
			\item[(1.a)] Let $u \in \{y^{[l; 2^+]}, y^{[l; -2^+]}\}$ for some $l \in [n]$. Without loss of generality, let  $u = y^{[l; 2^+]}$.  Then,  $\{y, y^{i_0, j_0}, y^{[l; 2^+]}\} \subseteq \tau$. If $l \notin \{i_0, j_0\}$, then $d(u, y^{i_0, j_0}) > 2$, and hence $l \in \{i_0, j_0\}$.  Without loss of generality, assume that $l = i_0$.  We show that $\tau =  \{y, y^{[i_0; 2^+]}\} \cup  \{y^{i_0, j} : j \in [n]^{\pm} \setminus \{\pm i_0\}\}$.  
            
            For every $j \in  [n]^{\pm} \setminus \{\pm i_0,\pm j_0\}$, each of $d(y, y^{i_0, j}), d(y^{[i_0; 2^+]}, y^{i_0, j} )$,  and $d(y^{i_0, j_0}, y^{i_0, j})$ is $2$. Therefore,  $\{y, y^{i_0, j_0}, y^{[i_0; 2^+]}\}$ is not a maximal simplex. Let $x \in \tau\setminus \{y, y^{i_0, j_0}, y^{[i_0; 2^+]}\} $. Since  $d(x, y) =  2=  d(x, y^{i_0, j_0} ) = d(x, y^{[i_0; 2^+]})$, we have $x = y^{i_0, j}$ for some  $j \in  [n]^{\pm} \setminus \{\pm i_0, j_0\}$. Furthermore, for two distinct integers $j_1, j_2 \in [n]^{\pm} \setminus \{\pm i_0, j_0\}$, we have $d(y^{i_0, j_1}, y^{i_0, j_2}) = 2$. Since $\tau$ is maximal,  we conlclude that $\{y, y^{[i_0; 2^+]}\} \cup  \{y^{i_0, j} : j \in [n]^{\pm} \setminus \{\pm i_0\}\} \subseteq \tau $. Now, for a vertex $z \in \{y , y^{[i_0; 2^+]}\} \cup  \{y^{i_0, j} : j \in [n]^{\pm} \setminus \{\pm i_0\}\} $, we have $d(y^{i_0}, z) = 1$, and hence there exists some $\tilde{z} \in \tau \setminus\{y,y^{[i_0; 2^+]}\} \cup  \{y^{i_0, j} : j \in [n]^{\pm} \setminus \{\pm i_0\}\}$. But the only choice for  $\tilde{z}$ is $y^{i_0}$. Since  $y^{i_0} \in N(y)$ and $N(y) \cap \tau = \emptyset$, we conclude that $y^{i_0} \notin \tau$. 
			This contradicts the fact that $\tau$ is a maximal simplex.

			Using a similar argument, if $u= y^{[l; -2^+]}$, then we get a contradiction.  Hence, this case is not possible.
			\item[(1.b)] Let $u= y^{i, j}$ for some  $ i, j \in [n]^{\pm}$.  If $\{i, j\} \cap \{i_0, j_0\} = \emptyset $, then $d(u, y^{i_0, j_0}) \geq 3$, a contradiction. Hence $\{i, j\} \cap \{i_0, j_0\} \neq \emptyset$. Without loss of generality, let $i = i_0$. Then $\{y, y^{i_0, j_0}, y^{i_0, j}\} \subseteq \tau.$  Since $N(y) \cap \tau = \emptyset$, we have $y^{i_0} \notin \tau$. Further, since $\tau$ is maximal, there exists  $z \in \tau$ such that $d(z, y^{i_0}) \geq 3$. Clearly $d(y, z) = 2$. Observe that   $z = y^{k,l}$ for some $k, l \in [n]^{\pm}$.

			Since $d(z, y^{i_0}) \geq 3, i_0 \notin \{k, l\}$. Using the fact that $d(z, y^{i_0, j_0}) = 2 = d(z, y^{i_0, j})$, we conclude that $\{k, l\} = \{j_0, j\}$. Thus $\{y, y^{i_0, j_0}, y^{i_0, j}, y^{j_0, j}\} \subseteq \tau.$  Suppose there exists a vertex $w \in \tau \setminus \{y, y^{i_0, j_0}, y^{i_0, j}, y^{j_0, j}\}$.  Then $N(y) \cap \tau = \emptyset $ implies that $d(y, w) = 2$ and therefore  $w = y^{s, t}$ for some $s, t \in [n]^{\pm}$. Since $d(w, y^{i_0, j_0}) = 2$, $\{i_0, j_0\} \cap \{s, t\} \neq \emptyset$. Further, $d(w, y^{i_0, j}) = 2$ implies  that $\{i_0, j\} \cap \{s, t\} \neq \emptyset$ and $d(w, y^{j_0, j}) = 2$ implies  that  $\{ j_0, j\} \cap \{s, t\} \neq \emptyset$, which is not possible. Hence $\tau = \{y, y^{i_0, j_0}, y^{i_0, j}, y^{j_0, j}\}$. Thus  $\tau$ is of the type $(iii)$.
			
		\end{itemize}
		
		\vspace{.35cm}
		
		\item[(2)] Let $v \in \{y^{[l; 2^+]}, y^{[l; -2^+]}\}$ for some $l \in [n]$. 
        
        Without loss of generality, let $v = y^{[l; 2^+]}$. Since $d(y^{l, i}, y^{[l; 2^+]}) = 2 = d(y^{l, i}, y)$ for every $i \in [n] \setminus \{l\}$,  $\tau$ is not a maximal simplex. Let $x \in \tau \setminus \{y, y^{[l; 2^+]}\}$. Since $d(x, y) = 2 = d(x, y^{[l; 2^+]})$, we  get $x = y^{l, j_0}$ for some $j_0\in [n]^{\pm} \setminus \{l, -l\}$.  Hence $\{y, y^{[l; 2^+]}, y^{l, j_0}\} \subset \tau$. Using the same argument as in $(1.a)$, we get a contradiction. Hence this case is not possible. 
	\end{itemize}
    \vspace{-0.4 cm}
\end{proof}

Fix an $m > 0$. Recall from Section \ref{sec:vr1}, that  $\mathcal{G}_m^n$ denote the induced  subgraph $\Z^n[\{0, \ldots, m\}^n]$ and   $\Delta_m^{n,2} = \vr{\mathcal{G}_m^n}{2}$. 
The following Lemma is a consequence of Lemma \ref{thm:maximal2}.

\begin{lemma} \label{thm:maximalsimplex_m_case} 
	Let  $n \geq 2$, and let $\tau $ be a maximal simplex of $\Delta_m^{n,2}$. Then one of the following is true:
	\begin{itemize}
		\item[(i)] $\tau = N[x] \cap V(\G_m^n)$ for some $x \in V(\Z^n)$.
		\item[(ii)] $\tau = \{x, x^{i_0}, x^{j_0}, x^{i_0, j_0}\} \cap V(\G_m^n)$ for some $x \in V(\Z^n)$ and $i_0, j_0 \in [n]^{\pm}$.
		\item[(ii)] $\tau = \{x,x^{i_0,j_0},x^{j_0,k_0},x^{i_0,k_0}\}\cap V(\G_m^n)$ for some $x \in V(\Z^n)$ and $i_0, j_0, k_0 \in [n]^{\pm}$.
		%					\item[(iv)] $\sigma = \{v, v_{2}^l\} \cup  \{v^{l, j} : j \in [n]^{\pm} \setminus \{l, -l\}\}$  for some $v \in V(\Z^n)$ and $l \in [n]$.
		%					
		%					\item[(v)] $\sigma = \{v, v_{-2}^l\} \cup  \{v^{-l, j} : j \in [n]^{\pm} \setminus \{l, -l\}\}$  for some $v \in V(\Z^n)$ and $l \in [n]$.
		
	\end{itemize}
\end{lemma}

We now give a brief description of Forman's discrete Morse theorey \cite{Forman1998}. For more detail, we refer to \cite{Kozlovbook}. 
\begin{definition}\cite[Definition 11.1]{Kozlovbook}
	A {\it partial matching} in a poset $P$ is a subset $\M$ of $P \times P$ such that
	\begin{itemize}
		\item $(a,b) \in \M$ implies $b \gg a$, {\it i.e. $a<b$ and $\not\exists \,c$ such that $a<c<b$}.
		\item Each element  in $P$ belongs to at most one element of $\M$.
	\end{itemize}
\end{definition}
If $\M$ is a  partial matching on a poset $P$,
then there exists  $A \subset P$ and an injective map $f: A
\rightarrow P\setminus A$ such that $f(x)\gg x$ for all $x \in A$.

\begin{definition}\label{d1}
	An {\it acyclic matching} is a partial matching  $\M$ on the poset $P$ such that there does not exist a cycle
	\begin{eqnarray*}
		f(x_1)  \gg x_1 \ll f( x_2) \gg x_2  \ll f( x_3) \gg x_3 \dots   f(x_t) \gg x_t  \ll f(x_1), t\geq 2.
	\end{eqnarray*}
\end{definition}
For an acyclic partial matching on $P$, those elements of $P$ that do not belong to the matching are called
{\it critical}.

\begin{theorem}\cite[Theorem 11.13]{Kozlovbook}\label{theorem:Morse_theory} (Main theorem of Discrete Morse Theory)
	
    Let $X$ be a simplicial complex and $\A$ be an acyclic matching on the face poset of $X$ 
	such that the empty set is not critical. 
	Then, $X$ is homotopy equivalent to a cell complex which has a $d$-dimensional cell for each $d$-dimensional
	critical face of $X$ together with an additional $0$-cell.
\end{theorem}

The following remark is an immediate consequence of Theorem \ref{theorem:Morse_theory}.

\begin{remark} \label{remark:wedge of spheres}
	If an acyclic matching on a face poset of a simplicial complex $\Delta$ has critical faces only in a fixed dimension $i$, then $\Delta$ is homotopy equivalent to a wedge of spheres of dimension $i$.
\end{remark}

Let $X$ be a simplicial complex with vertex set $V(X) = \{v_1, v_2, \ldots, v_n\}$. Assume that the vertices of $X$ are linearly ordered as $v_1 < v_2 < \cdots < v_n$.  Let $\mathcal{P}(X)$ denote the face poset of $X$. We define an acyclic matching $\mu^X$ on $\mathcal{P}(X)$ as follows:

Let $S_1^X = \{\sigma \in \cP(X): v_1 \notin \sigma \ \text{and} \ \sigma \cup \{v_1\} \in \cP(X)\}$. Define $$\mu_1^X : S_1^X \to \cP(X) \setminus S_1^X  \ \text{by} \  \mu_1^X(\sigma) = \sigma \cup \{v_1\}.$$
Then observe that $\mu_1^X$ is an acyclic matching on $\cP(X)$. Let $\T_1^X = \cP(X) \setminus (S_1^X \cup \mu_1(S_1^X))$. 
For  $2 \leq i \leq k ,$ define 
\begin{eqnarray}
	&&S_{i}^X = \{\sigma \in \T_{i-1}^X \ | \ v_i \ \notin \sigma  \ \text{and} \ \sigma \, \cup \{v_i \} \ \in \T_{i-1}^X \}, \label{am:1}\nonumber\\
	&&\mu_{i}^X : S_{i}^X  \rightarrow \T_{i-1}^X \setminus S_{i}^X ~\text{by} ~  \mu_{i}^X(\sigma) = \sigma \, \cup \{a_i\}~~\text{and}\label{am:2} \nonumber\\
	&&\T_i^X= \T_{i-1}^X \setminus (S_i^X \cup \mu_i^X(S_i^X)). \label{am:3} \nonumber
\end{eqnarray}

By the above construction, $S_{i}^X \cap S_j^X =\emptyset$ for all  $\ i \neq j$.
Let $S^X = \bigcup \limits_{i=1}^{k} S_{i}^X$. Define 
\begin{align}\label{equation:matching}
	&\mu^X : S^X \rightarrow \cP(X) \setminus S^X \ \text{by} \ \mu^X(\sigma) =\mu_{i}^X (\sigma),  
\end{align}
where $i$ is the unique element such that $\sigma \in S_{i}^X$.

From \cite[Proposition 3.2]{Goyal2021}, the matching $\mu^X$ defined in Equation (\ref{equation:matching}) is an acyclic matching.

Let $\mu^{\Delta_m^{n,2}}$ be the acyclic matching as defined in Equation (\ref{equation:matching}) with respect to the anti-lexicographic order $\prec$ on vertices of $\G_m^n$. In the rest of the section, we consider the matching $\mu^{\Delta_m^{n,2}}$ on $\cP(\Delta_m^{n,2})$, and for the convenience of notation, we denote the matching $\mu^{\Delta_m^{n,2}}$ simply by $\mu$. 
Moreover, $S_{i}^X,\T_{j}^X$, and $S^{X}$ will be denoted as $S_{i},\T_{j}$, and $S$ if the underlying simplicial complex $X$ is clear from the context. We now characterize the critical cells corresponding to the matching $\mu$.

\begin{proposition}\label{smallestvertex}
	
	Let $\sigma \in \Delta_m^{n,2}$ be a simplex. If there exists a vertex $x$ such that $x \prec y$ for all $y \in \sigma$, and $\sigma \cup \{x\} $ is a simplex, then  $\sigma$ is not a critical cell for the matching $\mu$. 
\end{proposition}
\begin{proof}
	Let $z$ be the smallest element such that $z \prec y$ for all $y \in \sigma$ and $\sigma \cup \{z\}$ is a simplex. Clearly, $z \notin \sigma$. Then $\sigma$ and $\sigma \cup \{z\} \in \T_{v}$ for all $v \prec z$. Therefore, by the definition of $\mu$, we get that $\mu(\sigma) = \mu_z(\sigma) = \sigma \cup \{z\}$. Hence, $\sigma$ is not a critical cell.
\end{proof}

\begin{lemma}\label{lemma:critical_1}
	The matching $\mu$ yields no critical cells of dimension $0$ and $1$ in $\cP(\D_m^{n,2})$ for $n\geq 3$.
\end{lemma}
\begin{proof}
	
	Let $v \in \D_m^{n,2}$ be a vertex. If there exists $i \in [n]^{-}$ such that $v^{i} \in \D_m^{n,2}$, then $v^{i} \prec v$ and $\{v,v^{i}\}$ is a simplex. This implies that $v$ is not a critical cell. 
	
	Now, suppose there is no $i \in [n]^{-}$ such that $v^{i} \in \D_m^{n,2}$. Then, $v = (0, \ldots, 0)$, and therefore, 
$v = \mu_{v}(\emptyset)$.   
   Thus, $v$ is not a critical cell. Hence, the matching $\mu$ yields no critical cells of dimension $0$ in $\D_m^{n,2}$.
	
	Let $\gamma \in \Delta_m^{n,2}$ be a $1$-simplex. Then $\gamma$ must be one of the following four types:  
	(i) $\gamma = \{v,v^{i}\}$ for some $v,v^{i} \in V(\G_m^n)$ and $i \in [n]$, 
	(ii) $\gamma = \{v,v^{i,j}\}$ for some $v,v^{i,j} \in V(\G_m^n)$ and $i,j \in [n]$,  
	(iii) $\gamma = \{v,v^{i,j}\}$ for some $v,v^{i,j} \in V(\G_m^n)$ with $i \in [n]$ and $j \in [n]^{-}$, and (iv) $\gamma = \{v,v^{[i;2^{+}]}\}$  for some $v \in V(\G_m^n)$ and $i \in [n]$.

    \vspace{.1cm}
	\noindent \textbf{Case (i)}: Let $\gamma = \{v,v^{i}\}$ for some $v,v^{i} \in V(\G_m^n)$ and $i \in [n]$. If there exists $l \in [n]^{-}$ such that $v^{l} \in V(\mathcal{G}_m^n)$, then $v^{l} \prec x$ for all $x \in \gamma$, and $\gamma \cup \{v^{l}\} \subseteq N[v]$. From Proposition \ref{smallestvertex}, $\gamma$ is not a critical cell. Suppose no such $l \in [n]^{-}$ exists. Then $v = (0, \ldots, 0)$. Clearly, $\gamma = \mu_{v}(\{v^{i}\})$. Therefore, $\gamma$ is not a critical cell.

    \vspace{.1cm}
	\noindent\textbf{Case (ii)}: Let $\gamma = \{v, v^{i,j}\}$ for some $v, v^{i,j} \in V(\G_m^n)$ and $i,j \in [n]$. Then $v \prec v^{i,j}$.
If for some $x \prec v$, the set $\gamma \cup \{x\}$ forms a simplex, then $d(v, x) \leq 2$. Thus, $x = v^k$ for some $k \in [n]^{-}$ or $x = v^{t,l}$ for some $l \in [n]^{-}$ with $|t| < |l|$, or $x = v^{[k; -2^{+}]}$ for some $k\in [n]$. Clearly, if $x = v^k$ or $x = v^{[t; -2^{+}]}$ for some $k \in [n]^{-}$ and $t \in [n]$, then $\{x, v^{i,j}\}$ is not a $1$-simplex in $\D_m^{n,2}$.

Now, if there exists an $s \in [n]^{-}$ such that $v^s \in \G_m^n$ and $|k| < |s|$ for some $k \in {i, j}$, then $\{v, v^{i,j}, v^{k,s}\}$ is a simplex and $v^{k,s} \prec y$ for all $y \in \gamma$. Thus, $\gamma$ is not a critical cell.

If there is no $s \in [n]^{-}$ such that $v^s \in \G_m^n$ and $|k| < |s|$ for some $k \in \{i, j\}$, then there is no $y \prec v$ such that $\{y,v^{i,j}\}$ is a simplex. In this case, we conclude that neither $\gamma$ nor $\gamma \setminus \{v\}$ belongs to $S_y \cup \mu_y(S_y)$ for any $y \prec v$. Hence, $\gamma = \mu_v(v^{i,j})$, which implies that $\gamma$ is not a critical cell.

    \vspace{.1cm}
	\noindent
	\textbf{Case (iii)}: Let $\gamma = \{v,v^{i,j}\}$ for some $v,v^{i,j} \in V(\G_m^n)$ with $i \in [n]$ and $j \in [n]^{-}$. Since $v^{i,j} \in V(\G_m^n)$, we also have $v^{j} \in V(\G_m^n)$. It is clear that $v^{j} \prec v$ and $v^{j} \prec v^{i,j}$. Since $\{v, v^{j}, v^{i,j}\}$ is a $2$-simplex in $\D_m^{n,2}$, it follows from Proposition \ref{smallestvertex} that $\gamma$ is not a critical cell.

    \vspace{.1cm}
	\noindent\textbf{Case (iv)}: Let $\gamma = \{v,v^{[i;2^{+}]}\}$ for some $v \in V(\G_m^n)$ and $i \in [n]$. First, assume that $\gamma = \{v,v^{[i;2^{+}]}\}$ for some $i \in [n]$. If for some  $k\in [n]^{-}$, where $|k|\neq i$, we have $v^k\in V(\G_m^n)$, then $\{v,v^{i,k},v^{[i;2^{+}]}\}$ is a simplex and $v^{i,k}\prec y$ for $y\in\gamma$. Thus, $\gamma$ is not a critical cell. On the other hand, if there is no $k\in [n]^{-}$, $|k|\neq i$ such that $v^k\in V(\G_m^n)$ then $v_t=0$ for $t\neq i$. Thus,  for any $x\prec v$, $\{x,v^{[i;2^{+}]}\}$ is not a simplex. Hence, we conclude that neither $\gamma$ nor $\gamma \setminus \{v\}$ belongs to $S_y \cup \mu_y(S_y)$ for any $y \prec v$. Therefore, $\gamma = \mu_v(v^{[i;2^{+}]})$, which implies that $\gamma$ is not a critical cell.

    \vspace{.1cm}
    
	Hence the matching $\mu$ yields no critical cell of dimension $1$ in $\D_m^{n,2}$.
\end{proof}

\begin{lemma}\label{lemma:critical_2}
	The matching $\mu$ yields no critical cells of dimension $2$ in $\cP(\D_m^{n,2})$.
\end{lemma}

\begin{proof}
	Let $\gamma \in \Delta_m^{n,2}$ be a $2$-simplex. Then $\gamma$ is a face of a maximal simplex of the three types given in Lemma \ref{thm:maximalsimplex_m_case}.
	
    \vspace{.1cm}
	\noindent\textbf{Case (a):} Let  $\gamma$ be a face of a maximal simplex of type $\sigma=N[v]$ for some vertex $v$. Then there exist $i_0, j_0, k_0 \in [n]^{\pm}$ such that   $\gamma=\{v,v^{i_0},v^{j_0}\}$, or $\gamma=\{v^{i_0},v^{j_0},v^{k_0}\}$. We have the following subcases:

    \vspace{.1cm}
	\noindent\textbf{Subcase (i):} $\gamma=\{v,v^{i_0},v^{j_0}\}$, or $\gamma=\{v^{i_0},v^{j_0},v^{k_0}\}$, where $i_0,j_0,k_0\in[n]$. 

    \smallskip
    
	\noindent If $\gamma=\{v^{i_0},v^{j_0},v^{k_0}\}$, then $v\prec x$ for all $x\in \gamma$  and $\gamma\cup \{v\} \subseteq N[v]$. Since $v \in V(\G_m^n)$, from Proposition \ref{smallestvertex}, $\gamma$  is not a critical cell.
    
    Let us now assume that $\gamma=\{v,v^{i_0},v^{j_0}\}$. If there exists a $l_0 \in [n]^{-}$ such that $v^{l_0} \in V(\mathcal{G}_m^n)$, then $v^{l_0}\prec x$ for all $x\in \gamma$ and $\gamma\cup \{v^{l_0}\} \subseteq N[v]$. From Proposition \ref{smallestvertex}, $\gamma$ is not a critical cell. 
	Suppose there exists no $l_0 \in [n]^{-}$ such that $v^{l_0} \in V(\mathcal{G}_m^n)$.  Then $v = (0, \ldots, 0)$, and it follows that $\gamma = \mu_{v}(\{v^{i_0}, v^{j_0}\})$. Therefore, $\gamma$ is not a critical cell.

    \smallskip
    
	\noindent \textbf{Subcase (ii):} $\gamma=\{v,v^{i_0},v^{j_0}\}$ or $\gamma=\{v^{i_0},v^{j_0},v^{k_0}\}$, where  $i_0\in [n]^{-}$, $j_0,k_0\in [n]$.

    \smallskip
    
	\noindent 
	Here, $v^{i_0} \prec v^{j_0}, v^{k_0}$.
	If there exists an $l_0 < i_0$ such that $v^{l_0} \in V(\mathcal{G}_m^n)$, then  $v^{l_0} \prec x$ for all $x\in\gamma$ and $\gamma\cup \{v^{l_0}\} \subseteq N[v]$. From Proposition \ref{smallestvertex}, $\gamma$ is not a critical cell. So, assume that there exists no $l_0 < i_0$ such that $v^{l_0} \in V(\mathcal{G}_m^n)$. This implies that $v(l)=0$ for every $l>|i_0|$.
	
If $\gamma \setminus \{v^{i_0}\} \cup \{x\}$ is a simplex and $x \prec v^{i_0}$ for some $x \in V(\mathcal{G}_m^n)$, then $d(v,x)\leq 2$, and thus one of the following holds:
	\begin{itemize}
		\item $x(|i_0|) = v(|i_0|) - 1$, and $x(s) = v(s) - 1$ for some $s \in [n]$ with $s < |i_0|$ and $x(j)=v(j)$ for $j\notin \{i_0,s\}$, 
		\item $x(|i_0|) = v(|i_0|) - 2$, and $x(j)=v(j)$ for $j\neq i_0$.
	\end{itemize}
	If $x(|i_0|) = v(|i_0|) - 1$ and $x(s) = v(s) - 1$ for some $s \in [n]$ with $s < |i_0|$ and $x(j)=v(j)$ for $j\notin \{i_0,s\}$, then $d(v^{j_0}, x) \geq 3$. Similarly, if $x(|i_0|) = v(|i_0|) - 2$ and $x(j)=v(j)$ for $j\neq i_0$, then also $d(v^{j_0}, x) \geq 3$. Hence for any $x\prec v^{i_0}$, $\gamma \setminus \{v^{i_0}\} \cup \{x\}$  is not a simplex. Thus, we conclude that both $\gamma $ and $\gamma\setminus \{v^{i_0}\}$, do not belong to     $S_y \cup \mu_y(S_y)$ for all $y \prec v^{i_0}$. Therefore, by definition, we get that  $\gamma = \mu_{v^{i_0}}(\gamma \setminus \{v^{i_0}\})$. Hence $\gamma$ is not a critical cell. 
	\smallskip

	% Then we claim that $\gamma =\mu(\gamma\setminus \{v^{i}\})$ and $v^{i}$ is the smallest such vertex. Clearly, $v^{i}<v, v^{i}<v^{j}$, and $v^{i}<v^{k}$. Let $\gamma\setminus \{v^{i}\}\cup \{x\}$ be a simplex and $x<v^{i}$. 
	
	% Since $v^{j}\in \gamma$, $d(v^j,x)\leq 2$. If $x\in N[v^j]$, then  $x=v^{l,j}$ for some $l\in[n]^{\pm}$. If $l<i$, then there is a vertex $v^{l}<v^{i}$, a contradiction. Thus, $l\geq i$. {\red We find that} \samir{Since $j > 0$, we observe that} $v^{l,j}>v^{i}$. Let $d(v^j,x)=2$. Then either $x=v^{j,l}_{\pm2}$ or $x=v^{j,p,q}$ for some $l\in[n]^{+}$ and $p,q\in [n]^{\pm}$. We get that $v^{j,l}_{+2}>v^{i}$. Thus, $x=v^{j,l}_{+2}$ is not possible. Let $x=v^{j,l}_{-2}$. If $l=|j|$, then $x=v^{-j}$ and $v^{-j}\geq v^{i}$. Thus, we consider that $l\neq |j|$. Then, $d(v,v^{j,l}_{-2})=3$ and $d(v^{k},v^{j,l}_{-2})\geq 3$. Thus, $x=v^{j,l}_{-2}$ is not possible. Let $x=v^{j,p,q}$. Then either $p=\pm j$ or $q=\pm j$ otherwise $d(v^{j},v^{j,p,q})=3$. Let $p=\pm j$. Since there is no $v^{l}<v^{i}$ for $l<i$, $\pm j \geq i$. If $p=-j$, then $x=v^q$ and then $v^q\geq v^{i}$. Let us assume that $ p,q\neq-j$. This implies that $p=j$. In this case $x=v^{q,j}_{+2}$. Then $d(v^{q,j}_{+2},v)\geq 3$. Therefore, in case of $\gamma=\{v,v^{i},v^{j}\}$, $v^{j,p,q}\geq v^{i}$. Further, if $\gamma=\{v^{i},v^{j},v^{k}$ then $q=k$ otherwise $d(v^{q,j}_{+2},v^{k})\geq 3$. Consequently, if $q=k$, then $x=v^{k,j}_{+2}>v^{i}$.
	
	% Hence, $\gamma=\{v,v^{i},v^{j}\}$ or $\gamma=\{v^{i},v^{j},v^{k}\}$ for $i\in [n]^{-}$ and $j,k\in [n]^{+}$ is not a critical cell. 

    \smallskip 
    
	\noindent\textbf{Subcase(iii):} $\gamma=\{v,v^{i_0},v^{j_0}\}$ or $\gamma=\{v^{i_0},v^{j_0},v^{k_0}\}$, where $i_0,j_0,k_0\in [n]^{-}$.
	
    \smallskip 
    
	\noindent
	Observe that $v^{i_0,j_0}\in V(\mathcal{G}_m^n), \{v,v^{i_0},v^{j_0}, v^{i_0,j_0}\} \in \Delta_m^{n,2}$,  and $v^{i_0,j_0}\prec v,v^{i_0},v^{j_0}$.  Thus, from Proposition \ref{smallestvertex}, $\{v,v^{i_0},v^{j_0}\}$ is not a critical cell. Similarly, if $\gamma=\{v^{i_0},v^{j_0},v^{k_0}\}$, then $v^{i_0,j_0,k_0}\in V(\mathcal{G}_m^n), \{v^{i_0},v^{j_0},v^{k_0}\}\cup \{v^{i_0,j_0,k_0}\} \in \Delta_m^{n,2}$, and $v^{i_0,j_0,k_0}\prec x$ for all $x\in \{v^{i_0},v^{j_0},v^{k_0}\}$.  Therefore, $\{v^{i_0},v^{j_0},v^{k_0}\}$ is not a critical cell. 
	
    \smallskip

	\noindent\textbf{Subcase (iv):} $\gamma=\{v^{i_0},v^{j_0},v^{k_0}\}$, where $i_0,j_0\in[n]^{-}$ and $k_0\in[n]$.

    \smallskip 
    
	\noindent
	Without loss of generality, we assume that $|j_0|<|i_0|$. Then   $v^{i_0} \prec v^{j_0}, v^{k_0}$.
	If there exists a vertex $v^{l_0}\in V(\mathcal{G}_m^n)$ for some $l_0<i_0$, then $v^{l_0}\prec x$ for all $x\in\gamma$ and $\gamma \cup \{v^{l_0}\} \subseteq N[v]$. Hence, 
	$\gamma$ is not a critical cell by Proposition \ref{smallestvertex}. 
	
	We now assume that there is no vertex in $V(\mathcal{G}_m^n)$ of type $v^{l_0}$ for $l_0<i_0$. This implies that for any $l>|i_0|$, $v(l)=0$. Clearly, $v^{i_0,j_0,k_0}\in V(\mathcal{G}_m^n)$. 
	If $k_0<|j_0|<|i_0|$, then  $v^{i_0,j_0,k_0}\prec x$ for all $x\in \gamma$ and $\gamma\cup\{v^{i_0,j_0,k_0}\}$ is a simplex in $\D_m^{n,2}$. Therefore, from Proposition \ref{smallestvertex}, $\gamma$ is not a critical cell. Thus, we also assume that $k_0 \geq |j_0|$.
	
	Now,  we claim that $\gamma=\mu(\gamma\setminus \{v^{i_0}\})$.
	If $\gamma \setminus \{v^{i_0}\} \cup \{x\}$ is a simplex and $x \prec v^{i_0}$ for some $x \in V(\mathcal{G}_m^n)$, then one of the following holds:
	\begin{itemize}
		\item $x(|i_0|) = v(|i_0|) - 1$, and $x(s) = v(s) - t$ for some $s \in [n]$ with $s < |i|$ and $t \geq 1$, 
		\item $x(|i_0|) = v(|i_0|) - r$ for some $r \geq 2$.
	\end{itemize}
	If $x(|i_0|) = v(|i_0|) - 1$ and $x(s) = v(s) - t$ for some $s \in [n]$ with $s < |i|$ and $t \geq 1$, then $d(v^{k_0}, x) \leq 2$ if and only if $x(k_0)=v(k_0)+1$, $k_0\neq s$, and $t=1$. If $x(k_0)=v(k_0)+1$, $k_0\neq s$, and $t=1$ then $k_0<s$ and thus $|j_0|\neq s$. Therefore, $d(v^{j_0},x)\geq 3$.
    
    Similarly, if $x(|i_0|) = v(|i_0|) - r$ for some $r \geq 2$, then also $d(v^{k_0}, x) \leq 2$ if and only if  $x(k_0)=v(k_0)+1$, $k_0\neq |i_0|$, and $r=2$. If $x(k_0)=v(k_0)+1$, $k_0\neq |i_0|$, and $r=2$ then from the fact that $|j_0|\leq k_0<|i_0|$, we find that $d(v^{j_0},x)\geq 3$.

Hence for any $x \prec v^{i_0}$, $\gamma \setminus \{v^{i_0}\} \cup \{x\}$  is not a simplex.

	Thus, both $\gamma $ and $\gamma\setminus \{v^{i_0}\}$, do not belong to  $S_x \cup \mu_x(S_x)$ for all $x \prec v^{i_0}$.
	Therefore, by definition, $\gamma = \mu_{v^{i_0}}(\gamma \setminus \{v^{i_0}\}) =\mu(\gamma \setminus \{v^{i_0}\}) $. Hence $\gamma$ is not a critical cell.

	\smallskip
	
	\noindent\textbf{Case (b):} If $\gamma$ is a face of a maximal simplex of type $\sigma=\{v,v^{i_0},v^{j_0},v^{i_0,j_0}\}$ for some vertex $v$ and $i_0,j_0,k_0\in [n]^{\pm}$, then the possible types of $2$-simplices in $\sigma$ are $\{v,v^{i_0},v^{j_0}\}$, $\{v,v^{i_0},v^{i_0,j_0}\}$, $\{v,v^{j_0},v^{i_0,j_0}\}$, and $\{v^{i_0},v^{j_0}, v^{i_0,j_0}\}$. Observe that  $\{v,v^{i_0},v^{j_0}\} \subseteq N[v]$, $\{v,v^{i_0},v^{i_0,j_0}\} \subseteq N[v^{i_0}]$, $\{v,v^{j_0},v^{i_0,j_0}\} \subseteq N[v^{j_0}]$ and $\{v^{i_0},v^{j_0},v^{i_0,j_0}\}\subseteq N[v^{i_0,j_0}]$. 
	Thus, from case $(a)$ above, $\gamma$ is not a critical cell.

	\smallskip
	\noindent\textbf{Case (c):} If $\gamma$ is a face of a maximal simplex of type $\sigma=\{v,v^{i_0,j_0},v^{i_0,k_0},v^{j_0,k_0}\}$ for some vertex $v$ and $i_0,j_0,k_o\in [n]^{\pm}$, then the possible types of $2$-simplices in $\sigma$ are $\{v,v^{i_0,j_0},v^{i_0,k_0}\}$, $\{v,v^{i_0,j_0},v^{j_0,k_0}\}$, $\{v,v^{i_0,k_0},v^{j_0,k_0}\}$ and $\{v^{i_0,j_0},v^{j_0,k_0},v^{i_0,k_0}\}$. Observe that $\{v,v^{i_0,j_0},v^{i_0,k_0}\}\subseteq N[v^{i_0}]$, $\{v,v^{i_0,j_0},v^{j_0,k_0}\} \subseteq N[v^{j_0}]$, $\{v,v^{i_0,k_0},v^{j_0,k_0}\}\subseteq N[v^{k_0}]$ and $\{v^{i_0,j_0},v^{j_0,k_0},v^{i_0,k_0}\}\subseteq N[v^{i_0,j_0,k_0}]$. 
	Thus, from case $(a)$ above, $\gamma$ is not a critical cell.
\end{proof}

\begin{lemma}\label{lemma:critical_4 or more}
	The matching $\mu$ yields no critical cells of dimension $4$ or more in $\cP(\D_m^{n,2})$.
\end{lemma}

\begin{proof}
	Since  $\mathrm{Card}(N[x]) \leq 2n+1$ for any $x \in V(\mathcal{G}_m^n)$, using  Lemma \ref{thm:maximalsimplex_m_case}, we get that for every $\gamma \in \D_m^{n,2}$, $\dim(\gamma) \leq 2n$. This implies that there is no critical cell of dimension $2n+1$ or higher.
	
	Let $\sigma$ be a simplex in $\D_m^{n,2}$ with $4 \leq \dim(\sigma) \leq 2n$. Then, from Lemma \ref{thm:maximalsimplex_m_case}, we have  $\sigma \subseteq N[v]$ for some $v$. We claim that $\sigma$ is not a critical cell.
	
	\smallskip
	\noindent\textbf{Case (i):} Let $v^{i_0}\notin \sigma$ for every $i_0\in[n]^{-}$. If $v\notin\sigma$, then for any $x\in \sigma$, we have $v\prec x$ and $\sigma\cup \{v\}\subseteq N[v]$. Moreover, since dim$(\sigma) \geq 4$, we have $v \in V(\G_m^n)$. Thus, it follows from Proposition \ref{smallestvertex} that $\sigma$ is not a critical cell. 
	Now, assume that $v\in \sigma$. If there exists a $l_0 \in [n]^{-}$ such that $v^{l_0} \in V(\mathcal{G}_m^n)$, then $v^{l_0}\prec x$ for all $x\in \sigma$ and $\sigma\cup \{v^{l_0}\} \subseteq N[v]$. From Proposition \ref{smallestvertex}, $\sigma$ is not a critical cell. 
	Suppose there exists no $k \in [n]^{-}$ such that $v^{k} \in V(\mathcal{G}_m^n)$.  Then $v = (0, \ldots, 0)$. Clearly $\sigma = \mu(\sigma\setminus \{v\})$. Therefore, $\sigma$ is not a critical cell.

	\smallskip
	
	\noindent\textbf{Case (ii):} Assume that, there exists $i_0\in[n]^{-}$ such that $v^{i_0}\in\sigma$. Let $v^{j_0} \in \sigma$ be the minimal such vertex in $\sigma$, i.e., $v^{j_0} \prec x$ for all $x \in \sigma \setminus \{v^{j_0}\}$, where $j_0 \in [n]^-$.
	
	% \samir{If $j_0 \neq -3$, then $v^{-3} \notin \sigma, v^{-3} < x $ for all $x \in \sigma$ and $\sigma \cup v^{-3} \subseteq N[v]$. Hence $\sigma$ is not a critical cell. So assume that $j_0 = -3$}
	
	Suppose there exists a vertex $v^{l_0} \in V(\mathcal{G}_m^n)$ with $l_0 < j_0$ such that $\sigma \cup \{v^{l_0}\} \subseteq N[v]$. Then $v^{l_0} \prec x$ for all $x \in \sigma$. Thus, from Proposition \ref{smallestvertex}, $\sigma$ is not a critical cell. So, we assume that there is no $l_0 < j_0$ with $v^{l_0} \in V(\mathcal{G}_m^n)$ and $\sigma \cup \{v^{l_0}\} \subseteq N[v]$. This implies that $v(k) = 0$ for all $k > |j_0|$. We claim that $\sigma = \mu_{v^{j_0}}(\sigma \setminus \{v^{j_0}\})$.
	
	Suppose $\sigma \setminus \{v^{j_0}\} \cup \{x\}$ is a simplex and $x \prec v^{j_0}$ for some $x \in V(\mathcal{G}_m^n)$. Then one of the following holds:
	\begin{itemize}
		\item $x(|j_0|) = v(|j_0|) - 1$, and $x(s) = v(s) - t$ for some $s \in [n]$ with $s < |j_0|$ and $t \geq 1$,
		\item $x(|j_0|) = v(|j_0|) - r$ for some $r \geq 2$.
	\end{itemize}
	
	Let $x(|j_0|) = v(|j_0|) - 1$ and $x(s) = v(s) - t$ for some $s \in [n]$ with $s < |j_0|$ and $t \geq 1$.
Since $\sigma \subseteq N[v]$ and $4 \leq \dim(\sigma) \leq 2n$, there exists a vertex $v^{p} \in \sigma$, where $p \in [n]^{\pm} \setminus \{j_0, s\}$. Suppose $d(v^{p}, x) \leq 2$. Then $t=1$ and $x(|p|) = v(|p|) \pm 1$, according to the sign of $p$. However, since $\mathrm{Card}(\sigma) \geq 5$, there exists a $q \in [n]^{\pm} \setminus \{j_0, p, s\}$ such that $v^{q} \in \sigma$. Then $d(v^{q}, x)$ must be at least $3$.

Similarly, if $x(|j_0|) = v(|j_0|) - r$ for some $r \geq 2$, then $d(v^{p}, x) \leq 2$ for some $v^{p} \in \sigma$ with $p \in [n]^{\pm} \setminus \{j_0\}$ implies that $x(|p|) = v(|p|) \pm 1$, according to the sign of $p$, and $r = 2$. However, since $\mathrm{Card}(\sigma) \geq 5$, there exists a $q \in [n]^{\pm} \setminus \{j_0, p\}$ such that $v^{q} \in \sigma$. Then $d(v^{q}, x)$ must be at least $3$.    Hence, for any $x \prec v^{j_0}$, $\sigma \setminus \{v^{j_0}\} \cup \{x\}$ is not a simplex.
	
	Thus, both $\sigma $ and $\sigma \setminus \{v^{i_0}\}$, do not belong to  $S_x \cup \mu_x(S_x)$ for all $x \prec v^{i_0}$. Therefore, by definition, $\sigma = \mu_{v^{j_0}}(\sigma \setminus \{v^{j_0}\}) = \mu(\sigma \setminus \{v^{j_0}\})$. Hence, $\sigma$ is not a critical cell. This completes the proof. 
\end{proof}

\begin{lemma}\label{lemma:3_critical_cells}
	Let $m\geq 3$. The matching $\mu$ yields at least $(m-2)^3$ critical cells of dimension $3$ in $\cP(\D_m^{3,2})$.
\end{lemma}

\begin{proof} 
	Let $(k_1,k_2,k_3)\in \mathcal{G}_m^3$ be  such that $k_1,k_2,k_3\geq 2$.
	Then $\sigma = \{(k_1,k_2,k_3), (k_1,$ $k_2-1,k_3), (k_1-1,k_2,k_3), (k_1-1,k_2-1,k_3)\}$ is a simplex of dimension three in $\D_m^{3,2}$. We show that $\sigma$ is a critical cell. 
	From Lemma \ref{thm:maximalsimplex_m_case} $(ii)$, $\sigma$ is a maximal simplex in $\D_m^{3,2}$, and thus $\sigma \cup \{x\} \notin \D_m^{3,2}$ for any $x \notin \sigma$. Thus, the only possibility for $\sigma$ to not be a critical cell is that  $\sigma = \mu(\sigma \setminus \{v\})$ for some vertex $v \in \sigma$.  We now have the following cases:

    \smallskip
    
	\noindent\textbf{Case (i):} $v = (k_1-1, k_2-1, k_3)$. In this case, we find that $\sigma \setminus \{(k_1-1,k_2-1,k_3)\} \cup \{(k_1,k_2,k_3-1)\}$ is a simplex, and $(k_1,k_2,k_3-1) < y$ for all $y \in \sigma$.
    
	\smallskip

	\noindent
	\textbf{Case (ii):} $v = (k_1-1,k_2,k_3)$. In this case,  $\sigma \setminus \{(k_1-1,k_2,k_3)\} \cup \{(k_1,k_2-1,k_3-1)\}$ is a simplex, and $(k_1,k_2-1,k_3-1) < y$ for all $y \in \sigma$.

    \smallskip

    \noindent
	\textbf{Case (iii):} $v = (k_1,k_2-1,k_3)$. In this case,  $\sigma \setminus \{(k_1,k_2-1,k_3)\} \cup \{(k_1-1,k_2,k_3-1)\}$ is a simplex, and $(k_1-1,k_2,k_3-1) < y$ for all $y \in \sigma$.

  \smallskip

	\noindent
	\textbf{Case (iv):} $v = (k_1,k_2,k_3)$. In this case,  $\sigma \setminus \{(k_1,k_2,k_3)\} \cup \{(k_1-1,k_2-1,k_3-1)\}$ is a simplex, and $(k_1-1,k_2-1,k_3-1) < y$ for all $y \in \sigma$.
	
	\smallskip
    
	Therefore, there is no $v \in \sigma$ such that $\sigma = \mu_v(\sigma \setminus \{v\})$. Hence, $\sigma$ is a critical cell. Since number of such $3$-tuples in $V(\mathcal{G}_m^3)$ is  $(m-2)^3$, the result follows.
\end{proof}

\begin{proposition}\label{prop:retract}
	For each $n \geq 3$, there exists a retraction  $r: \D_m^{n, 2} \to \D_m^{3, 2}$.  
\end{proposition}
\begin{proof}
	
	Define $r_{1}: V(\G_m^n) \to  V(\G_m^3)$ by $r_1((v_1, \ldots, v_n)) = (v_1, v_2, v_3)$. 
	We  extend the map $r_{1}$ to $r : \D_m^{n, 2} \to \D_m^{3, 2}$ by 
	$r(\sigma) := \{r_1(v) : v \in \sigma\}$ for all $\sigma \in \D_m^{n, 2}$. Since $d(r_1(v), r_1(w)) \leq d(v, w)$ for all $v , w \in V(\G_m^n)$, the map $r_1$ is a surjective simplicial map.  Hence $r$ is a retraction map. This completes the proof. 
\end{proof}

\begin{theorem}\label{theorem:finitecase}
	
	For $m \geq 3$, $\D_m^{n,2} \simeq \bigvee^{\nu_m} \Sp^3$, where $\nu_m \geq (m-2)^3$.  
\end{theorem}

\begin{proof}
	Using Theorem \ref{theorem:Morse_theory}, and  Lemmas \ref{lemma:critical_1}, \ref{lemma:critical_2}, and  \ref{lemma:critical_4 or more},  we obtain $\tilde{H}_i(\D_m^{n,2}; \Z) = 0$ if  $i \neq 3$. Further, using Proposition \ref{prop:retract} and Lemma \ref{lemma:3_critical_cells}, we have $\tilde{H}_3(\D_m^{n,2}, \Z) \neq  0$.  From Remark \ref{remark:wedge of spheres}, we conclude that $\D_m^{n,2} \simeq \bigvee^{\nu_m} \Sp^3$, where $\nu_m$ is the number of $3$-dimensional  critical cells corresponding to the matching $\mu$ defined above.  Then   the rank of $\tilde{H}_3(\D_m^{n,2}; \Z)$ is $\nu_m$.  Now, the result follows from  Proposition \ref{prop:retract} and Lemma \ref{lemma:3_critical_cells}. 
\end{proof}

\begin{theorem}\label{theorem:simplyconnected_inside}
	The complex $\vr{\Z^n}{r}$ is simply connected for all $r \geq 2$. 
\end{theorem}

\begin{proof}
	Let $\sigma: S^1 \to \vr{\Z^n}{r}$ be  a  closed path in $\Delta_n$. Since $\Delta_n$ is a simplicial complex, $\sigma$ is  homotopic to a closed path $c= x_1, x_2,\ldots,  x_1$, where  $\{x_i, x_{i+1}\} \in \Delta_n$ for ecch $i$. If for some $i$, $d(x_i, x_{i+1}) =  k \geq 2$, then there exist vertices  $ z_1,  \ldots, z_{k-1} \in \Z^n$ such  that $d(x_i, z_1) = 1 = d(z_1, z_2) = \ldots =  d(z_{k-1}, x_{i+1})$.   Clearly, the path $c_1 = x_1, \ldots, x_i, z_1, \ldots, z_{k-1}, x_{i+1}, \ldots,  x_1$ is homotopic to $\delta$.  Hence, by inserting a new vertices between each such pair of vertices of distance $\geq 2$, we can assume that $d(x_i, x_{i+1}) = 1 $ for all $i$. Using the compactness of $S^1$, we see that $c_1$ consists of finitely many edges of $\Z^n$. Hence $c_1$ is a closed edge path in $\D_m^{n,2}$ for some sufficiently large $m$.  Result follows from Theorem \ref{theorem:finitecase}.
	\end{proof}

\begin{theorem}\label{thm:(Z^n;2)}
	For  $n \geq 3$, $\vr{\Z^n}{2}$ is homotopy equivalent to the wedge sum of countably infinite copies of $\Sp^3$'s.
\end{theorem}

\begin{proof}
	
	Since any homology class of  $\vr{\Z^n}{2}$ lies in $\D_m^{n,2}$  for sufficiently large $m$, we conclude that $\tilde{H}_i(\vr{\Z^n}{2}, \Z) \neq 0$ if and only if $i = 3$. Further, using Theorem \ref{theorem:finitecase}, we see that $\tilde{H}_3(\vr{\Z^n}{2}, \Z)$ is free abelian and is of countably infinite rank. Suppose rank of $\tilde{H}_3(\vr{\Z^n}{2}, \Z)$ is $\nu$. Then $\tilde{H}_3(\vr{\Z^n}{2}, \Z) \cong \Z^{\nu}$.

	For an abelian group $G$ and a positive integer $k$, let  $M(G, k)$ denote a Moore  space, {\it i.e., } $\tilde{H}_k(M(G, k);\Z) \cong G$ and $\tilde{H}_i(M(G, k);\Z) =0 $ for $i \neq k$ (for more details on Moore space, see \cite{Hatcher02}).  Then  $\vr{\Z^n}{2}$ is a Moore space  $M(\Z^{\nu}, 3)$. 
	Since $\vr{\Z^n}{2}$ is simply connected (Theorem \ref{theorem:simplyconnected_inside}), and the fact that $\bigvee^{\nu} \Sp^3$ is a $M(\Z^{\nu}, 3)$, by uniqueness (upto homotopy equivalence for simply connected CW complexes) of Moore space, we conlcude that   $\vr{\Z^n}{2}\simeq \bigvee^{\nu} \Sp^3$.
	\end{proof}

% \begin{theorem}
	
	% For $0 \leq i \leq 2$, $\tilde{H}(\Delta) = 0$. 
	% \end{theorem}

% \begin{proof}
	% Let $V(\Delta) = \{v_1 < v_2 \ldots < v_m\}$ (here, I prefer this ordering as lexicographic). Let $\mu$ be an acyclic matching on the face poset of $\Delta$ as defined above. Our aim is to show that  there is no critical cell of dimension $\leq 2$, that any simplex of dimension $\leq 2$ is matched in $\mu$.
	
	% \end{proof}

% \sourav{$S'_y$ is used at some places in this section, that needs to be replaced by $S_y$}

% Once the above theorem is true, we show that $\Delta$ is simply connected and has non trivial homology in dimension $3$. All together we get that  $\Delta$ is homotopy equivalent to a wedge of spheres of dimension $3$.  

\section{Conclusion and Future Directions}\label{Conclusion}

In this article, we investigated the Vietoris-Rips complex of the Cayley graph (with respect to the standard generator) of the abelian group $\Z^n$ with the word metric. Building on earlier work, we confirmed Zaremsky’s conjecture (Conjecture \ref{conj:main}), for $n \leq 5$ and established the contractibility of $\vr{\Z^6}{r}$ for $r \geq 10$. Using discrete Morse theory, we further characterized the homotopy type of $\vr{\Z^n}{2}$ for $n \geq 3$, proving that it is homotopy equivalent to the wedge sum of countably infinite copies of  $3$-spheres.  

These results contribute to the growing understanding of Vietoris-Rips complexes beyond hyperbolic groups, highlighting the rich combinatorial and topological structure of $\Z^n$ under the word metric. Our findings not only extend the validity of Zaremsky’s conjecture but also provide new insights into the homotopy types of Vietoris-Rips complexes arising  from discrete spaces.  

Several questions remain open, particularly the Conjecture $\ref{conj:main}$, which is still open for $n\geq 7$. We believe that the Lemmas \ref{sumofthree},  \ref{sumoffour}, and \ref{sumoffour2} can be generalized for more coordinates. For a fix $m > 0$, let $\G_m^n$, $\Delta_m^{n, r}$, and  $\Gamma_{n}^{\alpha,r}$ be the same as defined in  the beginning of Section \ref{sec:vr1}. Then, we propose the following conjecture.

\begin{conj}\label{conj:future}
	Let $r\geq n\geq 2$ . Then $\Gamma_n^{\alpha, r}$ is homotopy equivalent to the induced subcomplex $\Delta$ of $\Gamma_n^{\alpha, r}$, where every $x \in V(\Delta)$ satisfies the following: (i) $|x_i|< \f $ for all $i \in [n]$ and (ii) for any set $S\subseteq [n]$ such that $\text{Card}(S)=n-2$, $\sum_{i\in S}|x_i|\leq r-2$.
\end{conj}

Assuming that the above conjecture is true, we can prove Conjecture \ref{conj:main} in the following way.

\noindent\textit{Proof of Conjecture \ref{conj:main}:} Recall that to establish that $\vr{\mathbb{Z}^n}{r}$ is contractible, it is sufficient to prove that $\Gamma_{n}^{\alpha,r}$ is contractible. 
From Theorem \ref{maintheorem}, $\Gamma_2^{\alpha, r}$ is contractible for $r \geq 2$. From Conjecture~\ref{conj:future}, $\Gamma_n^{\alpha, r}$ is homotopy equivalent to the induced subcomplex $\Delta_n$ of $\Gamma_n^{\alpha, r}$, where every $x \in V(\Delta_n)$ satisfies the following: $|x_i| < \f$ for all $i \in [n]$, and for any set $S \subseteq [n]$ with $\text{Card}(S) = n-2$, we have $\sum_{i \in S} |x_i| \leq r-2$.

    \smallskip
	\noindent \textbf{Case 1:} There exists an element in $V(\Delta_n)$ with a positive $n$-th coordinate.

Let $p=(0,0,\ldots,0,1) \in \Z^n$. Then observe that $p \in \Delta_n$.  Suppose there exists an element in $V(\Delta_n)$ with a positive $(n-1)$-th coordinate. Then $q = (0,0,\ldots,0,1,0) \in V(\Delta_n)$. Consider $w = (0,0,\ldots,0,1,1)$. We have $w \succ {\bf{0}}$ and $w \in V(\Delta_n)$. Let $x \in V(\Delta_n)$. 
\begin{itemize}
	\item If $x_n \geq 1$, then $d(x,w) = \sum_{i \in [n-2]} |x_i| + |x_{n-1}-1| + |x_n-1| \leq \sum_{i \in [n-2]} |x_i| + |x_{n-1}| + |x_n| \leq r$.
	\item If $x_n = 0$ and $x_{n-1} \neq 0$, then since $x \succ {\bf{0}}$, we have $x_{n-1} \geq 1$. Thus, $d(x,w) = \sum_{i \in [n-2]} |x_i| + |x_{n-1}-1| + 1 \leq r$.
	\item If $x_{n-1} = x_n = 0$, then $\sum_{i \in [n-2]} |x_i| \leq r-2$. Hence, $d(x,w) = \sum_{i \in [n-2]} |x_i| + 1 + 1 \leq r$.
\end{itemize}

Therefore, $d(x,w)\leq r$ for all $x\in V(\Delta_n)$. Thus, $\D_n$ is a cone with apex  $w$. This implies that $\D_n$ is contractible and hence 
$\Gamma_n^{\alpha, r}$ is contractible.

If there is no element in $V(\Delta_n)$ with a positive $(n-1)$-th coordinate, then for any $y \in V(\Delta_n)$, either $y_n\geq 1$ or $y_{n-1} = y_n = 0$. Since $\sum_{i\in [n-2]} |y_i| \leq r-2$,   we conclude that $d(y,p) \leq r$. Therefore, $\Delta_n$ is a cone with apex  $p$. Hence, $\Delta_n$ is contractible, and thus $\Gamma_n^{\alpha, r}$ is contractible.

\smallskip
\noindent
\textbf{Case 2:} For every $y \in V(\Delta_n)$, $y_n = 0$.

In this case, $\Delta_n$ is homotopy equivalent to an induced subcomplex $\Delta_{n-1}$ of $\Gamma_{n-1}^{\beta,r}$, for some $1 \leq \beta \leq \text{Card}(V(\mathcal{G}_m^{n-1}))$, where every $x \in V(\Delta_{n-1})$ satisfies the following: $|x_i| < \f$ for all $i \in [n-1]$, and for any set $S \subseteq [n-1]$ with $\text{Card}(S) = n-3$, we have $\sum_{i \in S} |x_i| \leq r-2$.

Now, we have the following subcases:

    \smallskip
	\noindent \textbf{Subcase (i):} There exists an element in $V(\Delta_{n-1})$ with a positive $(n-1)$-th coordinate.

Suppose there exists an element in $V(\Delta_{n-1})$ with a positive $(n-2)$-th coordinate. Let $u =  (0,0,\ldots,0, 1) \in V(\Delta_{n-1})$. Consider $v = (0,0,\ldots,0,1,1) \in \Z^{n-1}$. We have $v \succ {\bf{0}}$ and $v \in V(\Delta_{n-1})$.  Using a similar computation as in Case (1), we get that $\D_{n-1}$ is a cone with apex at $v$. Therefore, $\D_{n-1}$ is contractible. Since $\Gamma_n^{\alpha, r}\simeq \D_n\simeq \D_{n-1}$, $\Gamma_n^{\alpha, r}$ is contractible.

% \begin{itemize}
% 	\item If $y_{n-1} \geq 1$, then $d(y,v) = \sum_{i \in [n-3]} |y_i| + |y_{n-2}-1| + |y_{n-1}-1| \leq \sum_{i \in [n-3]} |y_i| + |y_{n-2}| + |y_{n-1}| \leq r$.
% 	\item If $y_{n-1} = 0$ and $y_{n-2} \neq 0$, then since $y \succ {\bf{0}}$, we have $y_{n-2} \geq 1$. Thus, $d(y,v) = \sum_{i \in [n-3]} |y_i| + |y_{n-2}-1| + 1 \leq r$.
% 	\item If $y_{n-2} = y_{n-1} = 0$, then $\sum_{i \in [n-3]} |y_i| \leq r-2$. Hence, $d(y,v) = \sum_{i \in [n-3]} |y_i| + 1 + 1 \leq r$.
% \end{itemize}

% Therefore, $N[y,\Delta_{n-1}] \subseteq N(v,\Delta_{n-1})$ for every $y \in V(\Delta_{n-1})$. By Proposition \ref{flag}, we remove all the vertices in $\Delta_{n-1}$ except $v$. Hence, $\Delta_{n-1}$ is contractible, and thus $\Delta_n$ is contractible, implying that $\Gamma_n^{\alpha, r}$ is contractible.  

If there is no element in $V(\Delta_{n-1})$ with a positive $(n-2)$-th coordinate, then for any $y \in V(\Delta_{n-1})$, either $y_{n-1}  \geq 1$ or $y_{n-2} = y_{n-1} = 0$. Thus, we conclude that $d(y,u) \leq r$ for all $y \in V(\Delta_{n-1})$. Therefore, $\Delta_{n-1}$ is a cone with apex at $u$. Hence, $\Delta_{n-1}$ is contractible, and thus $\Gamma_n^{\alpha, r}$ is contractible.

    \smallskip
	\noindent 
\textbf{Subcase (ii):} For every $z \in V(\Delta_{n-1})$, $z_{n-1} = 0$.

In this subcase, we proceed similarly to  Case~(2) again. We get that $\Delta_{n-1}$ is homotopy equivalent to an induced subcomplex $\Delta_{n-2}$ of $\Gamma_{n-2}^{\gamma,r}$, for some $1 \leq \gamma \leq \text{Card}(V(\mathcal{G}_m^{n-2}))$, where every $x \in V(\Delta_{n-2})$ satisfies the following: $|x_i| < \f$ for all $i \in [n-2]$, and for any set $S \subseteq [n-2]$ with $\text{Card}(S) = n-4$, we have $\sum_{i \in S} |x_i| \leq r-2$.

By repeatedly applying Case~(1) and Case~(2) a finite number of times, we obtain
$$
\Gamma_n^{\alpha, r}\simeq \Delta_n \simeq \Delta_{n-1} \simeq \cdots \simeq \Delta_2,
$$
Where, for every $x \in V(\Delta_2)$, we have $|x_i| < \f$ for all $i \in [2]$. Since $\Delta_2$ is contractible from  Case (i) of the proof of Theorem~\ref{maintheorem}, $\Delta_n$ is contractible, and thus $\Gamma_n^{\alpha, r}$ is contractible. This completes the proof.

 Theorem~\ref{thm:(Z^n;2)} is a  generalization from $\vr{\{0, 1\}^n}{2}$ to $\vr{\mathbb{Z}^n}{2}$. It establishes that while $\vr{\{0, 1\}^n}{2}$ is homotopy equivalent to a wedge sum of finitely many $\Sp^3$'s (\cite{Adamaszek2021}), the complex $\vr{\mathbb{Z}^n}{2}$ is homotopy equivalent to a wedge sum of countably infinite copies of $\Sp^3$'s.

An extension from $\vr{\{0, 1\}^n}{2}$ to $\vr{\{0, 1\}^n}{3}$ was carried out in \cite{Shukla2023, Ziqin2024}. The authors in \cite{Ziqin2024} proved that $\vr{\{0, 1\}^n}{3}$ is homotopy equivalent to a wedge sum of finite copies of $\Sp^4$'s and $\Sp^7$'s. This naturally leads to the following questions.  

\begin{ques} Fix $n\geq 4$. 
	Is $\vr{\mathbb{Z}^n}{3}$ homotopy equivalent to a wedge sum of countably infinite copies of $\Sp^4$'s and $\Sp^7$'s? 
\end{ques}
\begin{ques}
Let $2 < r < n$.  Is $\vr{\Z^n}{r}$ is homotopy equivalent to  wedge sum of spheres ?
\end{ques}

It is known that the inclusion $\{0,1\}^n \hookrightarrow \Z^n$ induces an injective homomorphism 
$\tilde{H}_i(\vr{\{0, 1\}^n}{r}; \Z) \longrightarrow \tilde{H}_i(\vr{\Z^n}{r}; \Z)$,
and $\tilde{H}_i(\vr{\{0, 1\}^n}{r}; \Z) \neq 0$ for $r<n$. 
We conjecture the following:

\begin{conj}\label{conj:homoloyof_lattices_and_cube}
$\tilde{H}_i(\vr{\Z^n}{r}; \Z) \neq 0$ if and only if 
$\tilde{H}_i(\vr{\{0, 1\}^n}{r}; \Z) \neq 0.$

\end{conj}

Since the complexes $\vr{\Z^n}{r}$ are simply connected for $r \geq 2$ (Theorem~\ref{thm:simplyconnected_intro}), and  the complexes $\vr{\{0, 1\}^n}{r}$ are contractible  for $r \geq n$ (they are just a simplex), the Conjecture \ref{conj:homoloyof_lattices_and_cube} implies Conjecture \ref{conj:main}.
\section{Acknowledgement}

The authors would like to thank Henry Adams and Anurag Singh for conversations related to Section \ref{sec:vr2}.

 {

\bibliographystyle{alpha}
% \bibliography{references}

\begin{thebibliography}{999}


\bibitem{Adamaszek2021}
 Micha\l \hspace{0.1em} Adamaszek and Henry Adams, {\em On Vietoris–Rips complexes of hypercube graphs},  J. Appl. Comput. Topol. {\bf 6} (2022), no. 2, pp. 177--192. MR 4423676

  \bibitem{HenryShuklaAnurag2025}
 Henry Adams, Samir Shukla, and Anurag Singh, {\em \v{C}ech complexes of hypercube graphs}, Homology Homotopy Appl. {\bf 27} (2025), no. 1, 83--105. MR 4870765



\bibitem{AdamsVirk2024}
Henry Adams and \v{Z}iga Virk, {\em Lower bounds on the homology of Vietoris-Rips complexes of hypercube graphs}, Bull. Malays. Math. Sci. Soc. {\bf 47} (2024), no. 3, Paper No. 72, 32, 2024. MR 4712847

\bibitem{Bauer2021}
Ulrich Bauer, {\em Ripser: efficient computation of Vietoris-Rips persistence barcodes}, J. Appl. Comput. Topol. {\bf 5} (2021), no. 3, 391--423.  MR 4298669

\bibitem{Martin2023}
Martin Bendersky and Jelena Grbic, {\em On the connectivity of the vietoris-rips complex
of a hypercube graph}, arXiv:2311.06407v1, 2023.

\bibitem{Brid1999}
Martin R. Bridson and André Haefliger, {\em Metric spaces of non-positive curvature}, Grundlehren der mathematischen Wissenschaften [Fundamental Principles of Mathematical Sciences], vol. 319, Springer-Verlag, Berlin, 1999. MR 1744486

\bibitem{Brown1985}
Kenneth S. Brown, {\em Finiteness properties of groups}, In Proceedings of the Northwestern
conference on cohomology of groups (Evanston, Ill.) {\bf 44} (1987), 45--75.

\bibitem{Carlsson2009}
Gunnar Carlsson, {\em Topology and data}, Bull. Amer. Math. Soc. (N.S.) {\bf 46} (2009), no. 2, 255--308. MR 2476414

\bibitem{Carlsson06}
Erik Carlsson, Gunnar Carlsson, and Vin de Silva, {\em An algebraic topological method
for feature identification}, Internat. J. Comput. Geom. Appl. {\bf 16} (2006), no. 4, 291--314. MR 2250511

\bibitem{CarlssonZomorodian2005}
Gunnar Carlsson, Tigran Ishkhanov, Vin de Silva, and Afra Zomorodian, {\em Persistence
barcodes for shapes}, In Proceedings of the 2004 Eurographics/ACM SIGGRAPH symposium on Geometry processing, 2004, 124--135.

\bibitem{CarlssonIshkhanovDeSilvaZomorodian2008}
Gunnar Carlsson, Tigran Ishkhanov, Vin de Silva, and Afra Zomorodian, {\em On the local
behavior of spaces of natural images}, Int. J. Comput. Vis. {\bf 76} (2008), no. 1, 1--12. MR 3715451


\bibitem{DesilvaCarlsson2004}
 Gunnar E. Carlsson and Vin De Silva, {\em Topological estimation using witness complexes}, In SPBG’04 Symposium on Point-Based Graphics 2004, The Eurographics Association, 2004, 157--166.

\bibitem{ZomorodianCarlsson2005}
 Gunnar E.  Carlsson and A.~J. Zomorodian, {\em Computing persistent homology}, Discrete Comput. Geom. {\bf 33} (2005), no.~2, 249--274; MR 2121296

\bibitem{Ziqin2024}
Ziqin Feng, {\em Homotopy types of Vietoris-Rips complexes of hypercube graphs}, J. Topol. Anal., 2025.

\bibitem{Forman1998}
Robin Forman, {\em Morse theory for cell complexes}, Adv. Math. 134 (1998), no. 1, 90--145.MR 1612391 

\bibitem{DesilvaGhrist2006}
 Robert Ghrist and Vin De Silva, {\em Coordinate-free coverage in sensor networks with controlled boundaries via homology}, Internat. J. Robotics Res. {\bf 25} (2006)(12), 1205--1222.

\bibitem{DesilvaGhrist2007}
Robert Ghrist and Vin De Silva, {\em Coverage in sensor networks via persistent homology}, Algebr. Geom. Topol. {\bf 7} (2007), 339--358. MR 2308949

\bibitem{Ghys}
 Etienne Ghys and Pierre de~la~Harpe, {\em Espaces m\'etriques hyperboliques}, in {\it Sur les groupes hyperboliques d'apr\`es Mikhael Gromov (Bern, 1988)}, 27--45, Progr. Math., 83, Birkh\"auser Boston, Boston, MA, ; MR 1086650

\bibitem{Goyal2021}
Shuchita Goyal, Samir Shukla, and Anurag Singh, {\em Homotopy type of independence complexes of certain families of graphs}, Contrib. Discrete Math. {\bf 16} (2021), no.~3, 74--92. MR 4369845



\bibitem{Gromov87}
Mikhael Gromov, {\em Hyperbolic groups}, Essays in group theory, Math. Sci. Res. Inst. Publ., vol. 8, Springer, New York, 1987, pp. 75--263.  MR 919829

\bibitem{Hatcher02}
Allen Hatcher, {\em Algebraic topology}, Cambridge University Press, Cambridge, 2002.


\bibitem{Kozlovbook}
Dmitry Kozlov, {\em Combinatorial algebraic topology}, Algorithms and Computation in Mathematics, vol. 21, Springer, Berlin, 2008. MR 2361455

\bibitem{McCarty2025}
Sarah Margarita McCarty, {\em On the structure of infinite families of sets: Functions with integral representations and topological data analysis}, Doctoral dissertation, Iowa State
University, 2025.






\bibitem{Muhammad2007}
Abubakr Muhammad and Ali Jadbabaie, {\em Dynamic coverage verification in mobile sensor networks via switched higher order laplacians}, Robotics: Science and Systems,
page 72, 2008.



\bibitem{Prisner1992}
Erich Prisner, {\em Convergence of iterated clique graphs}, Discrete Math. {\bf 103} (1992), no. 2, 199--207.




\bibitem{Shukla2023}
Samir Shukla, {\em On Vietoris-Rips complexes (with scale 3) of hypercube graphs}, SIAM J. Discrete Math. {\bf 37} (2023), no.~3, 1472--1495. MR 4613772

\bibitem{Vietoris27}
Leopold Vietoris, {\em Über den höheren Zusammenhang kompakter Räume und eine Klasse von zusammenhangstreuen Abbildungen}, Math. Ann. {\bf 97} (1927), no. 1, 454--472 (German). MR 1512371

\bibitem{Virk2022Persistance}
Žiga Virk, Contractions in persistence and metric graphs, Bull. Malays. Math. Sci. Soc. {\bf 45} (2022), no. 5, 2003--2016. MR 4489548

\bibitem{Virk2025}
Žiga Virk, {\em Contractibility of the Rips complexes of integer lattices via local domination}, Trans. Amer. Math. Soc. {\bf 378} (2025), no.~3, 1755--1770. MR 4866350



\bibitem{Wang2024}
Qingsong Wang, {\em Contractibility of vietoris-rips complexes of dense subsets in ($\mathbb{R}^n, \ell_1$) via hyperconvex embeddings}, arXiv:2406.08664, 2024.

\bibitem{Douglas} 
Douglas~B.~West, \emph{Introduction to Graph Theory}, 2nd ed., Prentice Hall, 2001.


\bibitem{Whitehead}
John H. C. Whitehead, {\em Combinatorial homotopy. I}, Bull. Amer. Math. Soc.,  {\bf 55} (1949), no.~3, 213--245.

\bibitem{Zaremsky2022}
Matthew C. B. Zaremsky, {\em Bestvina-Brady discrete Morse theory and Vietoris-Rips complexes}, Amer. J. Math. {\bf 144} (2022), no. 5, 1177--1200. MR 4494179

\bibitem{Zaremsky2025}
Matthew C. B. Zaremsky, {\em Contractible vietoris-rips complexes of $\mathbb{Z}^n$}, Proc. Amer. Math. Soc, To appear, 2025.

\bibitem{Zomorodian2010}
Afra Zomorodian, {\em Fast construction of the vietoris-rips complex}, Computers \& Graphics, {\bf 34} (2010), no.~3, 263--271.

 \end{thebibliography}

%\bibliographystyle{amsplain}
%\begin{thebibliography}{10}
%
%\bibitem {A} T. Aoki, \textit{Calcul exponentiel des op\'erateurs
%microdifferentiels d'ordre infini.} I, Ann. Inst. Fourier (Grenoble)
%\textbf{33} (1983), 227--250.
%
%\bibitem {B} R. Brown, \textit{On a conjecture of Dirichlet},
%Amer. Math. Soc., Providence, RI, 1993.
%
%\bibitem {D} R. A. DeVore, \textit{Approximation of functions},
%Proc. Sympos. Appl. Math., vol. 36,
%Amer. Math. Soc., Providence, RI, 1986, pp. 34--56.
%
%\end{thebibliography}

\end{document}